\documentclass[smallextended,final]{svjour3}
\usepackage[letterpaper,top=2in, bottom=1.5in, left=1in, right=1in]{geometry}

\usepackage{epsfig,epsf,fancybox}
\usepackage{amsmath}
\usepackage{mathrsfs}
\usepackage{amssymb}
\usepackage{amsfonts}
\usepackage{graphicx}
\usepackage{color}
\usepackage[linktocpage,pagebackref,colorlinks,linkcolor=blue,anchorcolor=blue,citecolor=blue,urlcolor=blue,hypertexnames=false]{hyperref}
\usepackage{boxedminipage}
\usepackage{stmaryrd}
\usepackage{multirow}
\usepackage{booktabs}
\usepackage{cite}
\usepackage{float}

\usepackage{amsthm,thmtools,algorithm,algpseudocode}

\usepackage{pifont}

\newcounter{algorithmicH}% New algorithmic-like hyperref counter
\let\oldalgorithmic\algorithmic
\renewcommand{\algorithmic}{%
  \stepcounter{algorithmicH}% Step counter
  \oldalgorithmic}% Do what was always done with algorithmic environment

\usepackage{xcolor}
\usepackage{tcolorbox}

\usepackage[thinlines]{easytable}
\usepackage{array}

\usepackage[normalem]{ulem} % to use \sout

\newtheorem{thm}{Theorem}[section]

\newtheorem{defn}{Definition}
\newtheorem{rmk}{Remark}
\newtheorem{lem}[thm]{Lemma}

\newtheorem{cor}[thm]{Corollary}

\newtheorem{assump}{Assumption}

\def\A{\hbox{$\mathcal A$}}
\def\B{\hbox{$\mathcal B$}}
\def\C{\hbox{$\mathcal C$}}

\DeclareMathOperator*{\argmin}{arg\,min}
\DeclareMathOperator{\EX}{\mathbb{E}}

\newcommand{\RR}{{\mathbb{R}}}
\newcommand{\EE}{{\mathbb{E}}}

\newcommand{\cN}{{\mathcal{N}}}

\newcommand{\vzero}{{\mathbf{0}}}

\numberwithin{equation}{section}
\numberwithin{thm}{section}
\numberwithin{lem}{section}
\numberwithin{cor}{section}
\numberwithin{prop}{section}
\numberwithin{rmk}{section}
\numberwithin{assump}{section}

\usepackage{mathtools}
\newcommand{\expect}{\EE\expectarg}%\operatorname{E}
\DeclarePairedDelimiterX{\expectarg}[1]{[}{]}{%
	\ifnum\currentgrouptype=16 \else\begingroup\fi
	\activatebar#1
	\ifnum\currentgrouptype=16 \else\endgroup\fi
}
\newcommand{\innermid}{\nonscript\;\delimsize\vert\nonscript\;}
\newcommand{\activatebar}{%
	\begingroup\lccode`\~=`\|
	\lowercase{\endgroup\let~}\innermid 
	\mathcode`|=\string"8000
}

\begin{document}

\title{Katyusha Acceleration for Convex Finite-Sum Compositional Optimization 
}

\titlerunning{Katyusha Acceleration in Compositional Optimization}        % if too long for running head

\author{Yibo Xu \and Yangyang Xu %etc.
}


\institute{Yibo Xu \at Department of Mathematical Sciences, Rensselaer Polytechnic Institute, Troy, NY\\
              \email{xuy24@rpi.edu}             \\
Yangyang Xu \at Department of Mathematical Sciences, Rensselaer Polytechnic Institute, Troy, NY\\
\email{xuy21@rpi.edu}           
}

\date{\today}

\maketitle

\begin{abstract}
Structured optimization problems arise in many applications. To efficiently solve these problems, it is important to leverage the structure information in the algorithmic design. This paper focuses on convex problems with a finite-sum compositional structure. Finite-sum problems appear as the sample average approximation of a stochastic optimization problem and also arise in machine learning with a huge amount of training data. One popularly used numerical approach for finite-sum problems is the stochastic gradient method (SGM). However, the additional compositional structure prohibits easy access to unbiased stochastic approximation of the gradient, so directly applying the SGM to a finite-sum compositional optimization problem (COP) is often inefficient. 

We design new algorithms for solving strongly-convex and also convex two-level finite-sum COPs. Our design incorporates the Katyusha acceleration technique and adopts the mini-batch sampling from both outer-level and inner-level finite-sum. We first analyze the algorithm for strongly-convex finite-sum COPs. Similar to a few existing works, we obtain linear convergence rate in terms of the expected objective error, and from the convergence rate result, we then establish complexity results of the algorithm to produce an $\varepsilon$-solution. Our complexity results have the same dependence on the number of component functions as existing works. However, due to the use of Katyusha acceleration, our results have better dependence on the condition number $\kappa$ and improve to $\kappa^{2.5}$ from the {best-known} $\kappa^3$. Finally, we analyze the algorithm for convex finite-sum COPs, which uses as a subroutine the algorithm for strongly-convex finite-sum COPs. Again, we obtain better complexity results than existing works in terms of the dependence on $\varepsilon$, improving to $\varepsilon^{-2.5}$ from the best-known $\varepsilon^{-3}$.


\vspace{0.3cm}

%\noindent {\bf Mathematics Subject Classification}
\end{abstract}

\section{Introduction}
\label{intro}

Utilizing structure information of a problem is crucial for designing efficient algorithms, especially when the problem involves a high-dimensional variable and/or a huge amount of data. For example, recent works (e.g., \cite{johnson2013accelerating, xiao2014proximal}) have shown that on solving a finite-sum problem, the variance-reduced stochastic gradient method, which utilizes the finite-sum structure information, can significantly outperform a deterministic gradient method and a non-variance-reduced stochastic gradient method.
 
In this paper, we focus on the finite-sum compositional optimization problem (COP):
	\begin{equation}\label{genP}
\min_{x\in \mathbb{R}^{N_2}}  H(x)%\equiv f(x)+h(x)\equiv F(G(x))+h(x) 
  \equiv \frac{1}{n_1}\sum_{i=1}^{n_1}F_{i}\left(\frac{1}{n_2}\sum_{j=1}^{n_2} G_{j}(x)\right)+h(x),
\end{equation}
where $F_{i}\colon \mathbb{R}^{N_1} \rightarrow \mathbb{R}$ %\comm{The ``M'' conflicts with the ``M'' used later. Change the ``M'' here.} 
is a differentiable function for each $i=1,\ldots,n_1$, $G_{j}\colon \mathbb{R}^{N_2}\rightarrow \mathbb{R}^{N_1}$ is a differentiable map for each $j=1,\ldots,n_2$, and $h$ is a simple (but possibly non-differentiable) function. For ease of description, we let $F\colon \mathbb{R}^{N_1} \rightarrow \mathbb{R}$ and $G\colon \mathbb{R}^{N_2}\rightarrow \mathbb{R}^{N_1}$ respectively denote the average of $\{F_{i}\}$ and $\{G_{j}\}$, i.e.,
$$F = \frac{1}{n_1}\sum_{i=1}^{n_1}F_{i},\quad G = \frac{1}{n_2}\sum_{j=1}^{n_2} G_{j}.$$
Also, we let $f\colon \mathbb{R}^{N_2}\rightarrow \mathbb{R}$ be the composition of $F$ with $G$, namely,
\begin{equation}\label{eq:def-f}
f = F\circ G.
\end{equation}  
The problem \eqref{genP} can be viewed as a sample average approximation (SAA) of a two-level stochastic COP, for which \cite{wang2017stochastic} propose and analyze a class of stochastic compositional gradient methods.  Very recently,  
\cite{yang2019multilevel,zhang2019multi} extend the results to a multiple-level stochastic COP.  
Although it is possible to extend our method and analysis to a multiple-level finite-sum COP, we will focus on the two-level case because of the applications that we are interested in.   
Our main goal is to design a gradient-based (also called \emph{first-order}) algorithm for \eqref{genP} and to analyze its complexity to produce a stochastic $\varepsilon$-solution $\bar x$, i.e., $\EE\big[H(\bar x)-H(x^{\ast})\big]\leq \varepsilon$, where $x^{\ast}$ is a minimizer of $H$. The complexity is measured by the number of gradient evaluations of each $F_i$ and Jacobian matrix evaluations of each $G_j$. The method in \cite{wang2017stochastic} can certainly be applied to \eqref{genP}. However, due to the utilization of the finite-sum structure, our method can have significantly better complexity result.  
 	
Below, we give two examples that motivate us to consider problems in the form of \eqref{genP}.

\vspace{5pt}

\noindent \textbf{Example I: Risk-Averse Learning.}  
Given a set of data $\{(\boldsymbol{a}_i,b_i)\}_{i=1}^n$ sampled from a certain distribution, the (sample) mean-variance minimization \cite{MR3235228,MR1094565} can be formulated as
	\begin{eqnarray}
	&&\displaystyle\min_{\boldsymbol{x}\in X}\,\,\frac{1}{n}\sum_{i=1}^{n}\ell(\boldsymbol{x};\boldsymbol{a}_i,b_i)+\frac{\lambda}{n} \sum_{i=1}^{n}\left[\ell(\boldsymbol{x};\boldsymbol{a}_i,b_i)-\frac{1}{n}\sum_{j=1}^{n}\ell(\boldsymbol{x};\boldsymbol{a}_j,b_j)\right]^2, \label{meanv01}
		\end{eqnarray}
		where $\ell$ is the loss function, $X$ is a closed convex set in $\RR^N$, and $\lambda>0$ is to balance the trade-off between mean and variance. 
 As shown in \cite{lian2017finite}, define $G_j:\RR^N\to\RR^{N+1}$  and $F_i:\RR^{N+1}\to\RR$ by
\[G_j(\boldsymbol{x})= \big[\boldsymbol{x} ; \ell(\boldsymbol{x};\boldsymbol{a}_j,b_j)\big] , \quad  F_i(\boldsymbol{z},y)= \lambda(\ell(\boldsymbol{z};\boldsymbol{a}_i,b_i)-y)^2+\ell(\boldsymbol{z};\boldsymbol{a}_i,b_i),\]
for $j=1,\ldots, n$ and $i=1,\ldots, n$, and let $h$ be the indicator function on $X$. Then  \eqref{meanv01} can be rewritten into the form of \eqref{genP} with $n_1=n_2=n$.

As an alternative, by expanding the square term, we write \eqref{meanv01} equivalently into 
\begin{equation}
\min_{\boldsymbol{x}\in X}\,\displaystyle\frac{1}{n}\sum_{i=1}^{n}\ell(\boldsymbol{x};\boldsymbol{a}_i,b_i)+\frac{\lambda}{n} \sum_{i=1}^{n}\left(\ell(\boldsymbol{x};\boldsymbol{a}_i,b_i)\right)^2  - \lambda\left[\frac{1}{n}\sum_{j=1}^{n}\ell(\boldsymbol{x};\boldsymbol{a}_j,b_j)\right]^2.\label{meanv02}
\end{equation}
Define  $G_j:\RR^N\to\RR^{N+1}$  and $F_i:\RR^{N+1}\to\RR$ by
\[G_j(\boldsymbol{x})= \big[\boldsymbol{x} ; \ell(\boldsymbol{x};\boldsymbol{a}_j,b_j)\big] , \quad F_i(\boldsymbol{z},y)= \lambda\ell^2(\boldsymbol{z};\boldsymbol{a}_i,b_i)+\ell(\boldsymbol{z};\boldsymbol{a}_i,b_i)-\lambda y^2,\] 
for $j=1,\ldots, n$ and $i=1,\ldots, n$, and also let $h$ be the indicator function on $X$. Then we can rewrite \eqref{meanv02} into the form of \eqref{genP} with $n_1=n_2=n$.

The case of \eqref{meanv02} where $\ell$ is a linear function in $\boldsymbol{x}$ is certainly a convex COP, sometimes even strongly convex, and it has been extensively tested computationally; see \cite{chen2020momentum,lin2018improved} for example about experiments on such a convex COP, and \cite{lian2017finite,huo2018accelerated,yu2017fast,zhang2019composite} about a strongly convex COP.

\vspace{5pt}

\noindent \textbf{Example II: finite-sum constrained problems via augmented Lagrangian.}	
	Consider a problem with a finite-sum objective and also finite-sum constraints: 
\begin{equation}\label{eq:finite-sum-cp}
 \min_{x\in X}  ~\frac{1}{m}\sum_{i=1}^{m}f_i(x), 
~~\text{s.t.} ~~\frac{1}{n}\sum_{j=1}^{n}g_{j k}(x)\leq 0, \, k=1,\ldots, K,
\end{equation}
where $X$ is a closed convex set in $\RR^N$. 
The Neyman-Pearson classification problem as in \cite{rigollet2011neyman} can be formulated as \eqref{eq:finite-sum-cp} with $K=1$, and the fairness-constrained classification problem as in \cite{zafar2015fairness} can be written in the form of \eqref{eq:finite-sum-cp} with $K=2$. 
The augmented Lagrangian method (ALM) is one popular and effective way for solving functional constrained problems. It has been shown by \cite{xu2019iter-ialm} that applying an optimal first-order method (FOM) within the ALM framework can yield an overall (near) optimal FOM for nonlinear functional constrained problems. By the classic AL function, the primal subproblem (that is the most expensive step in the ALM) takes the form: 
\begin{equation}\label{eq:alm-subprob}
\min_{x\in X}  ~\frac{1}{m}\sum_{i=1}^{m}f_i(x) + \sum_{k=1}^K \psi_\beta\left(\frac{1}{n}\sum_{j=1}^{n}g_{j k}(x),\, z_k\right),
\end{equation}
where $\beta>0$ is the augmented penalty parameter, and 
$$\psi_\beta(u,v):=\frac{1}{2\beta}\left[\max\{\beta u,\, -v\} \right]^2+\frac{v}{\beta} \max\{\beta u,\, -v\}.$$
Given $\beta>0$ and $z\in\RR^K$, define $G_j:\RR^N\to\RR^{N+K}$  and $F_i:\RR^{N+K}\to\RR$ by
$$G_j(x) = \big[x; g_{j1}(x); \ldots;g_{jK}(x)\big], \quad F_i(x, y) = f_i(x)+\sum_{k=1}^K \psi_\beta(y_k, z_k)$$
for $j=1,\ldots, n$ and $i=1,\ldots, m$, and also let $h$ be the indicator function on $X$. Then we can write \eqref{eq:alm-subprob} into the form of \eqref{genP} with $n_1=m$ and $n_2=n$.

\subsection{Related Work}
In this subsection, we review existing works on solving problems in the form of \eqref{genP} or its special cases.

\vspace{5pt}

\noindent \textbf{Approaches for convex finite-sum problems.}
~When each $G_j$ is the identity map, \eqref{genP} reduces to the traditional finite-sum problem
\begin{equation}\label{eq:finite-sum-prob}
\min_{x\in\RR^N} f(x)+h(x)\equiv\frac{1}{n}\sum_{i=1}^{n} F_i(x)+h(x).
\end{equation}
For solving \eqref{eq:finite-sum-prob}, one can apply the proximal gradient method (PG) or its accelerated version APG \cite{nesterov2013gradient}. Each iteration of PG or APG computes the {gradient of $f$ }at a point and performs a proximal mapping of $h$, and thus their per-iteration complexity is $n$, in terms of the number of gradient evaluation of component functions $\{F_i\}$. If each $F_i$ has an $L$-Lipschitz continuous gradient, and $f$ is $\mu$-strongly convex, then PG and APG respectively need $O\left(\kappa\log\frac{1}{\varepsilon}\right)$ and $O\left(\sqrt{\kappa}\log\frac{1}{\varepsilon}\right)$ iterations to produce an $\varepsilon$-solution, where $\kappa=\frac{L}{\mu}$ denotes the condition number. Hence, their total complexities are respectively $O\left(n \kappa\log\frac{1}{\varepsilon}\right)$ and $O\left(n\sqrt{\kappa}\log\frac{1}{\varepsilon}\right)$, which are high as $n$ is large. The stochastic gradient descent (SGD) can be used for the big-$n$ case of \eqref{eq:finite-sum-prob}. At each update, it only needs to evaluate the gradient of one or a few randomly sampled functions and can produce a stochastic $\varepsilon$-solution with a complexity of $O\left(\frac{G\kappa}{\mu\varepsilon}\right)$. Here, $G$ is a bound for the second moment of sample gradients. The complexity of SGD could be lower than those of PG and APG if $n$ is large and $\varepsilon$ is not too tiny. While PG and APG treat \eqref{eq:finite-sum-prob} as a regular deterministic problem, the SGD simply takes it as a stochastic program. None of them utilize the finite-sum structure. It turns out that a better complexity of  $O\left((n+\kappa)\log \frac{1}{\varepsilon}\right)$ can be obtained by utilizing the special structure, through a random sampling together with a variance reduction (VR) technique (e.g., \cite{schmidt2017minimizing, johnson2013accelerating, xiao2014proximal, defazio2014saga}). Furthermore, the Katyusha acceleration by \cite{allen2017katyusha} incorporates the VR and the linear coupling technique of \cite{allen2014linear} that is disassembled from Nesterov's acceleration. The Katyusha accelerated method achieves a complexity of $O\left((n+\sqrt{n \kappa})\log \frac{1}{\varepsilon}\right)$, which is lower than $O\left((n+\kappa)\log \frac{1}{\varepsilon}\right)$ if $n\ll \kappa$. Similar results have also been shown in  \cite{defazio2016simple, shang2018asvrg, lin2015universal}. They match with the lower complexity bound given in \cite{woodworth2016tight} and thus are optimal for solving problems in the form of \eqref{eq:finite-sum-prob}. 
{It is worth mentioning that  
for PG and APG, the condition number $\kappa=\frac{L_f}{\mu}$ where $L_f$ is the smoothness constant of $f$, while for Katyusha,  
$\kappa=\frac{\sum_{i=1}^{n}L_i}{n\mu}$ where $L_i$ is the smoothness constant of $F_i$ for each $i$. The latter $\kappa$ is always no smaller than the former one, but they can be the same in the extreme case. 
}

\vspace{5pt}

\noindent\textbf{Approaches for finite-sum COPs.} ~Several methods have been designed specifically for solving problems with finite-sum composition structure. However, their complexity results are generally worse than those obtained for solving problems in the form of \eqref{eq:finite-sum-prob}. For example, to produce a stochastic $\varepsilon$-solution, the methods in \cite{lian2017finite,huo2018accelerated} both use the VR technique and bear a complexity of $O\left((n_1+n_2+\kappa^3)\log\frac{1}{\varepsilon}\right)$ if the objective in \eqref{genP} is strongly convex and has condition number $\kappa$.  
This result is significantly worse than the complexity of $O\left((n+\kappa)\log \frac{1}{\varepsilon}\right)$ mentioned previously for solving \eqref{eq:finite-sum-prob}, in terms of the dependence on the condition number $\kappa$. The worse result is caused by the additional composition structure, which prohibits easy access to unbiased stochastic estimation of $\nabla f$. To see this, note that
	\begin{equation*}
 \nabla f(x) = [\nabla G(x) ]^\top  \nabla F(G(x))  
 =\displaystyle\left[\frac{1}{n_2}\sum_{j=1}^{n_2} \nabla G_{j}(x) \right]^\top  \nabla F\left(\frac{1}{n_2}\sum_{j=1}^{n_2} G_{j}(x)\right).
	\end{equation*}	
Hence, if we can unbiasedly estimate the Jacobian matrix $\nabla G(x)$ and the gradient $\nabla F(G(x))$ independently, then an unbiased estimation of $\nabla f(x)$ can be obtained. 
However, $\nabla F(\EX [\xi])= \EX[\nabla F(\xi)]$ does not generally hold for a random vector $\xi$, and thus though we can easily have an unbiased stochastic approximation of $\nabla G(x)$ and $G(x)$ by randomly sampling from $\{G_j\}$, this way does not guarantee an unbiased estimation of $\nabla F(G(x))$, let alone $\nabla f(x)$. 
Complexity results have been established in the literature for problem \eqref{genP} under different scenarios. For example, \cite{lian2017finite,huo2018accelerated} studied the scenario where $f$ is smooth and strongly convex. Both of them  
inherit the algorithmic design from \cite{johnson2013accelerating} 
and have achieved linear convergence. 
Besides the strongly convex scenario, other cases of \eqref{genP} have also been studied. 
The case with a smooth and convex $f$ was treated, for example, in \cite{huo2018accelerated,lin2018improved}. 
The case with a smooth but non-convex $f$ is studied, for example, in \cite{wang2017accelerating,liu2017variance,chen2020momentum,ghadimi2020single}. 
{  The case with a non-convex $f$ that has a convex but nonsmooth $F$ is studied, for example, in \cite{tran2020stochastic,zhang2020stochastic}. 
}
The work of \cite{liu2017variance} employs the variance reduction technique while sampling the inner map $G$ and its Jacobian and the outer map $F$, for which various mini-batch sizes can be taken.  
Instead of a finite-sum problem, \cite{wang2017accelerating,ghadimi2020single} study a stochastic composition problem.  
Assuming that the sampling at the inner layer is unbaised and has a bounded variance, \cite{wang2017accelerating} provide sublinear guarantees %independent of $n_1,\,n_2$ and $\kappa,$ 
for strongly convex, convex, and non-convex cases.  
The work of \cite{yu2017fast} deals with a finite-sum composition problem with an additional linear constraint. It integrates variance reduction with the alternating direction method of multipliers.  
More recently, \cite{zhang2019composite} propose a SAGA-style variance-reduced algorithm that handles non-convex COPs and optimally strongly convex COPs.
For comparison, we list complexity results of state-of-the-art methods for the strongly-convex and convex cases in Table~\ref{table:1}.

\begin{table}[h]
%\centering
\caption{A comparison of complexity results amongst several state-of-the-art algorithms for solving problems in the form of \eqref{genP} to produce a stochastic $\varepsilon$-solution; see Definition~\ref{def:eps-sol}. ``$h\neq 0$'' is to reflect whether the algorithm handles a proximal term. ``conv. outer sum.'' stands for convex outer summand, meaning whether the convexity of each $F_i \circ G$ is assumed. In that column, ``both'' indicates that the analysis is done with the assumption and also done without the assumption. In the fourth column, we use $\kappa$ for the condition number. Although the compared papers have different definitions of $\kappa$, they all have $\mu$ as the denominator in the fractions.  {C-SAGA}  of \cite{zhang2019composite} only assumes an optimally strong convexity of $f$.} 

\label{table:1}
\begin{center}
\scalebox{0.85}{
\begin{tabular}{|c|c|c|c|c|}
\hline
 Method                 & $h\neq 0$ & conv. outer sum. & $f$ strongly convex & $f$ convex \\ \hline
AGD \cite{nesterov1983method,nesterov2013introductory}              & yes       & both                 &   $(n_1+n_2)\sqrt{\kappa}\log\frac{1}{\varepsilon}$                  &     $(n_1+n_2)  \frac{1}{\sqrt{\varepsilon}} $     \\
C-SVRG-2 \cite{lian2017finite}     & no        & yes                  &      $(n_1+n_2+\kappa^3)\log\frac{1}{\varepsilon}$        &    --     \\
VRSC-PG \cite{huo2018accelerated}     & yes       & yes                  &       $(n_1+n_2+\kappa^3)\log\frac{1}{\varepsilon}$          &     $n_1+n_2+  \frac{(n_1+n_2)^{\frac{2}{3}}}{\varepsilon^2} $    \\
SCVRG  \cite{lin2018improved} & yes       & no                   &         --         &    $(n_1+n_2) \log\frac{1}{\varepsilon} + \frac{1}{\varepsilon^{3}}$     \\
{C-SAGA} \cite{zhang2019composite}     & yes        & no                  &      $(n_1+n_2+\kappa (n_1+n_2)^{\frac{2}{3}})\log\frac{1}{\varepsilon}$        &    --     \\
This paper & yes       & both                 &       $(n_1+n_2+\kappa^{2.5})\log\frac{1}{\varepsilon}$            &    $(n_1+n_2) \log\frac{1}{\varepsilon} + \frac{1}{\varepsilon^{2.5}}$ \\[0.1cm]\hline       
\end{tabular}
}
\end{center}
\end{table}

\subsection{Our Contributions}

The contributions of this paper are mainly on designing new algorithms for solving convex and strongly-convex finite-sum compositional optimization problems and establishing complexity results that appear the best so far. They are summarized as follows.
	\begin{itemize}
	
	\item First, we propose a new algorithm for solving strongly convex compositional optimization. Our design incorporates Katyusha acceleration by \cite{allen2017katyusha} together with a mini-batch sampling technique   which is also adopted by \cite{lian2017finite,huo2018accelerated} and \cite{lin2018improved}.  
	
	\item Second, we conduct the complexity analysis of the new algorithm for solving strongly convex problems in two %three
	 different scenarios. We start from the scenario where the outer finite-sum in \eqref{genP} has a relatively small number $n_1$ but the inner finite-sum takes a big number $n_2$. Then we analyze the scenario where both $n_1$ and $n_2$ are big.  
	 For both scenarios, our complexity results are roughly in the order of $(n_1+n_2+\kappa^{2.5})\log\frac{1}\varepsilon $ to produce a stochastic $\varepsilon$-solution.  
	Our complexity results are better than 
	\cite{lian2017finite,huo2018accelerated}
	by an order of $\sqrt \kappa$. This is due to the incorporation of the Katyusha acceleration in our algorithmic design. 
Furthermore, our complexity is lower than the complexity of Nesterov's accelerated method when $n_1+n_2>\kappa^{2}$, and lower than that of the method in \cite{zhang2019composite} when $n_1+n_2>\kappa^{9/4}$. {Yet, it is unknown if this complexity can be improved, as the lower complexity bound for \eqref{genP} is still an open question. }

	\item Thirdly, we propose a new algorithm for solving convex compositional optimization, by applying the optimal black-box reduction technique of \cite{allen2016optimal} and our proposed strongly-convex problem solver as a subroutine. For the two scenarios mentioned above, our complexity results are roughly in the order of $(n_1+n_2)\log\frac{1}{\varepsilon} + \frac{1}{\varepsilon^{2.5}}$. Compared to existing results, ours are better by an order of $\frac{1}{\sqrt\varepsilon}$.  
	\end{itemize}

\subsection{Notation and organization}

%\paragraph{Notation.}
Throughout the paper, we use $\|\cdot\|$ to denote the Euclidean norm of a vector and also the spectral norm of a matrix. For any real number $a$, we use $\lceil a \rceil$ for the least integer that is lower bounded by $a,$ $\lfloor a \rfloor$ for the greatest integer that is upper bounded by $a,$ and for any positive integer $n$, we use $[n]$ for the set $\{1,\ldots, n\}$. For a differentiable scalar function $f$, $\nabla f$ denotes its gradient, and for a differentiable vector function $G$, $\nabla G$ denotes its Jacobian matrix.  
$\EE$ is used for the full expectation, and a subscript will be added for conditional expectation.
We use the big-$O$, big-$\Omega$, and big-$\Theta$ notation with the standard meanings to compare two numbers that both can go to infinity. Specifically, $a = O(b)$ means that there exists a uniform constant $C>0$ such that $a\le C\cdot b$, $a = \Omega(b)$ means that there exists a uniform constant $c>0$ such that $a\ge c\cdot b$, and $a=\Theta(b)$ means that $a = O(b)$ and $a = \Omega(b)$ both hold. 

\begin{defn}[stochastic $\varepsilon$-solution]\label{def:eps-sol} ~Given $\varepsilon>0$, a random vector $\bar x$ is called a stochastic $\varepsilon$-solution of \eqref{genP} if $\EE\big[H(\bar x)-H(x^{\ast})\big]\leq \varepsilon$, where $x^{\ast}$ is a minimizer of $H$.
\end{defn}

\begin{defn}[$L$-smoothness]
~A differentiable scalar (resp. vector) function $\phi$ on a set $X$ is called $L$-smooth with $L\ge 0$ if its gradient (resp. Jacobian matrix) $\nabla \phi$ is $L$-Lipschitz continuous, namely,  
\[\|\nabla \phi(x)- \nabla \phi(x^{\prime})\|\leq L\|x-x^{\prime}\|,\,\forall x,\,x^{\prime}\in X.\]
\end{defn}

\begin{defn}[bounded gradient]
~A differentiable scalar (resp. vector) function $\phi$ on a set $X$ has a $b$-bounded gradient (resp. Jacobian matrix) $\nabla f,$ if 
\[\|\nabla f(x) \|\leq b,\,\forall x\in X.\]
\end{defn}

\begin{defn}[$\mu$-strong convexity]
~A function $\phi$ on a convex set $X$ is called $\mu$-strongly convex for some $\mu>0,$ if
\[\phi(x^{\prime})\geq \phi(x) + \langle \tilde\nabla \phi(x), x^{\prime}-x \rangle+ \frac{\mu}{2}\|x^{\prime}-x\|^2,\,\forall x,\,x^{\prime}\in X,\]
where $\tilde\nabla \phi(x)$ stands for a subgradient of $\phi$ at $x.$
\end{defn}

The rest of the paper is organized as follows. In Section~\ref{sec:algorithm}, we give the technical assumptions and also the algorithm for the strongly convex case of \eqref{genP}. A few lemmas are established in Section~\ref{sec:pre}. In Section~\ref{sec:sock}, we analyze the algorithm for the case of relatively small $n_1$ and big $n_2$, and in Section~\ref{sec:gock}, we conduct the analysis for big $n_1$ and big $n_2$. Strong convexity is assumed in Sections~\ref{sec:sock} and \ref{sec:gock}, and in Section~\ref{sec:nonsc}, we propose an algorithm for the convex case of \eqref{genP} and give the complexity results. Section~\ref{sec:conclusion} concludes the paper.

\section{Our algorithm, main complexity result, and a proof-sketch} \label{sec:algorithm} 
The following three assumptions are made throughout the analysis for strongly convex cases of \eqref{genP}.
\begin{assump}\label{singleone}
 ~The function $f$ given in \eqref{eq:def-f} is convex, and the function 
  $h$ in \eqref{genP} is $\mu$-strongly convex with $\mu>0$.
\end{assump}

\begin{assump}\label{gentwo}
~For every $i\in [n_1]$, $F_i$ is $L_F$-smooth and has a $B_F$-bounded gradient, and for every $j\in [n_2]$, $G_j$ is $L_G$-smooth and has a $B_G$-bounded Jacobian matrix.
\end{assump}
By this assumption, $f$ must be smooth, and also $F$ is $L_F$-smooth and has a $B_F$-bounded gradient. Note that we do not assume the smoothness of $h$, but the proximal mapping of $h$ needs to be easy to implement our algorithm.

\begin{assump}\label{genthree}
~For every $i\in [n_1] $ and every $ j\in [n_2]$, it holds 
$$\left\| \left[\nabla G_j(x)\right]^\top  \nabla F_i(G(x)) -\left[\nabla G_j(y)\right]^\top  \nabla F_i(G(y)) \right\|\leq L\|x-y\|,\,  \forall \,x,\,y\in \mathbb{R}^{N_2}.$$
\end{assump}
This assumption is a conventional one made in the literature. It guarantees the $L$-smoothness of $f_i:= F_i\circ G$ for each $i\in [n_1]$ by the following arguments: 
\begin{align}\label{eq:smooth-fi}
\|\nabla f_i(x)-\nabla f_i(y)\|=&\left\|[\nabla G(x)]^\top \nabla F_i(G(x))-[\nabla G(y)]^\top \nabla F_i(G(y))\right\| \cr
=&\left\|\frac{1}{n_2}\sum_{j=1}^{n_2}\Big([\nabla G_j(x)]^\top \nabla F_i(G(x))-[\nabla G_j(y)]^\top \nabla F_i(G(y))\Big)\right\| \cr
\leq& \frac{1}{n_2}\sum_{j=1}^{n_2}\left\|[\nabla G_j(x)]^\top \nabla F_i(G(x))-[\nabla G_j(y)]^\top \nabla F_i(G(y))\right\|
\leq L \|x-y\|,
\end{align}
and it implies the $L$-smoothness of $f$ as well by noting:
\[
\|\nabla f(x)-\nabla f(y)\| 
= \left\|\frac{1}{n_1}\sum_{i=1}^{n_1}\big(\nabla f_i(x)-\nabla f_i(y)\big)\right\| \\
\leq \frac{1}{n_1}\sum_{i=1}^{n_1}\|\nabla f_i(x)-\nabla f_i(y)\|
\leq L \|x-y\|.
\]
%The above assumption implies as indicated that $f$ is $L$-smooth. 
Notice that some model is provided with $f$ being both smooth and $\mu$-strongly convex while $h$ is only convex. For this case, one can let $f\leftarrow f-\frac{\mu}{2}\|\cdot\|^2$ and $h\leftarrow h+\frac{\mu}{2}\|\cdot\|^2$ to fit our assumptions. We assume $L \ge \mu.$

Assumption~\ref{gentwo} is also made in \cite{lian2017finite,huo2018accelerated,lin2018improved,zhang2019composite}. \cite{lian2017finite,huo2018accelerated} made Assumption~\ref{genthree} to obtain \eqref{eq:smooth-fi},  while \cite{lin2018improved,zhang2019composite} gave a formula of $L$ about $L_F,\,B_F,\,L_G,\,B_G;$ we choose not to use that formula because it provides an upper bound on $L$ and a smaller $L$ can exist, depending on applications. %could vary and is application dependent.

Under the above assumptions, we design the SoCK % KaSCoCO
 method to solve Problem~\eqref{genP}, and the pseudocode is given in Algorithm~\ref{alg:gock}.  The design incorporates the linear coupling technique of \cite{allen2014linear}, which dissembles Nesterov's acceleration into three parts and is also seen in \cite{allen2017katyusha}:
\begin{itemize}
		\item the linear coupling step: an $x$-trajectory that takes a convex combination of $y$, $z$-trajectories and the current snapshot $\widetilde{x}$; {when there is no $\widetilde{x}$, this step can be seen as the first weighted average update in the (three point) accelerated (proximal) gradient descent (APGD) of \cite{lan2020first}; and the weight on $\widetilde{x}$ ensures
		 that $x$-trajectory ``is not too far away from $\widetilde{x}$ so the gradient estimator remains `accurate enough','' as discussed in  \cite{allen2017katyusha}; see \eqref{SV1} or \eqref{SV3} for a relation between the bias and the $(x,\widetilde{x})$ pair;}
		\item the mirror descent step: a $z$-trajectory that performs traditional mirror descent steps on its own query points, the step size is larger and does not guarantee descent at each step; {this step can be seen as the proximal update in APGD of \cite{lan2020first}};
		\item the gradient descent step: a $y$-trajectory that walks a traditional gradient descent step from the current $x$-query point, with a step size $\Theta(\frac{1}{L})$; {this step can be seen as a close replicate of the last weighted average update in APGD of \cite{lan2020first}. %as the ``practical implementation'' in Katyusha is exactly that last update of APGD {\color{red} (What is ``the latter'' here?(updated))}
		}
	\end{itemize}
 %the step of taking convex combination of all the relevant data points from the previous iteration as in the $x_{k+1}$ update, and its application in Katyusha, 
Algorithm~\ref{alg:gock} also incorporates the variance reduction technique that is adopted by the state-of-the-art algorithms, e.g. \cite{lian2017finite,huo2018accelerated,lin2018improved}. By the linear coupling technique with carefully selected momentum weights ($\tau_1$ and $\tau_2$ here), %and also the mirror descent step size $\alpha$, 
one can achieve the optimal deterministic rates for convex objectives as in \cite{allen2014linear} and the optimal stochastic rate for a finite-sum objective as in \cite{allen2017katyusha}. For a finite-sum problem, if we view SVRG as a special SGD that has a constant step size and achieves linear convergence, then Katyusha in \cite{allen2017katyusha} essentially achieves a Nesterov's acceleration upon the gradient sampling of SVRG.

\begin{algorithm}[H]                      
\caption{ \textbf{S}trongly C\textbf{o}nvex \textbf{C}ompositional \textbf{K}atyusha (SoCK)}         
\label{alg:gock}   
{\small                        
\begin{algorithmic} 
\State \textbf{Input:} $x_0\in \mathbb{R}^{N_2},$ %$L,$ $\mu,$  
$S,$ $m\leq \frac{L}{2\mu},$ $\theta>1$, inner mini-batch sizes $A$ and $B$, and {outer mini-batch size $C$};          
		\For{$s=0$ to $S-1$}
		%\State $x_0^{s+1}=\widetilde{x}^s;$ 
		\State Compute $G(\widetilde{x}^s),$ $\nabla G(\widetilde{x}^s)$ and $\nabla f(\widetilde{x}^s)\leftarrow \left[\nabla G(\widetilde{x}^s)\right]^\top  \nabla F(G(\widetilde{x}^s));$ \algorithmiccomment{take a snapshot}
		\For{$j=0$ to $m-1$}
		\State $k\leftarrow s m+j;$
%		\State  Pick $i$ uniformly random from $\{1,\ldots,n\};$
		\State $x_{k+1}\leftarrow\tau_1 z_k+ \tau_2 \widetilde{x}^s +(1-\tau_1-\tau_2)y_k;$ \algorithmiccomment{linear coupling step}
		\State Sample $\A_k$ and $\B_k$ uniformly at random from $[n_2]$ with replacement such that $|\A_k|=A$ and $|\B_k|=B$;
		\State Let $\widehat{G}_{k}\leftarrow\frac{1}{A}\sum_{j_k \in \mathcal{A}_k}\left(G_{j_{k}}(x_{k+1})-G_{j_{k}}(\widetilde{x}^s) \right)+ G(\widetilde{x}^s);$
		\State Let $\nabla \widehat{G}_{k}\leftarrow\frac{1}{B}\sum_{j_{k} \in \mathcal{B}_k}\left(\nabla G_{j_{k}}(x_{k+1})-\nabla G_{j_{k}}(\widetilde{x}^s) \right)+ \nabla G(\widetilde{x}^s);$
		\State \textbf{Option I:} let 
		 $\widetilde{\nabla}_{k+1}\leftarrow \left[\nabla \widehat{G}_{k}\right]^\top  \nabla F(\widehat{G}_{k});$ \algorithmiccomment{batch step for outer function}
		\State \textbf{Option II:} 
		  Sample $\C_k$ uniformly at random from $[n_1]$ with $|\C_k|=C$ and let %\algorithmiccomment{mini-batch step for outer function}
		$$\textstyle\widetilde{\nabla}_{k+1}\leftarrow \frac{1}{C}\sum_{i\in\C_k}\left(\left[\nabla \widehat{G}_{k}\right]^\top  \nabla F_{i}(\widehat{G}_{k})-\left[\nabla G(\widetilde{x}^s)\right]^\top  \nabla F_{i}(G(\widetilde{x}^s))\right) +\nabla f(\widetilde{x}^s)$$
		
		\State  Let $z_{k+1}\leftarrow\argmin_{z}\, \langle \widetilde{\nabla}_{k+1},z\rangle +\frac{1}{2\alpha}\|z-z_{k}\|^2+h(z);$ \algorithmiccomment{mirror descent step}
		\State  Let $y_{k+1}\leftarrow\argmin_{y}\, \langle \widetilde{\nabla}_{k+1},y\rangle +\frac{3L}{2}\|y-x_{k+1}\|^2+h(y);$ \algorithmiccomment{gradient descent step}
		\EndFor
		\State   $\widetilde{x}^{s+1}\leftarrow(\sum_{j=0}^{m-1} \theta^{j})^{-1} \cdot \sum_{j=0}^{m-1} \theta^{j}y_{s m+j+1};$ \algorithmiccomment{update to the snapshot point}
		\EndFor
		\State \Return $\widetilde{x}^S$.
\end{algorithmic}
}
\end{algorithm}

The choices of $\tau_1,$  $\tau_2$ and $\alpha$ will be specified later in the technical sections. Algorithm~\ref{alg:gock} takes a snapshot $\widetilde{x}$ every $m$ iterations.  
The update to the snapshot query point is an artifact of the analysis, for the sake of telescoping the progress among all snapshot points. The two different settings of the final output are also artifacts from our analysis. Notice that by counting the number of component function/gradient/Jacobian evaluations, the overall complexity of the algorithm is $S\big(n_1+2n_2+m(2A+2B+n_1)\big)$ if \textbf{Option I} is taken and $S\big(n_1+2n_2+2m(A+B+C)\big)$ if \textbf{Option II} is taken. Hence, we will take \textbf{Option I} if $n_1$ is in the same magnitude of $A+B$ and  \textbf{Option II} only if $n_1\gg A+B$. 
Corresponding to the two options, we will conduct the analysis separately in Section~\ref{sec:sock} and Section~\ref{sec:gock}. 

The randomness of the algorithm comes from the uniform samples $\A_k,\,\B_k$ and $\C_k$, for all $0\leq k\leq Sm-1.$ In our analysis, we use $\EX_{k-1}$ for the conditional expectation with the history until the $k$-th iteration is fixed. More precisely, in Section~\ref{sec:sock} with \textbf{Option I} taken, $\EX_{k-1}[\,\cdot\,]=\EX\left[\,\cdot\,|\,\{\A_i,\,\B_i\}_{i=0}^{k-1}\right]$, and in Section~\ref{sec:gock} with \textbf{Option II} taken, $\EX_{k-1}[\,\cdot\,]=\EX\left[\,\cdot\,|\,\{\A_i,\,\B_i,\,\C_i\}_{i=0}^{k-1}\right]$. %in \textbf{Option II}/the context of .%, that we denote as $\EX_{k-1}[\cdot]$ for simplicity of notation.
For ease of notation, we will use the following shorthands in our analysis
\begin{equation}\label{def-Dk-Dtilde}
D_{k}\equiv \EE\big[ H(y_k)-H(x^{\ast}) \big],\ \widetilde{D}^s\equiv \EE\big[ H(\widetilde{x}^s)-H(x^{\ast}) \big],
\end{equation}
and for readers' convenience, we list some important parameters with their meanings in Table~\ref{table:2}.

\begin{table}[h]
\caption{ List of parameters and notations}
\label{table:2}
\scalebox{0.85}{
\begin{tabular}{ll}
\hline
$L$ & smoothness parameter for problem \eqref{genP}; see Assumption~\ref{genthree} \\
$\mu$ & strong-convexity parameter of $h$ in \eqref{genP} \\
$A$ & number of samples with replacement used for the output estimation of the inner map $G$ \\
$B$ & number of samples with replacement used for the Jacobian estimation of the inner map $G$ \\
$C$ & number of samples with replacement used for the gradient estimation of the outer function $F$ \\
$S$ & number of outer loops \\
$m$ & number of inner loops \\
$\widehat{G}_{k}$ & the output estimation of the inner map $G$ at iteration $k,$ implimented variance reduction and mini-batch \\
$\nabla \widehat{G}_{k}$ & the Jacobian estimation of the inner map $G$ at iteration $k,$ implimented variance reduction and mini-batch \\
$\widetilde{\nabla}_{k+1}$ & the gradient estimation of the composition $f$ at iteration $k,$ implimented variance reduction and mini-batch\\
$x^*$ & the optimal solution of \eqref{genP} \\[0.1cm]\hline
\end{tabular}
}
\end{table}

{An informal statement of our main complexity result is as follows.   %}
\begin{thm}[informal version] Given $\varepsilon>0$, under Assumptions~\ref{singleone}--\ref{genthree}, Algorithm~\ref{alg:gock} with appropriate parameter setting can produce $\widetilde{x}^S$ as a stochastic $\varepsilon$-solution within $T$ {  component function/gradient/Jacobian evaluations}, where  
	\begin{equation*}
	T = 
		\begin{cases}
			O\left( \left(n_1+n_2+{\textstyle\sqrt{\frac{L}{\mu}}}\Big({\textstyle\frac{\max\{B_G^4 L_F^2,\,B_F^2 L_G^2\}}{\mu^2}+n_1}\Big)\right) \log\frac{H(x_0)-H(x^{\ast})}{\varepsilon}\right),  &\textup{for \textbf{Option I},}\\[0.1cm]
			O\left( \Big(n_1 + n_2+\sqrt{\frac{L}{\mu}}\frac{\max\{B_G^4 L_F^2,\,B_F^2 L_G^2,\,L^2\}}{\mu^2}\Big) \log\frac{H(x_0)-H(x^{\ast})}{\varepsilon}\right),  &\textup{for \textbf{Option II}.}
		\end{cases}
	\end{equation*}
\end{thm} 

\vspace{0.2cm}

\noindent \textbf{A proof sketch:} before giving complete analysis of Algorithm~\ref{alg:gock}, we sketch a few important steps below. The analysis framework is similar to that of Katyusha \cite{allen2017katyusha}. One critical component in the convergence analysis of Algorithm~\ref{alg:gock} as well as the main difference from that of Katyusha lies in the biased gradient estimation $\widetilde{\nabla}_{k+1}$ at a query point $x_{k+1}$ given a snapshot point $\widetilde{x}^s$, and also in how the bias term $\EE\|\widetilde{\nabla}_{k+1}-\nabla f(x_{k+1})\|^2$ is bounded and dispensed within the analysis.

	\emph{First}, for every inner iteration (i.e., for each $k$), we will bound the progress of the gradient descent and mirror descent steps by using the results from \cite{allen2017katyusha};
	%\item 
	\emph{Second}, we will bound the bias term $\EE\|\widetilde{\nabla}_{k+1}-\nabla f(x_{k+1})\|^2$. Completely different from the bounding technique in \cite{allen2017katyusha}, we use the technique in \cite{lian2017finite,huo2018accelerated} and \cite{lin2018improved} for a COP and obtain a bound in the following form:
	\begin{equation}\label{eq:sock-grad-ess}
	  \EE\|\widetilde{\nabla}_{k+1}-\nabla f(x_{k+1})\|^2=O\left(\|x_{k+1}-\widetilde{x}^s\|^2\cdot\sum\frac{1}{\text{mini-batch size}}\right);
	\end{equation}
	%\item 
\emph{Thirdly},	we will utilize the linear coupling step to combine the progress within one inner iteration and generalize the key one-inner-iteration result in \cite{allen2017katyusha}, i.e., to obtain \eqref{eq:ineq-for-all-case} through Lemmas~\ref{lem:cp1} and \ref{lem:cp-step2};
		%\item 
\emph{Fourth},	we plug the bound in \eqref{eq:sock-grad-ess} about the bias term into our generalized one-inner-iteration result \eqref{eq:ineq-for-all-case}  
and obtain our key inequality on the progress of an entire inner iteration of Algorithm~\ref{alg:gock}, i.e., \eqref{eq:cp2}.  
\emph{Fifth},	we telescope \eqref{eq:cp2} and obtain the progress within an entire inner loop, i.e.,  \eqref{eq:key} for every outer iteration $s$. Then we carefully choose parameters to manage the deviation from Katyusha acceleration and obtain linear convergence in terms of the outer iteration number. \emph{Finally}, total computational complexity results are obtained from the linear convergence and the choices of mini-batch sizes. 
} %for revise

\section{Preparatory lemmas}\label{sec:pre}
In this section, we establish a few lemmas about the proposed SoCK %KaSCoCO
 method in Algorithm~\ref{alg:gock}. These results hold if either \textbf{Option I} or \textbf{Option II} is taken, and thus they can be used to show our main convergence rate results in Section~\ref{sec:sock} and Section~\ref{sec:gock}. 

The first lemma is about the progress that the algorithm makes after obtaining $y_{k+1}$. Its proof follows that of \cite[Lemma 2.3]{allen2017katyusha}.
\begin{lem}%[lemma 2.3 from \cite{allen2017katyusha}]
\label{lem:GD}
~If \[y_{k+1}=\argmin_{y}\, \langle \widetilde{\nabla}_{k+1},y-x_{k+1}\rangle +\frac{3L}{2}\|y-x_{k+1}\|^2+h(y)-h(x_{k+1}),\] and 
\[\textup{Prog}(x_{k+1})\equiv-\min_{y}\left\{ \langle \widetilde{\nabla}_{k+1},y-x_{k+1}\rangle +\frac{3L}{2}\|y-x_{k+1}\|^2+h(y)-h(x_{k+1})\right\}\geq 0,\] we have
\[ H(x_{k+1})- H(y_{k+1}) \geq \textup{Prog}(x_{k+1})-\frac{1}{4L} \|\widetilde{\nabla}_{k+1}-\nabla f(x_{k+1})\|^2 .\]
%where the expectation is with respect to (w.r.t.) $\A_k,\,\B_k.$ %and $\C.$
\end{lem}
\begin{proof} 
We have
\begin{align*}
\textup{Prog}(x_{k+1})=&-\left(\langle \widetilde{\nabla}_{k+1},y_{k+1}-x_{k+1}\rangle +\frac{3L}{2}\|y_{k+1}-x_{k+1}\|^2+h(y_{k+1})-h(x_{k+1}) \right) \\
=& -\left(\langle  \nabla f(x_{k+1}),y_{k+1}-x_{k+1}\rangle +\frac{L}{2}\|y_{k+1}-x_{k+1}\|^2+h(y_{k+1})-h(x_{k+1}) \right) \\
&\phantom{ } +\left( \langle \nabla f(x_{k+1})-\widetilde{\nabla}_{k+1},y_{k+1}-x_{k+1}\rangle -L\|y_{k+1}-x_{k+1}\|^2\right) \\
\leq& -\left( f(y_{k+1})-f(x_{k+1}) +h(y_{k+1})-h(x_{k+1})\right) + \frac{1}{4L}\|\nabla f(x_{k+1})-\widetilde{\nabla}_{k+1}\|^2.
\end{align*}
The last inequality above uses the $L$ smoothness of $f,$ as well as Young's inequality $\langle a,b\rangle\leq \frac{1}{2}\|a\|^2+\frac{1}{2}\|b\|^2.$ %Taking expectation on both sides we finish the proof.
\end{proof}

The bound on the variance of the biased sample gradient $\widetilde{\nabla}_{k+1}$ is critical for the convergence result. The next lemma will be used to derive such a bound on the variance for the algorithm with either \textbf{Option I} or \textbf{Option II}.

\begin{lem} \label{lem:summandbound}
~Let $\widehat{G}_{k}$ and $\nabla \widehat{G}_{k}$ be those in Algorithm~\ref{alg:gock}, and let $g(\cdot)$ be any function on $\mathbb{R}^{N_1}$ that is $l$-smooth and has $b$-bounded gradient, then
\[ \EX_{k-1}\left[ \left\|\left[\nabla \widehat{G}_{k}\right]^\top  \nabla g(\widehat{G}_{k}) - \left[\nabla G(x_{k+1})\right]^\top  \nabla g(G(x_{k+1}))\right\|^2 \right]\leq \left(\frac{2 B_G^4 l^2}{A} + \frac{2 b^2 L_G^2}{B}\right)\|\widetilde{x}^s-x_{k+1}\|^2.\]
\end{lem}
\begin{proof} 
First, we observe that 
\begin{align}\label{eq:lem-summandbound-ineq1}
&~  \left\|\big[\nabla \widehat{G}_{k}\big]^\top  \nabla g(\widehat{G}_{k}) - \big[\nabla G(x_{k+1})\big]^\top  \nabla g(G(x_{k+1}))\right\|^2   \cr
\stackrel{\text{\ding{172}}}{\leq}& ~2  \left\|\big[\nabla \widehat{G}_{k}\big]^\top  \nabla g(\widehat{G}_{k}) - \big[\nabla G(x_{k+1})\big]^\top  \nabla g(\widehat{G}_{k})\right\|^2   
+2  \left\|\big[\nabla G(x_{k+1})\big]^\top  \nabla g(\widehat{G}_{k}) - \big[\nabla G(x_{k+1})\big]^\top  \nabla g(G(x_{k+1}))\right\|^2   \cr
\stackrel{\text{\ding{173}}}{\leq}&~ 2b^2  \left\| \nabla \widehat{G}_{k}  -  \nabla G(x_{k+1})  \right\|^2   
+ 2B_G^2  \left\| \nabla g(\widehat{G}_{k}) - \nabla g(G(x_{k+1}))\right\|^2  \cr
\stackrel{\text{\ding{174}}}{\leq}& ~2b^2  \left\| \nabla \widehat{G}_{k}  -  \nabla G(x_{k+1})  \right\|^2  
+ 2B_G^2 l^2  \left\|  \widehat{G}_{k}  -  G(x_{k+1}) \right\|^2  .
\end{align}
Here, \ding{172} uses the Young's inequality, \ding{173} follows from the boundedness of $\nabla g$ and $\nabla G$, and \ding{174} is from the  $l$-smoothness of $g$.

Notice that $\widehat{G}_{k}$ and $\nabla \widehat{G}_{k}$ are respectively unbiased estimators of $G(x_{k+1})$ and $ \nabla G(x_{k+1})$. %using variance reduction and mini-batch techniques.
Hence,
\begin{align}\label{eq:lem-summandbound-ineq2}
&\EX_{k-1}\left[ \left\|  \widehat{G}_{k}  -  G(x_{k+1}) \right\|^2\right]= \EX_{k-1}\left[ \Big\| \frac{1}{A}\sum_{j_k \in \mathcal{A}_k}\big(G_{j_{k}}(x_{k+1})-G_{j_{k}}(\widetilde{x}^s) \big)- \big(G(x_{k+1})  - G(\widetilde{x}^s) \big) \Big\|^2 \right]\cr
\stackrel{\text{\ding{172}}}{=}& \frac{1}{A^2}\sum_{j_k \in \mathcal{A}_k}\EX_{k-1}\left[ \Big\|  \big(G_{j_{k}}(x_{k+1})-G_{j_{k}}(\widetilde{x}^s) \big)- \big(G(x_{k+1})  - G(\widetilde{x}^s) \big) \Big\|^2 \right]   \cr
\stackrel{\text{\ding{173}}}{\leq}& \frac{1}{A^2}\sum_{j_k \in \mathcal{A}_k}\EX_{k-1}\left[ \left\|  G_{j_{k}}(x_{k+1})-G_{j_{k}}(\widetilde{x}^s)   \right\|^2 \right] 
\stackrel{\text{\ding{174}}}{\leq}\frac{B_G^2}{A^2}\sum_{j_k \in \mathcal{A}_k}  \left\|  x_{k+1}-\widetilde{x}^s   \right\|^2 = \frac{B_G^2}{A }  \left\|  x_{k+1}-\widetilde{x}^s   \right\|^2.
\end{align}
Here, \ding{172} comes from the fact that $\big\{\left(G_{j_{k}}(x_{k+1})-G_{j_{k}}(\widetilde{x}^s) \right)-(G(x_{k+1})  - G(\widetilde{x}^s) )\big\}$ are conditionally independent with each other, and their expectations all equal 0, \ding{173} holds because the variance is bounded by the second moment, and \ding{174} follows from the intermediate value theorem and the boundedness  of the Jacobian of each $G_{j_k}$. %B_G$ on the norms of the Jacobians. 

A completely parallel argument gives $\EX_{k-1}\left[ \big\| \nabla \widehat{G}_{k}  -  \nabla G(x_{k+1})  \big\|^2 \right]\leq \frac{L_G^2}{B }  \left\|  x_{k+1}-\widetilde{x}^s   \right\|^2$. Plugging this inequality and that in \eqref{eq:lem-summandbound-ineq2} into \eqref{eq:lem-summandbound-ineq1} leads to the desired result. 
\end{proof}

The next lemma is from \cite[Lemma 2.5]{allen2017katyusha}.
\begin{lem}\label{lem:MD}
~Suppose $h(\cdot)$ is $\mu$-strongly convex. Given $\widetilde{\nabla}_{k+1}$, if
\[z_{k+1}=\argmin_{z}\, \alpha\langle \widetilde{\nabla}_{k+1},z-z_k\rangle +\frac{1}{2}\|z-z_{k}\|^2+\alpha h(z)-\alpha h(z_k),\]
then it holds for any $u\in \mathbb{R}^{N_2}$ that
\begin{equation}\label{eq:prog-z} 
\alpha\langle \widetilde{\nabla}_{k+1},z_{k+1}-u\rangle +\alpha h(z_{k+1})-\alpha h(u) \leq -\frac{1}{2}\|z_{k}-z_{k+1}\|^2+\frac{1}{2}\|z_{k}-u\|^2-\frac{1+\alpha\mu}{2}\|z_{k+1}-u\|^2.  
\end{equation}
\end{lem}

The following lemma serves as a critical step in combining the progress of an entire iteration, enabled by the linear coupling update. 
\begin{lem}%[coupling step 1]
\label{lem:cp1}
~Let $x_{k+1}, y_{k+1}$ and $z_{k+1}$ be those given in Algorithm~\ref{alg:gock}.
If $\tau_1\in (0, \frac{1}{3\alpha L}]$ and $\tau_2\in[0,1-\tau_1] $ in the linear coupling step,  then for any $u\in \mathbb{R}^{N_2}$ and any positive $\beta$, it holds
\begin{align}\label{eq:lem-cp1-ineq}
&\alpha \big\langle  \nabla f(x_{k+1}),z_{k}-u \big\rangle -\alpha h(u)\cr
\leq&  \frac{\alpha}{\tau_1}\big( f(x_{k+1})- H(y_{k+1}) \big) +\alpha\Big(\frac{1}{4\tau_1L}+\frac{\beta}{2}\Big) \|\widetilde{\nabla}_{k+1}-\nabla f(x_{k+1})\|^2  
+\frac{1+\alpha/\beta}{2}\|z_{k}-u\|^2  \\
&  -\frac{1+\alpha\mu}{2} \|z_{k+1}-u\|^2 
+ \frac{\alpha\tau_2}{\tau_1}  h(\widetilde{x}^s)+ \frac{\alpha(1-\tau_1-\tau_2)}{\tau_1} h(y_k).\nonumber
\end{align}
\end{lem}
\begin{proof} 
Let $v=\tau_1 z_{k+1}+\tau_2 \widetilde{x}^s+(1-\tau_1-\tau_2)y_k$. We have $x_{k+1}-v=\tau_1(z_k-z_{k+1})$, and therefore
\begin{align}
&  \alpha\langle \widetilde{\nabla}_{k+1},z_{k}-z_{k+1}\rangle -\frac{1}{2}\|z_k-z_{k+1}\|^2  
=  \frac{\alpha}{\tau_1}\langle \widetilde{\nabla}_{k+1},x_{k+1}-v\rangle -\frac{1}{2\tau_1^2}\|x_{k+1}-v\|^2  \nonumber\\
=& \frac{\alpha}{\tau_1}\left(\langle \widetilde{\nabla}_{k+1},x_{k+1}-v\rangle -\frac{1}{2\alpha\tau_1}\|x_{k+1}-v\|^2 -h(v)+h(x_{k+1}) \right) + \frac{\alpha}{\tau_1}\big( h(v)-h(x_{k+1})\big)  \nonumber\\
\stackrel{\text{\ding{172}}}{\leq}& \frac{\alpha}{\tau_1}\left(\langle \widetilde{\nabla}_{k+1},x_{k+1}-v\rangle -\frac{3L}{2}\|x_{k+1}-v\|^2 -h(v)+h(x_{k+1}) \right) + \frac{\alpha}{\tau_1}\big( h(v)-h(x_{k+1})\big)  \nonumber\\
\stackrel{\text{\ding{173}}}{\leq}& \frac{\alpha}{\tau_1}\left( H(x_{k+1})-H(y_{k+1})+\frac{1}{4L}\|\widetilde{\nabla}_{k+1}-\nabla f(x_{k+1})\|^2 \right) + \frac{\alpha}{\tau_1}\big( h(v)-h(x_{k+1})\big)  \nonumber\\
\stackrel{\text{\ding{174}}}{\leq}& \frac{\alpha}{\tau_1}\left( H(x_{k+1})-H(y_{k+1})+\frac{1}{4L}\|\widetilde{\nabla}_{k+1}-\nabla f(x_{k+1})\|^2 \right)\nonumber\\
& \phantom{\EX\Big[\frac{\alpha}{\tau_1}\left( H()+ \|\widetilde{\nabla}_{k+1}-\nabla f(x_{k+1})\|^2 \right)} + \frac{\alpha}{\tau_1}\big( \tau_1 h(z_{k+1})+\tau_2 h(\widetilde{x}^s)+(1-\tau_1-\tau_2)h(y_k)-h(x_{k+1})\big)  . \label{cp1:one}
\end{align}
Here \ding{172} holds because $\tau_1\leq \frac{1}{3\alpha L}$, \ding{173} uses Lemma~\ref{lem:GD}, and %which is the only step that the (conditional) expectation is needed. 
\ding{174} follows from the convexity of $h(\cdot)$ and the definition of $v$.

Meanwhile, for any vector $u$ and any positive $\beta$,
\begin{align}
&~\alpha \langle  \nabla f(x_{k+1}),z_{k}-u\rangle \cr
= &~  \alpha \langle  \widetilde{\nabla}_{k+1},z_{k}-u\rangle +  \alpha \langle  \nabla f(x_{k+1})-\widetilde{\nabla}_{k+1},z_{k}-u\rangle \nonumber\\
\leq &~  \alpha \langle  \widetilde{\nabla}_{k+1},z_{k}-z_{k+1}\rangle + \alpha \langle  \widetilde{\nabla}_{k+1},z_{k+1}-u\rangle   
+ \frac{\alpha}{2}\left[ \beta\|\widetilde{\nabla}_{k+1}-\nabla f(x_{k+1})\|^2+\frac{1}{\beta}\|z_{k}-u\|^2\right],\label{cp1:two}
\end{align}
where the inequality follows from the Young's inequality.

%Take the conditional expectation 
Over \eqref{cp1:two}, we substitute in \eqref{cp1:one} and Lemma~\ref{lem:MD} and combine alike terms to get: %\comm{explain how the three inequalities are applied} 
\begin{align*}
&\alpha \langle  \nabla f(x_{k+1}),z_{k}-u\rangle -\alpha h(u)\\
\leq&  \frac{\alpha}{\tau_1}\big( H(x_{k+1})- H(y_{k+1}) \big) +\alpha\Big(\frac{1}{4\tau_1L}+\frac{\beta}{2}\Big) \|\widetilde{\nabla}_{k+1}-\nabla f(x_{k+1})\|^2 
+\frac{1+\alpha/\beta}{2}\|z_{k}-u\|^2  \\
& \phantom{ H(x_{k+1})-\EX[H(y_{k+1}) } -\frac{1+\alpha\mu}{2} \|z_{k+1}-u\|^2 
+ \frac{\alpha\tau_2}{\tau_1}  h(\widetilde{x}^s)+ \frac{\alpha(1-\tau_1-\tau_2)}{\tau_1} h(y_k)-\frac{\alpha}{\tau_1}h(x_{k+1}) 
\end{align*}
which finishes the proof upon simplification. 
\end{proof}

The previous lemma generalizes the analyses of Katyusha algorithms in \cite{allen2017katyusha} to take the bias of $\widetilde{\nabla}_{k+1}$ into consideration. We follow a process similar to the analyses in \cite{allen2017katyusha}, and derive the following lemmas to manage this bias within an entire iteration.

\begin{lem}\label{lem:split}
~Let $x_{k+1}$ be given in the linear coupling step and $x^*$ be the solution of \eqref{genP}. Then
\begin{equation}\label{eq:gaps}
\|x_{k+1}-\widetilde{x}^s\|^2 \le 3\left(\tau_1^2\|z_k-x^{\ast}\|^2+\frac{2(1-\tau_2)^2}{\mu}\big(H(\widetilde{x}^s) -H(x^{\ast})\big)+\frac{2(1-\tau_1-\tau_2)^2}{\mu}\big(H( y_{k})-H(x^{\ast})\big)\right).
\end{equation}
\end{lem}

\begin{proof} 
By the update formula of $x_{k+1}$ and also the Young's inequality, we have
\begin{align*}
 \|x_{k+1}-\widetilde{x}^s\|^2=&~\|\tau_1(z_k-x^{\ast})+(1-\tau_2)(x^{\ast}-\widetilde{x}^s)+(1-\tau_1-\tau_2)( y_{k}-x^{\ast})\|^2 \nonumber \\ 
{\leq} &~ 3\left(\tau_1^2\|z_k-x^{\ast}\|^2+(1-\tau_2)^2\|x^{\ast}-\widetilde{x}^s\|^2+(1-\tau_1-\tau_2)^2\| y_{k}-x^{\ast}\|^2\right)
%&\stackrel{\text{\ding{173}}}{\leq} 3\left(\tau_1^2\|z_k-x^{\ast}\|^2+\frac{2(1-\tau_2)^2}{\mu}(H(\widetilde{x}^s) -H(x^{\ast}))+\frac{2(1-\tau_1-\tau_2)^2}{\mu}(H( y_{k})-H(x^{\ast}))\right)
\end{align*}
%Here \ding{172} uses triangle inequality and then Cauchy-Schwarz inequaity, \ding{173} uses the fact 
Since $H$ is $\mu$-strongly convex, it holds that $H(u)-H(x^{\ast})\geq\frac{\mu}{2}\|u-x^{\ast}\|^2$ for any $u$. Hence, we obtain \eqref{eq:gaps} by bounding $\|x^{\ast}-\widetilde{x}^s\|^2$ and $\| y_{k}-x^{\ast}\|^2$ by the function values of $H$. 
\end{proof}

\begin{lem}\label{lem:cp-step2}
~Suppose $\tau_1\in (0, \frac{1}{3\alpha L}]$ and $\tau_2\in[0,1-\tau_1] $ in the linear coupling step of Algorithm \ref{alg:gock}. Let $x^*$ be the solution of \eqref{genP}. Then for any positive number $\beta$, it holds
\begin{align}\label{eq:ineq-for-all-case}
0 \leq&~ \frac{\tau_2}{\tau_1}\big( H(\widetilde{x}^s)-H(x^{\ast})\big) +\frac{(1-\tau_1-\tau_2)}{\tau_1}\big(H(y_k)-H(x^{\ast})\big)- \frac{1}{\tau_1}\big(H(y_{k+1}) -H(x^{\ast})  \big)  \\
&~  + \left(\frac{1}{4\tau_1L}+\frac{\beta}{2}\right)\|\widetilde{\nabla}_{k+1}-\nabla f(x_{k+1})\|^2 
+\frac{1+\alpha/\beta}{2\alpha}\|z_{k}-x^{\ast}\|^2 -\frac{1+\alpha\mu}{2\alpha}\|z_{k+1}-x^{\ast}\|^2.\nonumber
\end{align}
\end{lem}

\begin{proof} 
For any $u$, we have
\begin{align}
&\alpha \big(f(x_{k+1}) -f(u) \big)
\stackrel{\text{\ding{172}}}{\leq} \alpha \big\langle  \nabla f(x_{k+1}),x_{k+1}-u\big\rangle =\alpha \langle  \nabla f(x_{k+1}),x_{k+1}-z_{k}\rangle+\alpha\langle  \nabla f(x_{k+1}),z_{k}-u\rangle \nonumber\\
\stackrel{\text{\ding{173}}}{=}& \frac{\alpha\tau_2}{\tau_1}\langle  \nabla f(x_{k+1}),\widetilde{x}^s-x_{k+1}\rangle +\frac{\alpha(1-\tau_1-\tau_2)}{\tau_1}\langle  \nabla f(x_{k+1}),y_k-x_{k+1}\rangle  +\alpha\langle  \nabla f(x_{k+1}),z_{k}-u\rangle \label{cp2step1-2}\\
\stackrel{\text{\ding{174}}}{\leq}& \frac{\alpha\tau_2}{\tau_1}\big( f(\widetilde{x}^s)-f(x_{k+1})\big) +\frac{\alpha(1-\tau_1-\tau_2)}{\tau_1}\big(f(y_k)-f(x_{k+1})\big) +\alpha\langle  \nabla f(x_{k+1}),z_{k}-u\rangle .\label{cp2step1}
\end{align}
Here, \ding{172} uses the convexity of $f(\cdot),$ \ding{173} is by the definition of $x_{k+1},$ \ding{174} uses the convexity of $f(\cdot)$ twice. Adding \eqref{eq:lem-cp1-ineq} and \eqref{cp2step1}, we have
\begin{align*}
&~\alpha \big(f(x_{k+1}) -H(u) \big)\\
\leq&~ \frac{\alpha\tau_2}{\tau_1}\big( H(\widetilde{x}^s)-f(x_{k+1})\big) +\frac{\alpha(1-\tau_1-\tau_2)}{\tau_1}\big(H(y_k)-f(x_{k+1})\big)+ \frac{\alpha}{\tau_1}\big( f(x_{k+1})-H(y_{k+1})\big)  \\
&  +\alpha\left(\frac{1}{4\tau_1L}+\frac{\beta}{2}\right)\|\widetilde{\nabla}_{k+1}-\nabla f(x_{k+1})\|^2 
+\frac{1+\alpha/\beta}{2}\|z_{k}-u\|^2 -\frac{1+\alpha\mu}{2}\|z_{k+1}-u\|^2.
\end{align*}
Setting $u=x^{\ast}$ in the above inequality and dividing both sides by $\alpha$, we obtain the desired result by rearranging terms.
\end{proof}
{By bounding $\|\widetilde{\nabla}_{k+1}-\nabla f(x_{k+1})\|^2$ in the form of \eqref{eq:sock-grad-ess}, and applying Lemma~\ref{lem:split} to dispense the bias deviation in \eqref{eq:ineq-for-all-case} with $\beta=\frac{6}{\mu}$, we will eventually obtain the following inequality}
	\begin{align}\label{eq:cp2}
0 \leq&~ \left(\frac{(1-\tau_1-\tau_2)}{\tau_1}+\frac{M}{\mu}2(1-\tau_1-\tau_2)^2\right)(H(y_k)-H(x^{\ast}))- \frac{1}{\tau_1}\big(\EX_{k-1}[H(y_{k+1})] -H(x^{\ast})  \big)  \\
 &~ \hspace{-0.6cm}+ \! \left(\frac{\tau_2}{\tau_1} \! + \! \frac{M}{\mu}2(1 -\tau_2)^2\right)\!( H(\widetilde{x}^s)-H(x^{\ast}))
 \! + \! \left(\frac{1}{2\alpha} \! + \! \frac{\mu}{12} \! + \! M\tau_1^2\right)\!\|z_{k}- x^{\ast}\|^2  \! - \! \frac{1+\alpha\mu}{2\alpha}\EX_{k-1}[\|z_{k+1}- x^{\ast}\|^2],\nonumber
\end{align}
where $M$ is a positive number to be defined. The proof of \eqref{eq:cp2} differs slightly for \textbf{Option I} and \textbf{Option II} of Algorithm \ref{alg:gock} by choosing appropriate $M$ for different scenarios. 
We state \eqref{eq:cp2} here as the generalization of the key step of Katyusha in \cite{allen2017katyusha} for an entire iteration of Algorithm~\ref{alg:gock}, which takes the bias of $\widetilde{\nabla}_{k+1}$ into consideration. 
The next lemma builds upon \eqref{eq:cp2} with an appropriate choice of $M$ and provides the progress of an entire inner loop, which would consequently split into distinct cases of linear convergence as in \cite{allen2017katyusha}.

%The following lemma is a key to establishing our main results.
\begin{lem}\label{telescoping}
~Suppose that \eqref{eq:cp2} holds for each $k$ and $s$ and that $2 M\tau_1^2 < \frac{5\mu}{6}$. Let $\overline{A}, \overline{B}$ and $\overline{C}$ be numbers that satisfy
$$\overline{A}\geq \frac{2M}{\mu}(1-\tau_1-\tau_2)^2,\ \overline{B}\geq \frac{2M}{\mu}(1 -\tau_2)^2,\ \overline{C} = \frac{1}{2\alpha}+\frac{\mu}{12}+M\tau_1^2.$$ 
If $\theta\in (1,\frac{1+\alpha\mu}{1+\alpha(\mu/6+2 M\tau_1^2)}]$, then 
\begin{multline}\label{eq:key}
 \left(\frac{(\tau_1+\tau_2-1+1/\theta)}{\tau_1}-\overline{A}\right)\theta\widetilde{D}^{s+1} \sum_{j=0}^{m-1}\theta^j  \leq  \left(\frac{(1-\tau_1-\tau_2)}{\tau_1}+\overline{A}\right)\left( D_{s m}-  \theta^{m} D_{(s+1) m } \right)\\
  + \left(\frac{\tau_2}{\tau_1}+\overline{B}\right) \widetilde{D}^s\sum_{j=0}^{m-1}\theta^j 
+\overline{C} \EE\big[\|z_{s m}-x^{\ast}\|^2\big] -\overline{C}\theta^m\EX\big[\|z_{(s+1)m}-x^{\ast}\|^2\big],
\end{multline}
where $\widetilde D^s$ and $D_{sm}$ are defined in \eqref{def-Dk-Dtilde}. 
\end{lem}
\begin{proof} 
Taking full expectation over both sides of \eqref{eq:cp2} and using the definition of $\widetilde D^s$ and $D_{sm}$ in \eqref{def-Dk-Dtilde}, we have from the conditions on $\overline{A}, \overline{B}$ and $\overline{C}$ that %implies the follows, provided that it holds for each $k:$
\[ 0 \leq  \left(\frac{(1-\tau_1-\tau_2)}{\tau_1}+\overline{A}\right)D_{k}- \frac{1}{\tau_1}  D_{k+1} 
  + \left(\frac{\tau_2}{\tau_1}+\overline{B}\right) \widetilde{D}^s
+\overline{C} \EE\big[\|z_{k}-x^{\ast}\|^2 \big] -\overline{C}\theta\EX\big[\|z_{k+1}-x^{\ast}\|^2\big].
 \]
%where the deterministic terms have an underlining expectation $\EX_{k-2}$ omitted. 
Multiplying the above inequality by $\theta^j$ for each $k=s m +j$ and summing up the resulting inequalities for all $j=0,\ldots,m-1$, we obtain
\begin{multline*}
 0 \leq  \left(\frac{(1-\tau_1-\tau_2)}{\tau_1}+\overline{A}\right)\sum_{j=0}^{m-1}\theta^j D_{s m+j}- \frac{1}{\tau_1} \sum_{j=0}^{m-1} \theta^{j} D_{s m+j+1} 
  + \left(\frac{\tau_2}{\tau_1}+\overline{B}\right) \widetilde{D}^s\sum_{j=0}^{m-1}\theta^j \\
+\overline{C} \EE\big[\|z_{s m}-x^{\ast}\|^2\big] -\overline{C}\theta^m\EX\big[\|z_{(s+1)m}-x^{\ast}\|^2\big].
\end{multline*}
which can be rewritten as
\begin{multline*}
 \left(\frac{(\tau_1+\tau_2-1+1/\theta)}{\tau_1}-\overline{A}\right)\sum_{j=1}^{m}\theta^j D_{s m+j}   \leq  \left(\frac{(1-\tau_1-\tau_2)}{\tau_1}+\overline{A}\right)\left( D_{s m}-  \theta^{m} D_{(s+1) m } \right)\\
  + \left(\frac{\tau_2}{\tau_1}+\overline{B}\right) \widetilde{D}^s\sum_{j=0}^{m-1}\theta^j 
+\overline{C} \EE\big[\|z_{s m}-x^{\ast}\|^2\big] -\overline{C}\theta^m\EX\big[\|z_{(s+1)m}-x^{\ast}\|^2\big].
\end{multline*}
%following the notation convention of this lemma. 
By the convexity of $H(\cdot)$ and the choice of $\widetilde{x}^{s+1}$, %=(\sum_{j=0}^{m-1} \theta^{j})^{-1} \cdot \sum_{j=0}^{m-1} \theta^{j}y_{s m+j+1}$, 
we have $\widetilde{D}^{s+1}\leq(\sum_{j=0}^{m-1} \theta^{j})^{-1} \cdot \sum_{j=0}^{m-1} \theta^{j}D_{s m+j+1}$. Therefore, the above inequality implies \eqref{eq:key}, and we complete the proof. 
\end{proof}

In Sections~\ref{sec:sock} and \ref{sec:gock} we will consider different scenarios of a strongly convex $H$, bound $\|\widetilde{\nabla}_{k+1}-\nabla f(x_{k+1})\|^2$ {in the form of \eqref{eq:sock-grad-ess}} respectively, and utilize the key convergence steps \eqref{eq:ineq-for-all-case}, \eqref{eq:cp2} and \eqref{eq:key}, to show our final theoretical guarantees with carefully chosen parameters.% and $M$.

\section{Convergence results of the SoCK %KaSCoCO
 method with outer batch step}\label{sec:sock}
In this section, we analyze the convergence rate of Algorithm~\ref{alg:gock} that takes \textbf{Option I} and estimate its complexity to produce a stochastic $\varepsilon$-solution of \eqref{genP}. More precisely, we assume $n_1$ is not too big, so we view the outer finite-sum as a single function $F$. %and consider

First, we show that \eqref{eq:cp2} holds with an appropriate choice of $M$ and thus \eqref{eq:key} follows. Then we establish the convergence rate by using \eqref{eq:key}. The next lemma bounds the variance of $\widetilde{\nabla}_{k+1}$, and it follows from Lemma~\ref{lem:summandbound} with $g=F$.
\begin{lem} \label{lem:samplevariance}
~Let $f$ be that in \eqref{eq:def-f}, and let $\widetilde{x}^s$ and $x_{k+1}$ be those given in Algorithm \ref{alg:gock} with $\widetilde{\nabla}_{k+1}$ computed by \textbf{Option I}. Then
\begin{equation}\label{SV1}
\EX_{k-1}\big[\|\widetilde{\nabla}_{k+1}-\nabla f(x_{k+1})\|^2\big]\leq \left(\frac{2 B_G^4 L_F^2}{A} + \frac{2 B_F^2 L_G^2}{B}\right)\|\widetilde{x}^s-x_{k+1}\|^2.
\end{equation}
\end{lem}

Plugging \eqref{SV1} into \eqref{eq:ineq-for-all-case}, we are able to show \eqref{eq:cp2} and thus \eqref{eq:key} with an appropriate $M$. 
\begin{lem}%[coupling step 2]
\label{lem:cp2}
~Suppose $\tau_1\in (0, \frac{1}{3\alpha L}]$ and $\tau_2\in[0,1-\tau_1] $ in the linear coupling step of Algorithm \ref{alg:gock}. Let $x^*$ be the solution of \eqref{genP}. If $\widetilde{\nabla}_{k+1}$ is computed by \textbf{Option I}, then \eqref{eq:cp2} holds with  
\begin{equation}\label{eq:M-set1}
M = 3(\frac{1}{4\tau_1L}+\frac{3}{\mu})\left(\frac{2 B_G^4 L_F^2}{A} + \frac{2 B_F^2 L_G^2}{B}\right).
\end{equation} 
\end{lem}
\begin{proof} 
Taking conditional expectation $\EE_{k-1}$ on both sides of \eqref{eq:ineq-for-all-case} with $\beta=\frac{6}{\mu}$ and plugging \eqref{SV1}, we immediately have \eqref{eq:cp2} by using the choice of $M$ in \eqref{eq:M-set1}.
\end{proof}

The next result is easy to show. Its proof is given in the appendix.

\begin{lem}\label{usualrate}
~Let $\overline{A},\overline{B}\in(0,\frac{1}{8}],$ $\tau_1\in[\frac{1}{2m},1),$ $\tau_2\leq\tau_1.$ If $\theta= 1+\frac{1}{12m},$ then $\frac{(\tau_1+\tau_2-1+1/\theta-\overline{A}\tau_1)\theta}{\tau_2+\overline{B}\tau_1} > \frac{13}{12}.$
\end{lem}

Now we are ready to show our first main convergence rate result.
%The result is summarized in the following theorem.

\begin{thm}[convergence rate with \textbf{Option I}]\label{thm:sock}

Under Assumptions~\ref{singleone}, \ref{gentwo} and \ref{genthree}, let $\{\widetilde x^s\}$ be generated from Algorithm~\ref{alg:gock} with $\widetilde\nabla_{k+1}$ computed by \textbf{Option I} and with parameters set as follows:
\begin{equation}\label{eq:para-1}
m\leftarrow\left\lceil \frac{1}{2}\sqrt{\frac{L}{\mu}}\right\rceil,\, \tau_1\leftarrow  \frac{1}{2m},\,\tau_2\leftarrow \frac{1}{2m},\, \theta\leftarrow 1+\frac{1}{12m},\, \alpha\leftarrow \frac{1}{3\tau_1 L},\, A\leftarrow\frac{720 B_G^4 L_F^2}{\mu^2},\, B\leftarrow\frac{720 B_F^2 L_G^2}{\mu^2}.
\end{equation}
Then 
\begin{equation*}
\EX\big[H(\widetilde{x}^{S})-H(x^{\ast}) \big] \leq %\begin{cases}
%O\left( (1+\sqrt{\frac{\mu}{18m L_f}})^{-Sm}\right)(H(x_0)-H(x^{\ast})) &\text{if $m\leq \frac{L_f}{2\mu};$}\\
11\cdot\left(  \frac{12}{13}\right)^{S}\big(H(x_0)-H(x^{\ast})\big). %&\text{if  $m> \frac{L_f}{2\mu}.$}
%\end{cases}
\end{equation*}
\end{thm}

\begin{proof} %[Proof of Theorem~\ref{thm:sock}]
First, notice $m^2\geq\frac{L}{4\mu}.$ %Since $\tau_1=\tau_2=\frac{1}{2m},$ 
Hence, $\alpha=\frac{1}{3\tau_1L}=\frac{2m}{3L}\geq\frac{1}{6m\mu},$ and thus $ \alpha\mu\geq\frac{1}{6m}$. Secondly, $m<\frac{1}{2}\sqrt{\frac{L}{\mu}}+1\le\frac{3L}{2\mu},$ so $\alpha\mu<1$. Thirdly, by the choices of $A$ and $B$, the $M$ given in \eqref{eq:M-set1} satisfies 
\begin{equation}\label{eq:M-set1-ineq}
M=\frac{\mu^2}{60}\left(\frac{1}{4\tau_1 L}+\frac{3}{\mu}\right)\le \frac{\mu^2}{60}\left(\frac{3}{4\mu}+\frac{3}{\mu}\right) = \frac{\mu}{16},
\end{equation}
where the first inequality holds because $\frac{1}{\tau_1}= 2m<\frac{3L}{\mu}$. Therefore, $2M\tau_1^2 = \frac{M}{2m^2}\le \frac{M}{2} < \frac{\mu}{6}$, and thus
\begin{equation}\label{eq:bound-theta}
\frac{1+\alpha\mu}{1+\alpha(\mu/6+2 M\tau_1^2)} > \frac{1+\alpha\mu}{1+\alpha\mu/3} \ge 1+\frac{\alpha\mu}{2} \ge 1+\frac{1}{12m} =\theta,
\end{equation}
where we have used $\alpha\mu <1$ in the second inequality and $ \alpha\mu\geq\frac{1}{6m}$ in the third inequality.

Let $\overline{A} = \frac{2M}{\mu}(1-\tau_1-\tau_2)^2,$ $\overline{B} = \frac{2M}{\mu}(1 -\tau_2)^2$, and $\overline{C} = \frac{1}{2\alpha}+\frac{\mu}{12}+M\tau_1^2$. Then by Lemma~\ref{lem:cp2}, all conditions required in Lemma~\ref{telescoping} are satisfied. Hence, we have \eqref{eq:key}.

Since $\frac{2M}{\mu} \le \frac{1}{8}$ from \eqref{eq:M-set1-ineq}, we have $\overline{A} \le \frac{1}{8}$ and $\overline{B}\leq \frac{1}{8}$, and thus by Lemma~\ref{usualrate}, 
\begin{equation}\label{eq:temp-ineq1}
\frac{(\tau_1+\tau_2-1+1/\theta-\overline{A}\tau_1)\theta}{\tau_2+\overline{B}\tau_1} > \frac{13}{12}.
\end{equation}
In addition, since $\theta= 1+\frac{1}{12m},$ we have $\theta^{m} \geq\frac{13}{12}.$
Therefore, \eqref{eq:key} %upon rearranging and utilizing $\theta^{m} \geq\frac{13}{12} $ and \eqref{eq:temp-ineq1}, 
implies
\begin{multline}\label{expdecay}
    \left(\frac{(1-\tau_1-\tau_2)}{\tau_1}+\overline{A}\right) D_{(s+1) m} +\left(\frac{\tau_2}{\tau_1}+\overline{B}\right) \widetilde{D}^{s+1}\sum_{j=0}^{m-1}\theta^j +\overline{C} \EE\big[\|z_{(s+1) m}-x^{\ast}\|^2 \big]\\
\leq  \frac{12}{13} \left(\left(\frac{(1-\tau_1-\tau_2)}{\tau_1}+\overline{A}\right) D_{s m} +\left(\frac{\tau_2}{\tau_1}+\overline{B}\right) \widetilde{D}^s\sum_{j=0}^{m-1}\theta^j +\overline{C} \EE\big[\|z_{s m}-x^{\ast}\|^2\big]\right).
\end{multline}
Repeatedly using \eqref{expdecay} for $s=0$ through $s=S-1$ gives 
\[\left(\frac{\tau_2}{\tau_1}+\overline{B}\right)\widetilde{D}^S\sum_{j=0}^{m-1}\theta^j\le \Big(\frac{12}{13}\Big)^S\left(\left(\frac{(1-\tau_1-\tau_2)}{\tau_1}+\overline{A}\right) D_{0} +\left(\frac{\tau_2}{\tau_1}+\overline{B}\right) \widetilde{D}^0\sum_{j=0}^{m-1}\theta^j +\overline{C} \|z_{0}-x^{\ast}\|^2\right).\]
Dividing by $\left(\frac{\tau_2}{\tau_1}+\overline{B}\right)\sum_{j=0}^{m-1}\theta^j$ both sides of the above inequality, we have
\begin{equation}\label{expdecay-ineq2}
\widetilde{D}^S\le \Big(\frac{12}{13}\Big)^S\left(\frac{\frac{(1-\tau_1-\tau_2)}{\tau_1}+\overline{A}}{\big(\frac{\tau_2}{\tau_1}+\overline{B}\big)\sum_{j=0}^{m-1}\theta^j} D_{0} + \widetilde{D}^0 +\frac{\overline{C}}{\big(\frac{\tau_2}{\tau_1}+\overline{B}\big)\sum_{j=0}^{m-1}\theta^j} \|z_{0}-x^{\ast}\|^2\right).
\end{equation}
Notice $\sum_{j=0}^{m-1}\theta^j\geq m$, $\tau_1=\tau_2=\frac{1}{2m}$, $\overline A \le \frac{1}{8}$, and $\overline B > 0$, and thus we have 
\begin{equation}\label{expdecay-ineq3}
\frac{\frac{(1-\tau_1-\tau_2)}{\tau_1}+\overline{A}}{\big(\frac{\tau_2}{\tau_1}+\overline{B}\big)\sum_{j=0}^{m-1}\theta^j}\le 2.
\end{equation}
In addition, recall $M\le \frac{\mu}{16}$ from \eqref{eq:M-set1-ineq}. Hence, $\overline{C} = \frac{1}{2\alpha}+\frac{\mu}{12}+M\tau_1^2 =  \frac{1}{2\alpha}+\frac{\mu}{12} + \frac{\mu}{4m^2} \le \frac{1}{2\alpha}\left(1+\frac{\alpha\mu}{3}\right)$, and thus 
\begin{equation}\label{expdecay-ineq4}
\frac{\overline{C}}{\big(\frac{\tau_2}{\tau_1}+\overline{B}\big)\sum_{j=0}^{m-1}\theta^j} \le \frac{ 1+\alpha\mu/3}{2m\alpha}.
\end{equation}
Moreover, by the $\mu$-strong convexity of $H$, it holds $H(x_0) - H(x^*) \ge \frac{\mu}{2}\|x_0-x^{\ast}\|^2=\frac{\mu}{2}\|z_0-x^{\ast}\|^2$. Plugging this inequality and also \eqref{expdecay-ineq3} and \eqref{expdecay-ineq4} into \eqref{expdecay-ineq2}, we obtain by recalling the definition of $\widetilde D^s$ and $D_k$ in \eqref{def-Dk-Dtilde} that
$$\EX\big[H(\widetilde{x}^{S})-H(x^{\ast})\big] \le \Big( \frac{12}{13} \Big)^{S} \left(3+\frac{ 1+\alpha\mu/3}{m\alpha\mu}\right)\big(H(x_0)-H(x^{\ast})\big).$$
Since $\frac{1}{m\alpha\mu}\le 6$ and $\alpha\mu <1$ as we showed at the beginning of the proof, it holds $\frac{ 1+\alpha\mu/3}{m\alpha\mu} \le 8$, which together with the above inequality gives the desired result.
\end{proof}

By Theorem~\ref{thm:sock}, we can estimate the complexity of  Algorithm~\ref{alg:gock} in terms of the number of evaluations on $\nabla F_i$, $G_j$, $\nabla G_j$, and the proximal mapping of $h$.

\begin{cor}[complexity result with \textbf{Option I}]\label{cor:sock}
~Given $\varepsilon>0$, under the same assumptions as in Theorem~\ref{thm:sock}, the complexity of Algorithm~\ref{alg:gock} to obtain a stochastic $\varepsilon$-solution is 
\begin{equation}\label{eq:cor:sock}
O\left( \left(n_1+n_2+{\textstyle\sqrt{\frac{L}{\mu}}}\Big({\textstyle\frac{\max\{B_G^4 L_F^2,\,B_F^2 L_G^2\}}{\mu^2}+n_1}\Big)\right) \log\frac{H(x_0)-H(x^{\ast})}{\varepsilon}\right).
\end{equation} 
\end{cor}

\begin{proof} 
To produce one snapshot point $\widetilde x^s$, the complexity is $O\big(n_1+n_2+ m(A + B+n_1)\big)=O \Big(n_1+n_2+\sqrt{\frac{L}{\mu}}\big(\frac{\max\{B_G^4 L_F^2,\,B_F^2 L_G^2\}}{\mu^2}+n_1\big)\Big)$.
In addition, by Theorem~\ref{thm:sock}, to have $\EX\big[H(\widetilde{x}^{S})-H(x^{\ast}) \big] \le \varepsilon$, it suffices to have $S=O\left( \log\frac{H(x_0)-H(x^{\ast})}{\varepsilon}\right)$.   Hence, the total complexity is $O\big(S(n_1+n_2+ m(A + B+n_1))\big)$, which is the desired result in \eqref{eq:cor:sock}. %and we complete the proof. 
\end{proof}

\begin{remark}
From the complexity result in \eqref{eq:cor:sock}, we can easily see that when $n_1=O\left(\frac{\max\{B_G^4 L_F^2,\,B_F^2 L_G^2\}}{\mu^2}\right)$, the result will not become better even if we replace $\nabla F(\widehat G_k)$ in \textbf{Option I} by its stochastic approximation with a mini-batch sampling. Hence, in this case of relatively small $n_1$, we should always take \textbf{Option I} in Algorithm~\ref{alg:gock}.
\end{remark}

\section{Convergence results of the SoCK %KaSCoCO
 method with outer mini-batch step}\label{sec:gock}

In this section, we assume that both of $n_1$ and $n_2$ are very big, and we analyze the convergence rate of Algorithm~\ref{alg:gock} that takes \textbf{Option II} and estimate its complexity to produce a stochastic $\varepsilon$-solution of \eqref{genP}.  
Our analysis shares a similar flow as that in the previous section. We first bound $\|\widetilde{\nabla}_{k+1}-\nabla f(x_{k+1})\|^2$ and show that \eqref{eq:cp2} holds with an appropriate choice of $M$.
\begin{lem} \label{lem:samplevariance3}
~If Assumptions~%\ref{singleone},  
\ref{gentwo} and \ref{genthree} are satisfied, and $\widetilde{\nabla}_{k+1}$ is computed by \textbf{Option II} in Algorithm~\ref{alg:gock}, then
\begin{equation}\label{SV3}
\EX_{k-1}\big[\|\widetilde{\nabla}_{k+1}-\nabla f(x_{k+1})\|^2\big]\leq \left(\frac{4 B_G^4 L_F^2}{A} + \frac{4 B_F^2 L_G^2}{B} + \frac{2L^2}{C}\right)\|\widetilde{x}^s-x_{k+1}\|^2.
\end{equation}
\end{lem}

\begin{proof} 
Define 
\begin{equation}\label{eq:def-Uk}
U_{k+1} = \frac{1}{C}\sum_{i\in\C_k}\left(\left[\nabla G(x_{k+1})\right]^\top  \nabla F_{i}( G(x_{k+1}))-\left[\nabla G(\widetilde{x}^s)\right]^\top  \nabla F_{i}(G(\widetilde{x}^s))\right) +\nabla f(\widetilde{x}^s).
\end{equation} % the unbiased counterpart of $\widetilde{\nabla}_{k+1}$ on the $\nabla f (x_{k+1})$ estimation. 
By the Young's inequality, it holds that
%We systematically expand $\EX_{k-1}[\|\widetilde{\nabla}_{k+1}-\nabla f(x_{k+1})\|^2]$ and rewrite as follows: it can be expanded into 
\begin{equation}\label{chain1}
\EX_{k-1}\big[\|\widetilde{\nabla}_{k+1}-\nabla f(x_{k+1})\|^2\big]\leq 2\EX_{k-1}\big[\|\widetilde{\nabla}_{k+1}-U_{k+1}\|^2\big]+2\EX_{k-1}\big[\|U_{k+1}-\nabla f(x_{k+1})\|^2\big].
\end{equation}
By the definition of $\widetilde{\nabla}_{k+1}$ in \textbf{Option II} of Algorithm~\ref{alg:gock} and the definition of $U_{k+1}$ in \eqref{eq:def-Uk}, we have
\begin{align*}
\EX_{k-1}\big[\|\widetilde{\nabla}_{k+1}-U_{k+1}\|^2\big] =&~ \frac{1}{C^2}\EX_{k-1}\left[ \left\| \sum_{i\in\C_k} \left(\big[\nabla \widehat{G}_{k}\big]^\top  \nabla F_{i}(\widehat{G}_{k}) - \left[\nabla G(x_{k+1})\right]^\top  \nabla F_{i}\big(G(x_{k+1})\big)\right)  \right\|^2 \right]\cr
\le &~\frac{1}{C }\sum_{i\in\C_k}\EX_{k-1}\left[ \left\|  \big[\nabla \widehat{G}_{k}\big]^\top  \nabla F_{i}(\widehat{G}_{k}) - \left[\nabla G(x_{k+1})\right]^\top  \nabla F_{i}\big(G(x_{k+1})\big)  \right\|^2 \right].
\end{align*}
Applying Lemma~\ref{lem:summandbound} with $g= F_i$ for each $i\in \C_k$ to the right-hand-side (r.h.s.) of the above inequality gives
\begin{equation}\label{chain1-ineq1}
\EX_{k-1}\big[\|\widetilde{\nabla}_{k+1}-U_{k+1}\|^2\big] \le 2\left(\frac{ B_G^4 L_F^2}{A} + \frac{ B_F^2 L_G^2}{B} \right)\|\widetilde{x}^s-x_{k+1}\|^2.
\end{equation}
For the second term of the r.h.s. of \eqref{chain1}, we have
	\begin{align}\label{chain1-ineq2}
&~\EX_{k-1}\big[\|U_{k+1}-\nabla f(x_{k+1})\|^2\big]\cr
=  & \! \frac{1}{C^2} \! \EX_{k-1} \! \left[ \left\| \sum_{i\in\C_k} \! \left(\left[\nabla G( \! x_{k+1} \! )\right]^\top \! \nabla F_{i}( \! G( \! x_{k+1} \! ) \! ) \! - \! \left[\nabla G(\widetilde{x}^s)\right]^\top \! \nabla F_{i}( \! G(\widetilde{x}^s) \! ) \! - \! \nabla f( \! x_{k+1} \! )  \! +  \! \nabla f(\widetilde{x}^s)\right) \! \right\|^2 \right]\cr
=  & \! \frac{1}{C^2} \! \EX_{k-1} \! \left[\sum_{i\in\C_k} \!  \left\|  \left[\nabla G( \! x_{k+1} \! )\right]^\top \! \nabla F_{i}( \! G( \! x_{k+1} \! ) \! ) \! - \! \left[\nabla G(\widetilde{x}^s)\right]^\top \! \nabla F_{i}( \! G(\widetilde{x}^s) \! ) \! - \! \nabla f( \! x_{k+1} \! )  \! +  \! \nabla f(\widetilde{x}^s)   \right\|^2 \right] \! ,
\end{align}
where the second equality holds because the summands are conditionally independent and each has a zero mean. Since the variance of a random vector is bounded by its second moment, we have for each $i\in \C_k$ that
\begin{align}
&~\EX_{k-1}\left[ \left\|  \left[\nabla G(x_{k+1})\right]^\top  \nabla F_{i}( G(x_{k+1}))-\left[\nabla G(\widetilde{x}^s)\right]^\top  \nabla F_{i}(G(\widetilde{x}^s))-\nabla f(x_{k+1}) + \nabla f(\widetilde{x}^s)   \right\|^2 \right]\cr
\le &~\EX_{k-1}\left[ \left\|  \left[\nabla G(x_{k+1})\right]^\top  \nabla F_{i}( G(x_{k+1}))-\left[\nabla G(\widetilde{x}^s)\right]^\top  \nabla F_{i}(G(\widetilde{x}^s))  \right\|^2 \right] \label{chain1-ineq3-1}\\
\le &~ L^2\|\widetilde{x}^s-x_{k+1}\|^2, \label{chain1-ineq3}
\end{align}
where the second inequality follows from \eqref{eq:smooth-fi}. Substituting \eqref{chain1-ineq3} into \eqref{chain1-ineq2} yields
\begin{equation}\label{chain1-ineq4}
\EX_{k-1}\big[\|U_{k+1}-\nabla f(x_{k+1})\|^2\big] \le \frac{L^2}{C} \|\widetilde{x}^s-x_{k+1}\|^2.
\end{equation}
We obtain the desired result by plugging \eqref{chain1-ineq1} and \eqref{chain1-ineq4} into \eqref{chain1}. 
\end{proof}

Plugging \eqref{SV3} into \eqref{eq:ineq-for-all-case}, we are able to show \eqref{eq:cp2} and thus \eqref{eq:key} with an appropriate $M$. 
\begin{lem}%[coupling step 2]
\label{lem:cp2nos}
~Suppose $\tau_1\in (0, \frac{1}{3\alpha L}]$ and $\tau_2\in[0,1-\tau_1] $ in the linear coupling step of Algorithm \ref{alg:gock}. Let $x^*$ be the solution of \eqref{genP}. If $\widetilde{\nabla}_{k+1}$ is computed by \textbf{Option II}, then \eqref{eq:cp2} holds with  
\begin{equation}\label{eq:M-set2}
M = 3\Big(\frac{1}{4\tau_1L}+\frac{3}{\mu}\Big)\left(\frac{4 B_G^4 L_F^2}{A} + \frac{4 B_F^2 L_G^2}{B}+ \frac{2 L^2}{C}\right).
\end{equation} 
\end{lem}
\begin{proof} 
Taking conditional expectation $\EE_{k-1}$ on both sides of \eqref{eq:ineq-for-all-case} with $\beta=\frac{6}{\mu}$ and plugging \eqref{SV3}, we immediately have \eqref{eq:cp2} by using the choice of $M$ in \eqref{eq:M-set2}. 
\end{proof} 

We are now to show our second main result.

\begin{thm}[convergence result for \textbf{Option II}%without Assumption~\ref{genfour}
	]\label{thm:gocknos}
~Under Assumptions~\ref{singleone}, \ref{gentwo} and \ref{genthree}, let $\{\widetilde x^s\}$ be generated from Algorithm~\ref{alg:gock} with $\widetilde\nabla_{k+1}$ computed by \textbf{Option II} and with parameters set as follows:
\begin{subequations}\label{eq:para-2}
\begin{align}
m\leftarrow\left\lceil \frac{1}{2}\sqrt{\frac{L}{\mu}}\right\rceil,\, \tau_1\leftarrow  \frac{1}{2m},\, \tau_2\leftarrow \frac{1}{2m},\, \theta\leftarrow 1+\frac{1}{12m},\, \alpha\leftarrow \frac{1}{3\tau_1 L},\\
A\leftarrow\frac{2160 B_G^4 L_F^2}{\mu^2},\, B\leftarrow\frac{2160 B_F^2 L_G^2}{\mu^2},\, C\leftarrow\frac{1080  L^2}{\mu^2}.\label{eq:para-2-ABC}
\end{align}
\end{subequations}
Then
\begin{equation*}
\EX\big[H(\widetilde{x}^{S})-H(x^{\ast})\big] \leq %\begin{cases}
%O\left( (1+\sqrt{\frac{\mu}{24m L_f}})^{-Sm}\right)\sqrt{\frac{L_f}{m\mu}}(H(x_0)-H(x^{\ast})) &\text{if $m\leq \frac{2L_f}{3\mu};$}\\
11\cdot \left(  \frac{12}{13}\right)^{S}\big(H(x_0)-H(x^{\ast})\big). %&\text{if  $m> \frac{2L_f}{3\mu}.$}
%\end{cases}
\end{equation*}
\end{thm}

\begin{proof} 
Notice that the parameters given in \eqref{eq:para-2} are the same as those in \eqref{eq:para-1} except for $A,B$ and $C$, and also notice that the choices of $A,B$ and $C$ only affect the value of $M$. Plugging into \eqref{eq:M-set2} the values of $A,B,C$ given in \eqref{eq:para-2-ABC}, we can easily verify that $M\le \frac{\mu}{16}$. Now following the same arguments in the proof of Theorem~\ref{thm:sock}, we obtain the desired result. 
\end{proof}

By Theorem~\ref{thm:gocknos}, we can estimate the complexity of  Algorithm~\ref{alg:gock} in terms of the number of evaluations on $\nabla F_i$, $G_j$, $\nabla G_j$, and the proximal mapping of $h$. Its proof follows that of Corollary~\ref{cor:sock}, and we omit it.

\begin{cor}\label{cor:gocknos}
~Given $\varepsilon>0$, under the same assumptions as in Theorem~\ref{thm:gocknos}, the complexity of Algorithm~\ref{alg:gock} to produce a stochastic $\varepsilon$-solution is \[O\left( \Big(n_1 + n_2+\sqrt{\frac{L}{\mu}}\frac{\max\{B_G^4 L_F^2,\,B_F^2 L_G^2,\,L^2\}}{\mu^2}\Big) \log\frac{H(x_0)-H(x^{\ast})}{\varepsilon}\right).\] 
\end{cor}

\section{Treating non-strongly convex compositional optimization}\label{sec:nonsc}

In this section, we study convex (but may not be strongly-convex) finite-sum compositional optimization in the form of \eqref{genP}, namely, instead of Assumption~\ref{singleone}, we make the following assumption: 
\begin{assump}\label{singleone-cvx}
 ~The function $f$ given in \eqref{eq:def-f} is convex, and the function %and $L_f$-smooth, i.e. with an $L_f$-Lipschitz gradient. 
  $h$ in \eqref{genP} is convex. %$\mu$-strongly convex with $\mu>0$.
\end{assump}
The non-strongly convex case has been studied in \cite{lin2018improved}.  
The algorithmic design in \cite{lin2018improved} is built on the algorithms for strongly convex models. They modify earlier algorithms for strongly convex COPs in a way that their step size ranges in a fixed magnitude that is free from $\kappa,$ 
%They take a strongly convex problem solver as a subroutine 
and their number of inner iterations exponentially increases %while taking them as inner loops with the number of inner loops growing exponentially 
as the outer loop proceeds. This way, the final outer loop dominates the total computational cost. We adopt a similar trick to find an approximate solution of a convex model by solving a sequence of slightly perturbed but strongly convex problems. More precisely, we follow \cite{allen2016optimal} and implement the black-box reductions; see Algorithm~\ref{alg:gock-cvx}, where $H_t^*$ denotes the optimal value of the problem in the $t$-th outer loop. 

\begin{algorithm}[H]                      
\caption{ \textbf{N}on-strongly C\textbf{o}nvex  \textbf{C}ompositional \textbf{K}atyusha (NoCK)}          
\label{alg:gock-cvx}                           
\begin{algorithmic} 
\State \textbf{Input:} $x_0\in \mathbb{R}^{N_2}$, initial strong-convexity constant $\mu_0$, and outer loop number $T$
\For{$t=0,1,\ldots,T-1$}
\State Apply Algorithm~\ref{alg:gock} with starting point $x_t$ to $\min_x \big\{H_t(x):=H(x)+\frac{\mu_t}{2}\|x-x_0\|^2\big\}$ and find $x_{t+1}$ such that
\begin{equation}\label{eq:rule-x-t}
\EE\big[H_t(x_{t+1})-H_t^*\,\big|\, x_t\big] \le \frac{H_t(x_t)-H_t^*}{4}.
\end{equation}
\State Let $\mu_{t+1}=\mu_t/2$.
\EndFor
\State \textbf{Return} $x_T$
\end{algorithmic}
\end{algorithm}

The following result is from \cite[Theorem 3.1]{allen2016optimal} and can be proved in the same way.
\begin{lem}\label{AdaptReg}
~Let $x^*$ be an optimal solution of \eqref{genP} and $H^*$ be the optimal objective value.
Suppose $H(x_0) - H^* \le D_H$ and $\|x_0-x^*\|^2\le D_x$. Set $\mu_0=D_H/D_x$ and $T=\log_2(D_H/\varepsilon)$. Then $\EE [H(x_T)-H^*] \le \varepsilon$ and the total complexity is $\sum_{t=0}^{T-1}\mathrm{Time}(t)$, where $\mathrm{Time}(t)$ denotes the complexity to produce $x_{t+1}$.
\end{lem}

\begin{rmk}
~The setting of $T$  and $\mu_0$ requires an estimate of $D_H$ and  $D_H/D_x$, respectively. This can be done if the domain of $h$ is bounded and we know the bound. Otherwise, {since the algorithm can terminate at any moment,} one can simply tune $\mu_0$ {as a ratio of two unknowns}, as explained in \cite{allen2016optimal}.  Note that the theoretical guarantee of \cite{lin2018improved} requires $D_H+L\cdot D_x$ instead for its termination. 

{In practice, \eqref{eq:rule-x-t} is not checkable as it requires the optimal value $H^*_t$. Nevertheless,  
it would hold by Theorems~\ref{thm:sock} and \ref{thm:gocknos} if $S\ge\log 44/\log(12/11)$ is used in the subroutine SoCK. This fixed $S$ can be too big to yield good numerical performance.  
In Section~\ref{sec:comp-nock}, we notice that NoCK can still perform well even with $S=3$.}
\end{rmk}

Applying Corollaries~\ref{cor:sock}, and \ref{cor:gocknos} %, and \ref{cor:gockcos}
 together with Lemma~\ref{AdaptReg}, we can easily have the complexity result of Algorithm~\ref{alg:gock-cvx} to produce a stochastic $\varepsilon$-solution of \eqref{genP} when it is non-strongly convex. %obtain the following through similar derivation as the above.

\begin{thm}[case of relatively small $n_1$]\label{non:sock}
Let $\varepsilon>0$. Suppose Algorithm~\ref{alg:gock} with \textbf{Option I} is applied as the subroutine in Algorithm~\ref{alg:gock-cvx}. Then under Assumptions~\ref{singleone-cvx}, \ref{gentwo} and \ref{genthree}, Algorithm~\ref{alg:gock-cvx} with $\mu_0$ and $T$ set to those in Lemma~\ref{AdaptReg} can produce a stochastic $\varepsilon$-solution of \eqref{genP} with overall complexity  
\begin{equation}\label{eq:non:sock}
\sum_{t=0}^{T-1}\mathrm{Time}(t) = O\left((n_1+n_2)\log\frac{D_H}{\varepsilon}  + \frac{D_x^{2.5}\sqrt{L}\cdot\max\{B_G^4 L_F^2,\,B_F^2 L_G^2 \}}{\varepsilon^{2.5}} + \frac{\sqrt{D_x L}n_1}{\sqrt\varepsilon}\right). 
\end{equation}
\end{thm}

\begin{proof} 
From Lemma~\ref{AdaptReg}, it follows that $x_T$ is a stochastic $\varepsilon$-solution of \eqref{genP}, and thus we only need to show the overall complexity $\sum_{t=0}^{T-1}\mathrm{Time}(t)$.
By Corollary~\ref{cor:sock} and $\mu_t = \frac{\mu_0}{2^t}$, the complexity to produce $x_{t+1}$ that satisfies \eqref{eq:rule-x-t} is
\begin{equation*}
\textup{Time}(t)= %\begin{cases}
O\left(n_1+n_2+ 2^{2.5t}\sqrt{\frac{L}{\mu_0}}\frac{\max\{B_G^4 L_F^2,\,B_F^2 L_G^2 \}}{\mu_0^2} + 2^{0.5t}n_1\sqrt{\frac{L}{\mu_0}}\right).
\end{equation*}
Hence, the total complexity is
\begin{align*}
\sum_{t=0}^{T-1}\mathrm{Time}(t) = O\left((n_1+n_2)T+\frac{2^{2.5T}-1}{2^{2.5}-1}\sqrt{\frac{L}{\mu_0}}\frac{\max\{B_G^4 L_F^2,\,B_F^2 L_G^2 \}}{\mu_0^2}+\frac{2^{0.5T}-1}{2^{0.5}-1}n_1\sqrt{\frac{L}{\mu_0}}\right).
\end{align*}
Since $\mu_0=\frac{D_H}{D_x}$ and $T=\log_2\frac{D_H}{\varepsilon}$, the above equation implies the desired result. 
\end{proof}

\begin{remark}
From the result in \eqref{eq:non:sock}, we see that if $n_1=O\left(\frac{D_x^2\max\{B_G^4 L_F^2,\,B_F^2 L_G^2 \}}{\varepsilon^2}\right)$, we should always take \textbf{Option I} while applying Algorithm~\ref{alg:gock}, since even if we do mini-batch sampling, the complexity result will not become better.
\end{remark}

\begin{thm}[case of big $n_1$]%without Assumption~\ref{genfour}]
	\label{non:gocknos}
Let $\varepsilon>0$, and let Algorithm~\ref{alg:gock} with \textbf{Option II} be applied as the subroutine in Algorithm~\ref{alg:gock-cvx}. Then under Assumptions~\ref{singleone-cvx}, \ref{gentwo} and \ref{genthree}, Algorithm~\ref{alg:gock-cvx} with $\mu_0$ and $T$ set to those in Lemma~\ref{AdaptReg} can produce a stochastic $\varepsilon$-solution of \eqref{genP} with overall complexity 
\begin{equation*}
\sum_{t=0}^{T-1}\mathrm{Time}(t) =O\left((n_1+n_2)\log\frac{D_H}{\varepsilon}  + \frac{D_x^{2.5}\sqrt{L}\cdot\max\{B_G^4 L_F^2,\,B_F^2 L_G^2, L^2 \}}{\varepsilon^{2.5}}\right). 
\end{equation*}
\end{thm}

\begin{proof} 
First, notice again that $x_T$ is a stochastic $\varepsilon$-solution of \eqref{genP}. By Corollary~\ref{cor:gocknos} and $\mu_t = \frac{\mu_0}{2^t}$, the complexity to produce $x_{t+1}$ that satisfies \eqref{eq:rule-x-t} is
\begin{equation*}
\textup{Time}(t)= %\begin{cases}
O\left(n_1+n_2+ 2^{2.5t}\sqrt{\frac{L}{\mu_0}}\frac{ \max\{B_G^4 L_F^2,\,B_F^2 L_G^2,\,L^2\}}{\mu_0^2}\right) 
\end{equation*}
Now summing $\textup{Time}(t)$ over $t$ and following the proof of Theorem~\ref{non:sock}, we obtain the desired result. 
\end{proof}

{\section{Numerical Experiments}
%}
In this section, we conduct some experiments to demonstrate the numerical performance of Algorithms~\ref{alg:gock} and \ref{alg:gock-cvx}.  The problem that we test is an $\ell_1$ regularized version of \eqref{meanv01} with linear loss function $\ell$, namely, we solve %\ell(\boldsymbol{x};\boldsymbol{a}_i,b_i)\equiv \boldsymbol{a}_i^\top\boldsymbol{x}:$
\begin{equation}\label{eq:exp1}
	\displaystyle\min_{\boldsymbol{x}\in \RR^N}\,\,\frac{1}{n}\sum_{i=1}^{n}\boldsymbol{a}_i^\top\boldsymbol{x} +\frac{\lambda_1}{n} \sum_{i=1}^{n}\left[\boldsymbol{a}_i^\top \boldsymbol{x} -\frac{1}{n}\sum_{j=1}^{n}\boldsymbol{a}_j^\top\boldsymbol{x} \right]^2+\lambda_2\|\boldsymbol{x}\|_1.
\end{equation}
This problem has also been tested in other papers such as \cite{lian2017finite,lin2018improved,huo2018accelerated,zhang2019composite}. %{\color{red}Mention this problem has also been tested in other papers such as [?](updated)} 
In our tests, we generate the data matrix $\boldsymbol{A}=[\boldsymbol{a}_1,\ldots,\boldsymbol{a}_n]$ by drawing i.i.d. columns from a multivariate normal distribution $\cN(\vzero, \boldsymbol{\Sigma})$ with $\boldsymbol{\Sigma}=\boldsymbol{L}^\top \boldsymbol{L} +v\boldsymbol{I}_{N}$. Here, $\boldsymbol{L}\in \RR^{N\times N}$ is generated by drawing i.i.d. entries from a standard normal distribution, $\boldsymbol{I}_N$ is the identity matrix, and $v$ is a positive parameter to control the smoothness and strong-convexity parameters $L$ and $\mu$ of the smooth part in the objective of \eqref{eq:exp1}. Notice that \eqref{eq:exp1} can be rewritten as 
\begin{equation}\label{eq:exp2}
	\displaystyle\min_{\boldsymbol{x}\in \RR^N}\,\,\bar{\boldsymbol{a}}^\top\boldsymbol{x} +\lambda_1   \boldsymbol{x}^\top\left(\frac{1}{n}\boldsymbol{A} \boldsymbol{A}^\top -\bar{\boldsymbol{a}} \bar{\boldsymbol{a}}^\top \right)\boldsymbol{x}+\lambda_2\|\boldsymbol{x}\|_1,
\end{equation} 
where $\bar{\boldsymbol{a}}\equiv\frac{1}{n}\sum_{i=1}^{n}\boldsymbol{a}_i$. Hence, $L$, $\mu$, and $\kappa= L /\mu$ can be computed explicitly. For all the tests, we set the dimension to $N=500$.

\subsection{Strongly convex instances}
 In this subsection, we test Algorithm~\ref{alg:gock} on solving \eqref{eq:exp1} and compare to VRSC-PG in \cite{huo2018accelerated}, C-SAGA and the generalized C-SAGA (that we name as GC-SAGA) in \cite{zhang2019composite}. There are two options in Algorithm~\ref{alg:gock}. We name the method as SOCK if Option I is adopted and as GOCK if Option II is used. We tuned the parameters of these methods in order for them to have fast linear convergence, and their values are set as follows. For SOCK and GOCK, we set $m=\left\lceil \frac{1}{2}\sqrt{\kappa}\right\rceil$, $\theta=1+\frac{1}{4 m}$, $\tau_1=\tau_2=\frac{1}{2 m}$, $\alpha=\frac{2 m}{3 L}$, and $A=B=\left\lceil \kappa^2/256 \right\rceil$, and in addition, $C=\left\lceil \kappa^2/16 \right\rceil$ is set for GOCK. For VRSC-PG, we set the length of its inner loop to $m^\prime=\left\lceil \frac{1}{4}{\kappa}\right\rceil$, %{\color{red}(Is it better to say $m^\prime=\left\lceil \frac{1}{4}{\kappa}\right\rceil$?)}
  its inner mini-batch size to $\left\lceil \kappa^2/256 \right\rceil$, and its outer mini-batch size to $\left\lceil \kappa^2/16 \right\rceil$. In addition, its step size is set to $\eta=1/(5 L)$. For both of C-SAGA and GC-SAGA, we set their mini-batch sizes to $s=\left\lceil n^{2/3} /2\right\rceil$ and step size to $\eta=1/(2 L)$. The length of their inner loop is set to $\tau=\left\lceil (\kappa+ n^{1/3} )/2\right\rceil$ and $\tau=\left\lceil (\kappa+ 2 n^{1/3} )/2\right\rceil$, respectively.

The results for one trial are shown in Figure~\ref{fig:sock_comp_t1}, where we generate $n=5000$ or $n=50,000$ data samples and vary the value of $v$ to change the condition number $\kappa$. The results for more trials are shown in Appendix~\ref{app:test}. From the figures, we see that all compared methods converge linearly, and the proposed SOCK and GOCK methods need fewer oracles to reach the same accuracy as compared to the other three methods. When the data samples are the same, all methods take more oracles for the larger $\kappa$ to reach the same accuracy.

\begin{figure}[h]
	\caption{Comparison of the proposed SOCK (i.e., Algorithm~\ref{alg:gock} with Option I) and GOCK (i.e., Algorithm~\ref{alg:gock} with Option II) to VRSC-PG in \cite{huo2018accelerated}, C-SAGA and GC-SAGA in \cite{zhang2019composite} on solving instances of \eqref{eq:exp1} with the number of samples $n=5000$ or $n=50,000$. All methods take the explicitly-computed strong convexity constant $\mu$. The value of $v$ is varied to change the condition number $\kappa$.}%\vspace{-0.1cm}}
	\label{fig:sock_comp_t1}
	\begin{center}
	\begin{tabular}{cccc}
{\small$(n,\kappa,v)=(5e3, 96, 30)$} & {\small$(n,\kappa,v)=(5e3, 52, 60)$} & {\small$(n,\kappa,v)=(5e4, 211, 10)$} & {\small$(n,\kappa,v)=(5e4, 93, 30)$}  \\	
		\includegraphics[width=.23\linewidth]{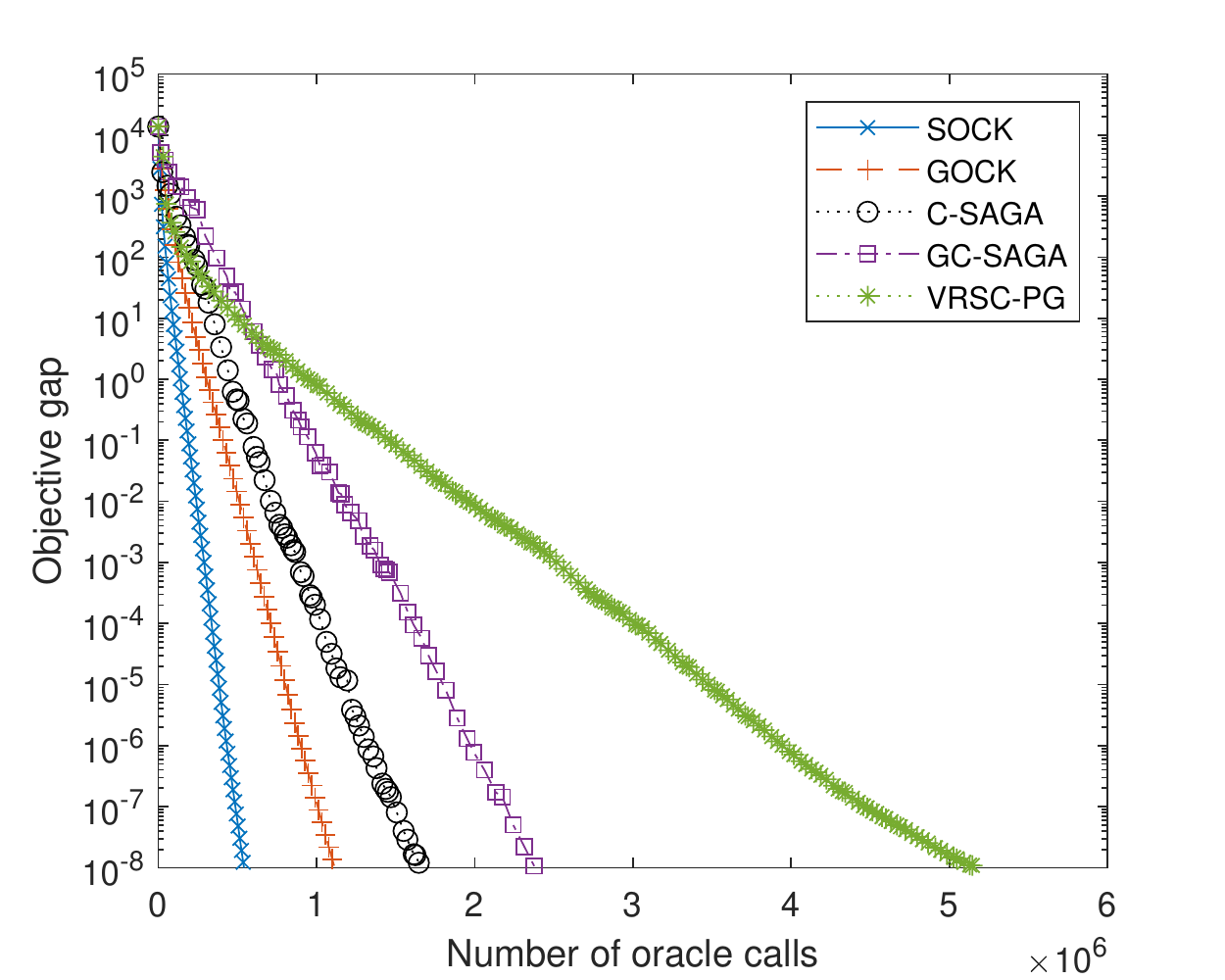} &
    	\includegraphics[width=.23\linewidth]{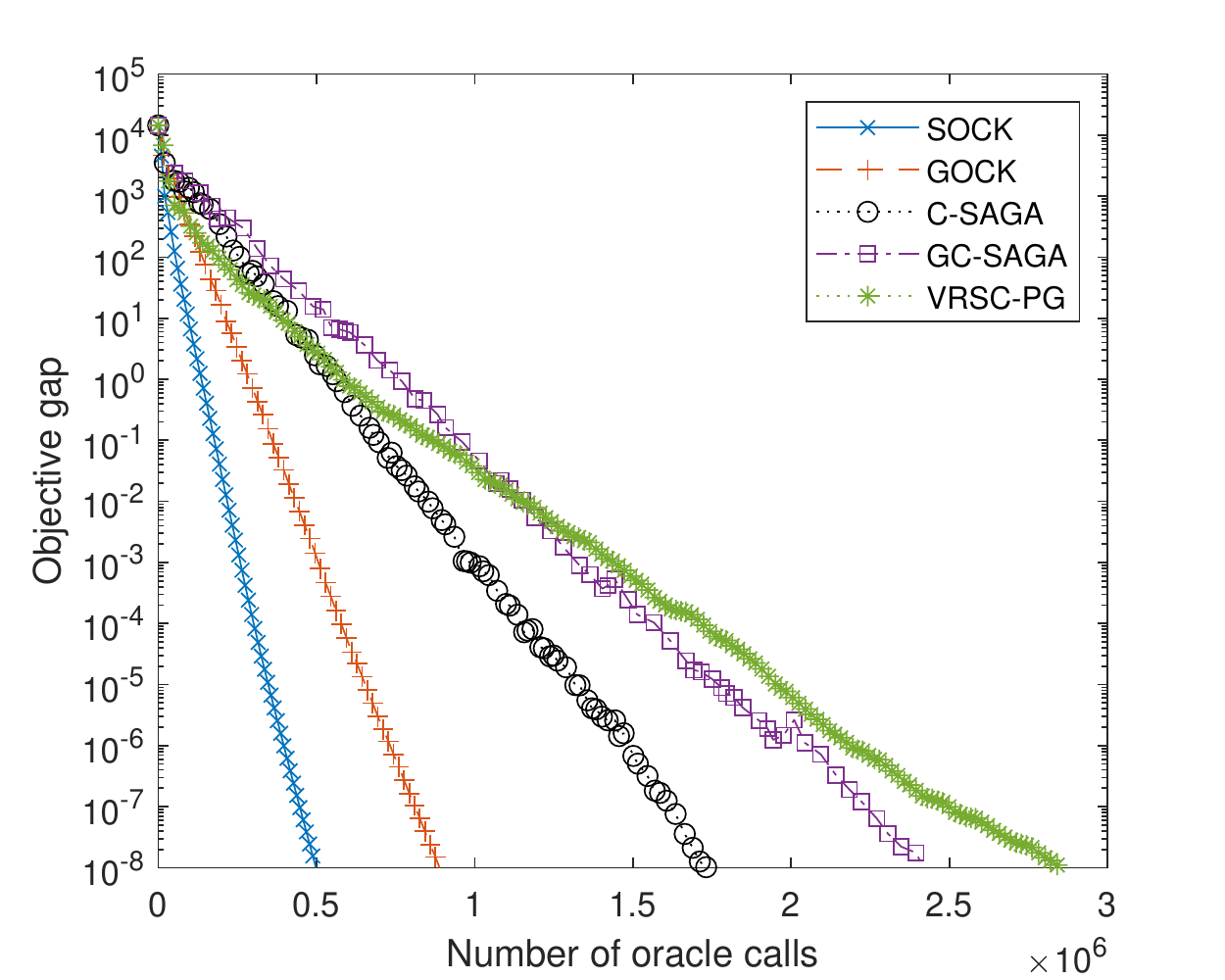}		&
    	\includegraphics[width=.23\linewidth]{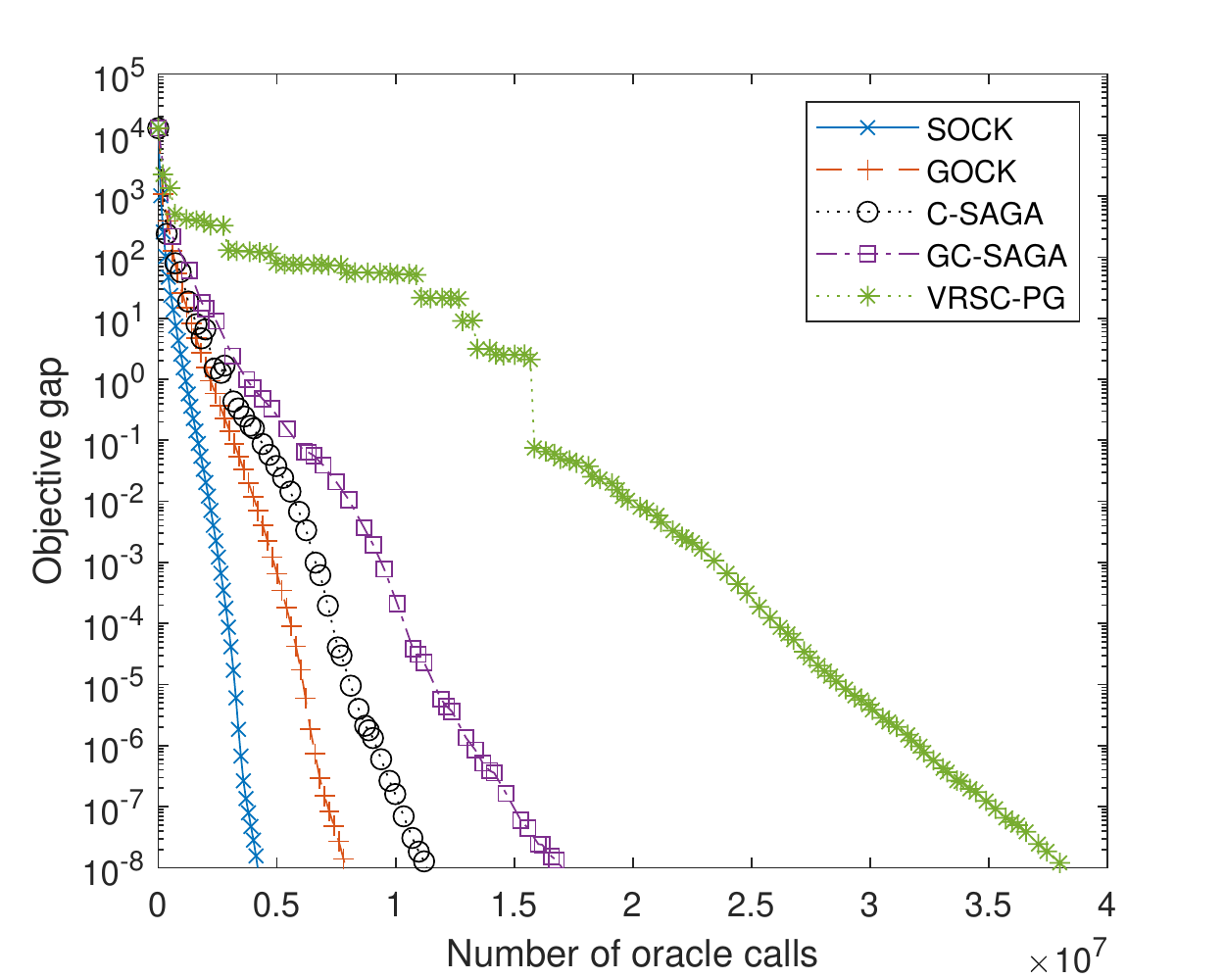}	&
	\includegraphics[width=.23\linewidth]{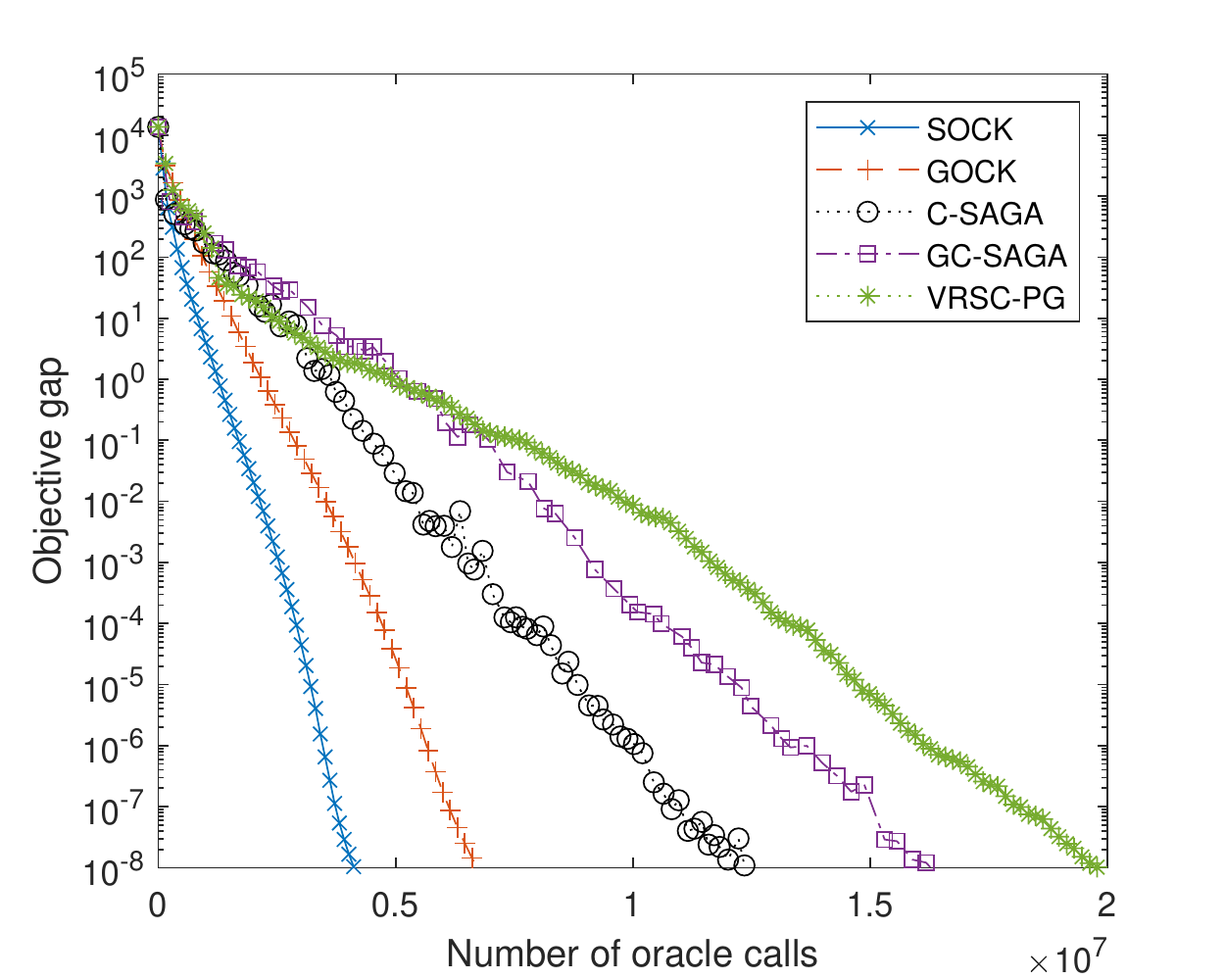}
	\end{tabular}
\end{center}
\end{figure}

\subsection{General convex instances}\label{sec:comp-nock}
In this subsection, we test Algorithm~\ref{alg:gock-cvx} on solving \eqref{eq:exp1} and compare it to SCVRG in \cite{lin2018improved}. Even though the problem instances can be strongly convex with the generated data, we treat it as general convex by simply using $\mu=0$. Since SCVRG in theory requires outer mini-batch sampling, we use Algorithm~\ref{alg:gock} with $S=3$ and Option-II as the subroutine in Algorithm~\ref{alg:gock-cvx} for a fair comparison. In the $t$-th call of Algorithm~\ref{alg:gock}, we set $m$ to $\left\lceil \frac{1}{2}\sqrt{\frac{L_t}{\mu_t}}\right\rceil$. We name our method as NoCK, and for simplicity, we set $\mu_0=1$.  

For both NoCK and SCVRG, we set the maximum number of total outer iterations to $T=\lceil \log_2 \frac{L}{\varepsilon} \rceil$, where $\varepsilon=10^{-4}$ and $L$ can be explicitly computed from the reformulation in \eqref{eq:exp2}; for SCVRG, we set its maximum number of total iterations to $\cal{T}=2 k_0 \left((\frac{3}{2})^T-1\right)$ as \cite{lin2018improved} suggests in its implementation, where $k_0=5$ is set to the default value of the initial inner loop length. %{\color{red}(what is $k_0$? Is $T=\lceil \log_2 \frac{L}{\varepsilon} \rceil$ in the formula of $\cal{T}$?)(updated)} 
The step size %{\color{red} (use step size or step size but not step size)}
 of SCVRG is set to $\eta=\frac{1}{5L}$, the same as that we use for VRSC-PG. Theoretically, SCVRG requires the inner and outer mini-batches $A=B=C=\Theta(\frac{L^2}{\varepsilon^2})$. However, it can have much better numerical performance by using smaller mini-batches \footnote{In the linked code of \cite{lin2018improved}, { $A=B=5$ and $C=1$} is adopted.}. For this reason, we adopt two settings as follows.

1. \textbf{Theoretical parameter choice:}	%To be consistent with their theory, but also make the computation less tedious, we let 
For NoCK, we set the inner mini-batch size to $A_t=B_t=\min\{(\frac{L+\mu_t}{ 400\mu_t })^2,\frac{n}{200},500\} $ and outer mini-batch to $C_t=\min\{(\frac{L+\mu_t}{ 20\mu_t })^2,\frac{n}{200},500\}$, where $t$ is the number of calling Algorithm~\ref{alg:gock}; for SCVRG, we set its inner mini-batch to  $A=B=\min\{(\frac{L}{ 4000\varepsilon })^2,\frac{n}{200},500\} $ and outer mini-batch to $C=\min\{(\frac{L}{ 200\varepsilon })^2,\frac{n}{200},500\}$. Notice that our setting somehow follows the theory by keeping the order $O(\kappa^2)$ or $O(\frac{L^2}{\varepsilon^2})$ but deviates from it by introducing a cap to the min-batch sizes to avoid the use of deterministic gradients.

2. \textbf{Heuristic parameter choice:} For NoCK, we set $A_t=B_t=C_t=500$ for all $t$, and for SCVRG, we set $A=B=C=50$. Notice that in this setting, the mini-batch sizes for NoCK are actually larger than those in the previous choice, while SCVRG uses smaller mini-batch sizes and achieves better performance. During our tuning, we find that SCVRG can diverge for the tested instances if the default setting {  $A=B=5$ and $C=1$} is adopted.

The results are shown in Figure~\ref{fig:nock_comp_t1} for the two parameter settings, where we generate $n=50,000$ or $n=250,000$ samples and vary $v\in \{0, 10\}$. For the former choice, our proposed method NoCK converges faster than SCVRG on all instances, while for the latter choice, SCVRG can be faster when $v=10$. Notice that the strong-convexity constant $\mu$ becomes bigger as $v$ increases. Although neither NoCK or SCVRG exploit the strong convexity, the results indicate that SCVRG can somehow benefit more from the strong convexity.

\begin{figure}[h]
	\caption{Comparison of the proposed NoCK method (i.e., Algorithm~\ref{alg:gock-cvx}) to SCVRG in \cite{lin2018improved} on solving instances of \eqref{eq:exp1} with the number of samples $n=50,000$ or $n=250,000$. The instances are treated as convex, even though they can be strongly convex. Parameters of the algorithms are set by the theoretical parameter choice in the top row and the heuristic choice in the bottom row.}
	\label{fig:nock_comp_t1}
	\begin{center}
	\begin{tabular}{cccc}
{\footnotesize $(n, L, v) = (5e4, 791, 0)$} & {\footnotesize $(n, L, v) = (5e4, 798, 10)$} & {\footnotesize $(n, L, v) = (2.5e5, 806, 0)$} & {\footnotesize $(n, L, v) = (2.5e5, 799, 10)$}\\	
\includegraphics[width=.23\linewidth]{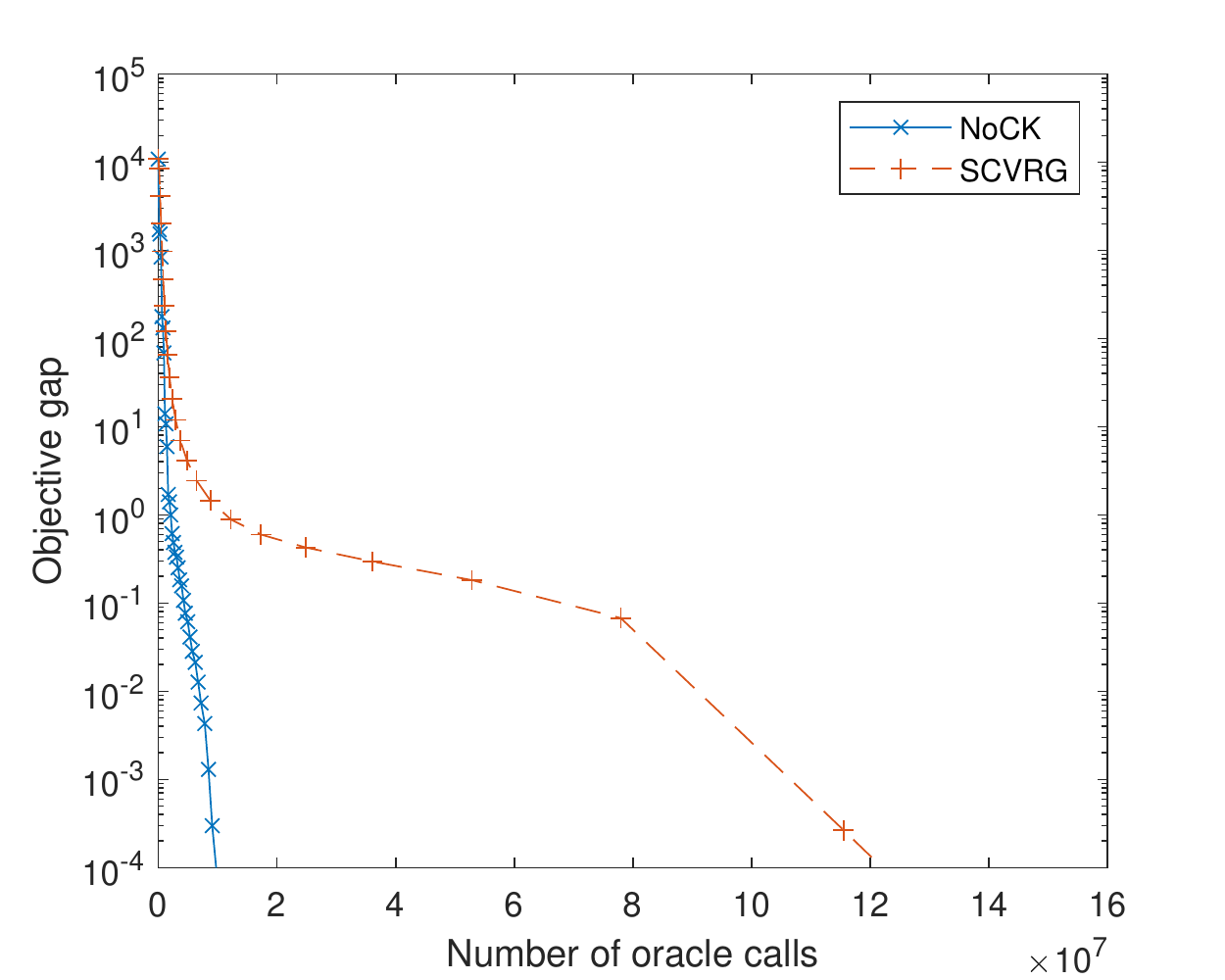}		&
		\includegraphics[width=.23\linewidth]{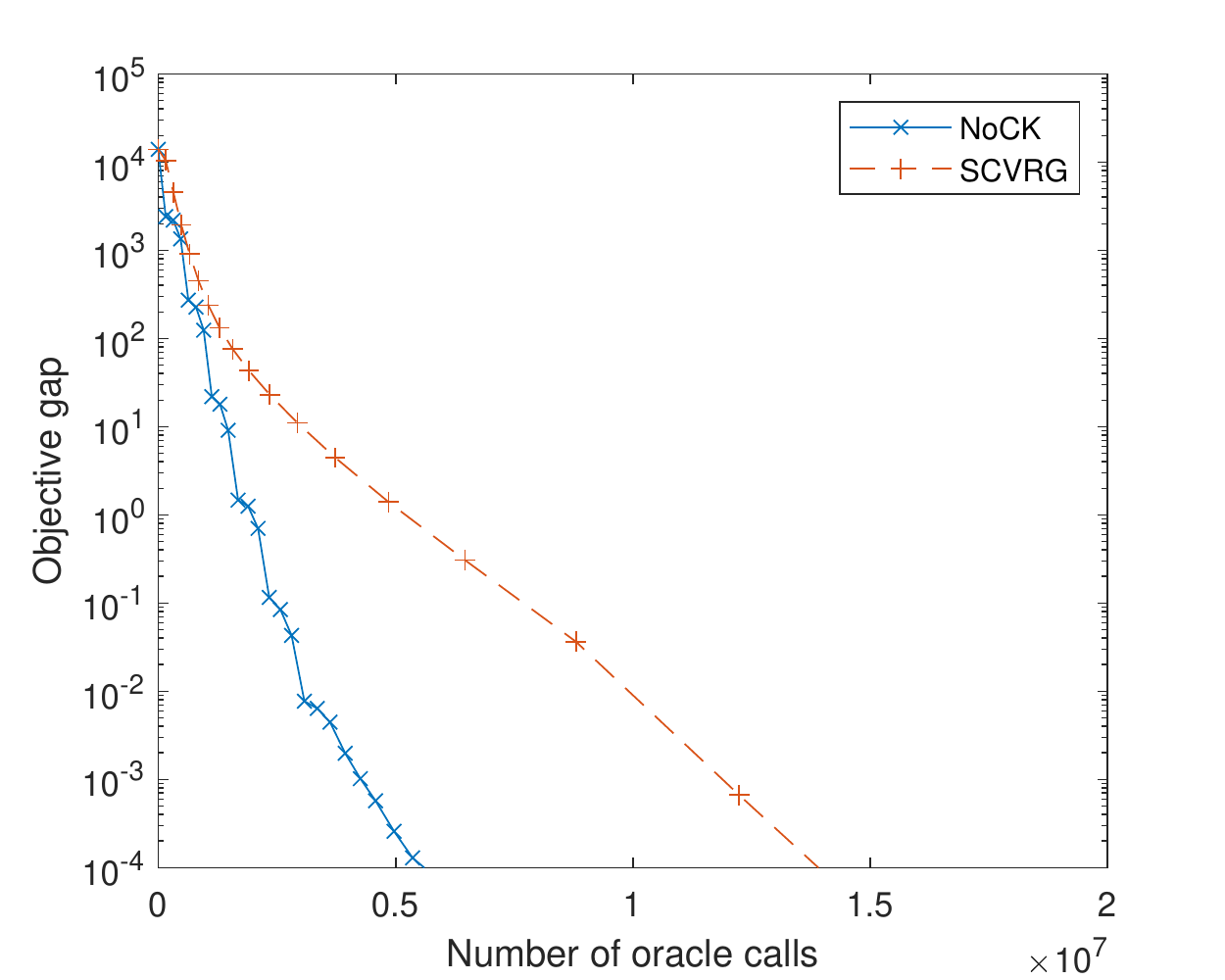}	&
		\includegraphics[width=.23\linewidth]{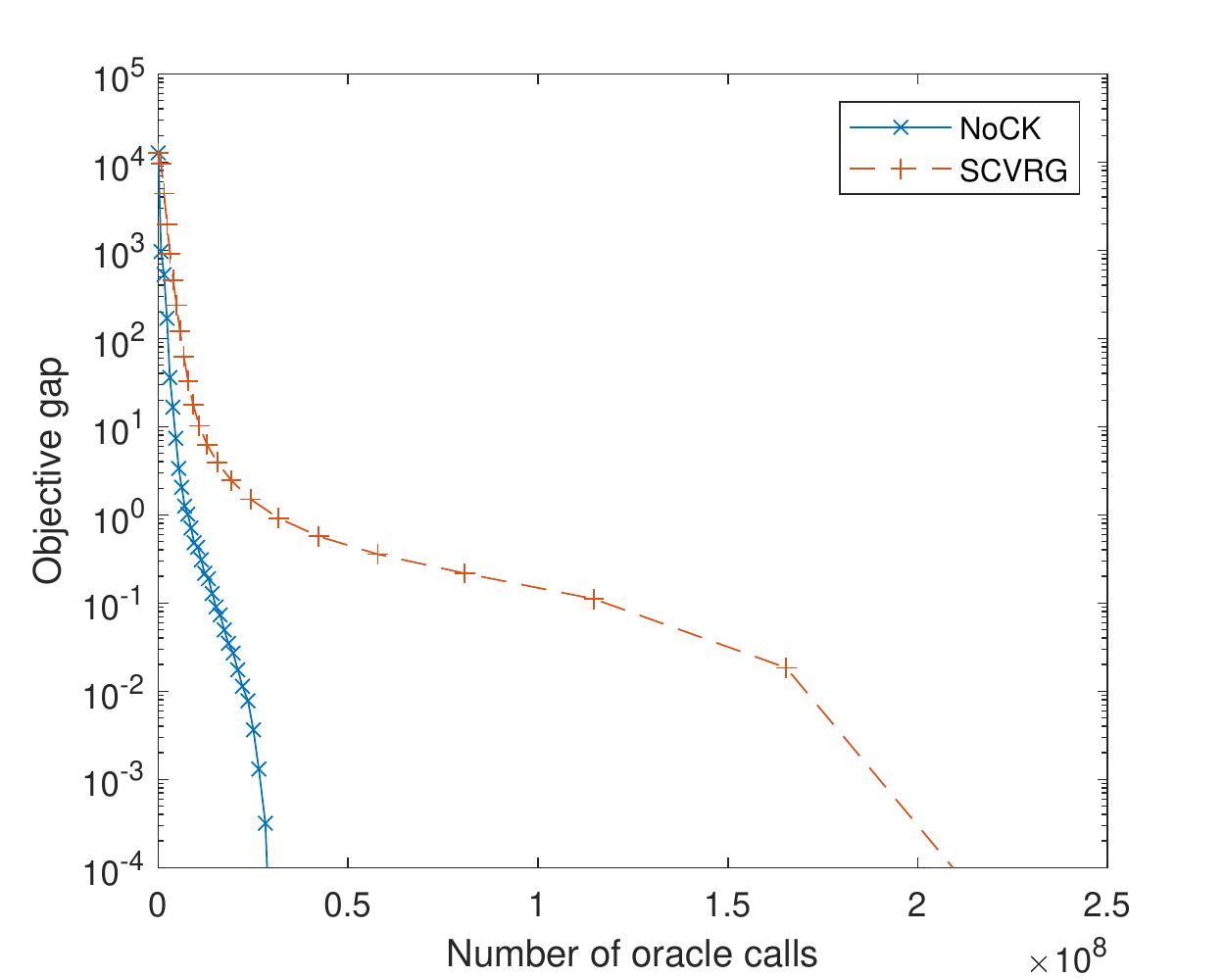}	&
		\includegraphics[width=.23\linewidth]{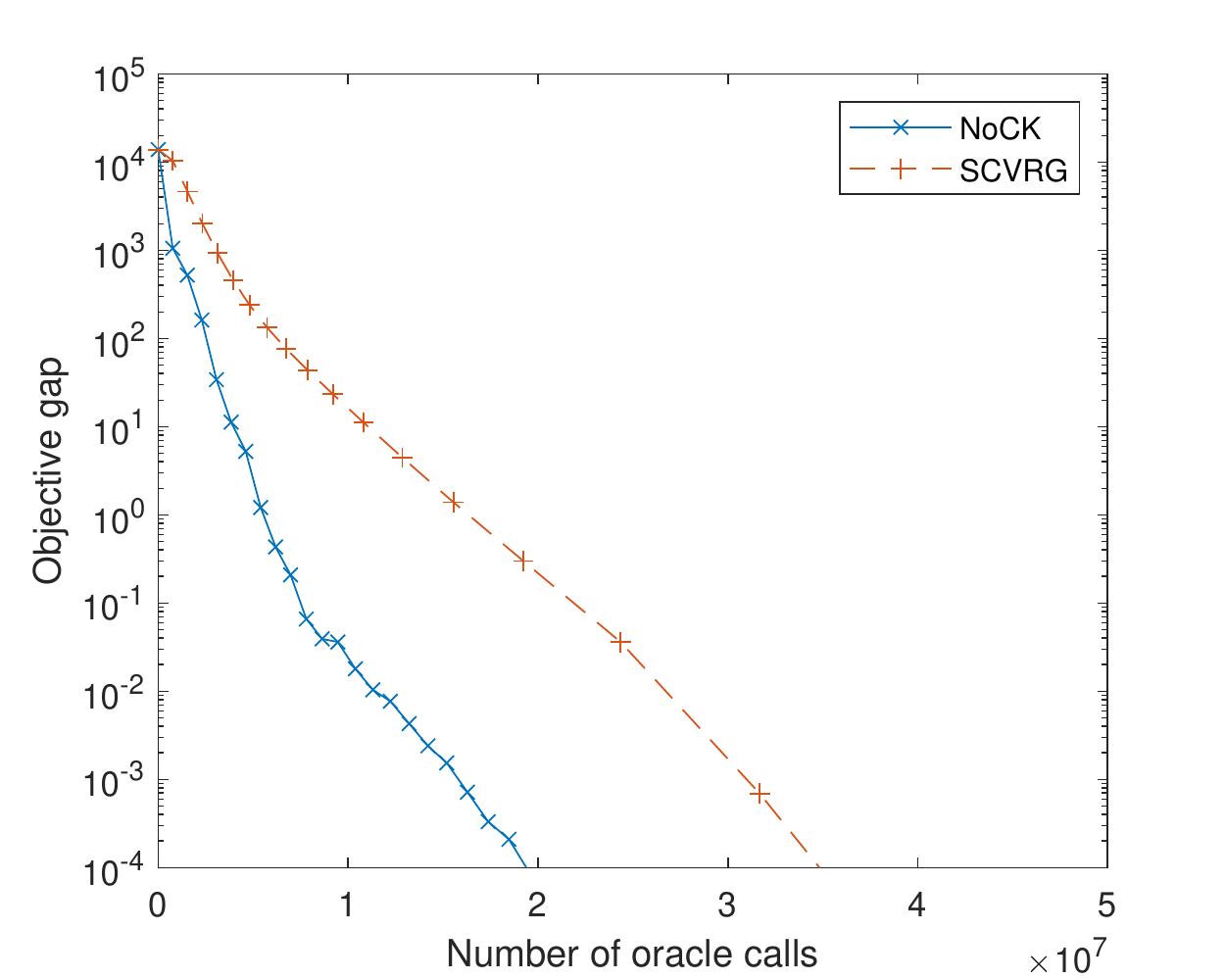}	\\
\includegraphics[width=.23\linewidth]{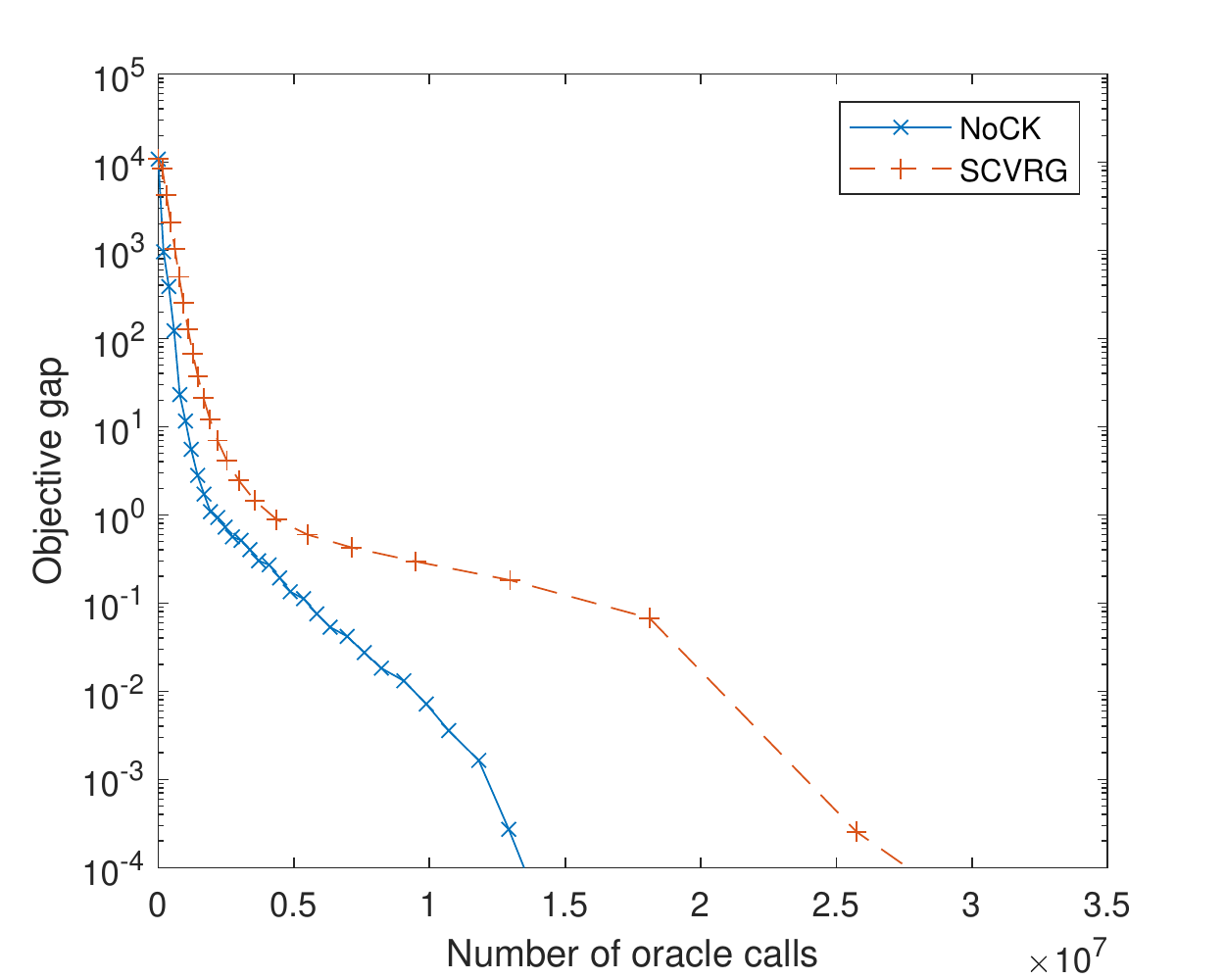} &
	\includegraphics[width=.23\linewidth]{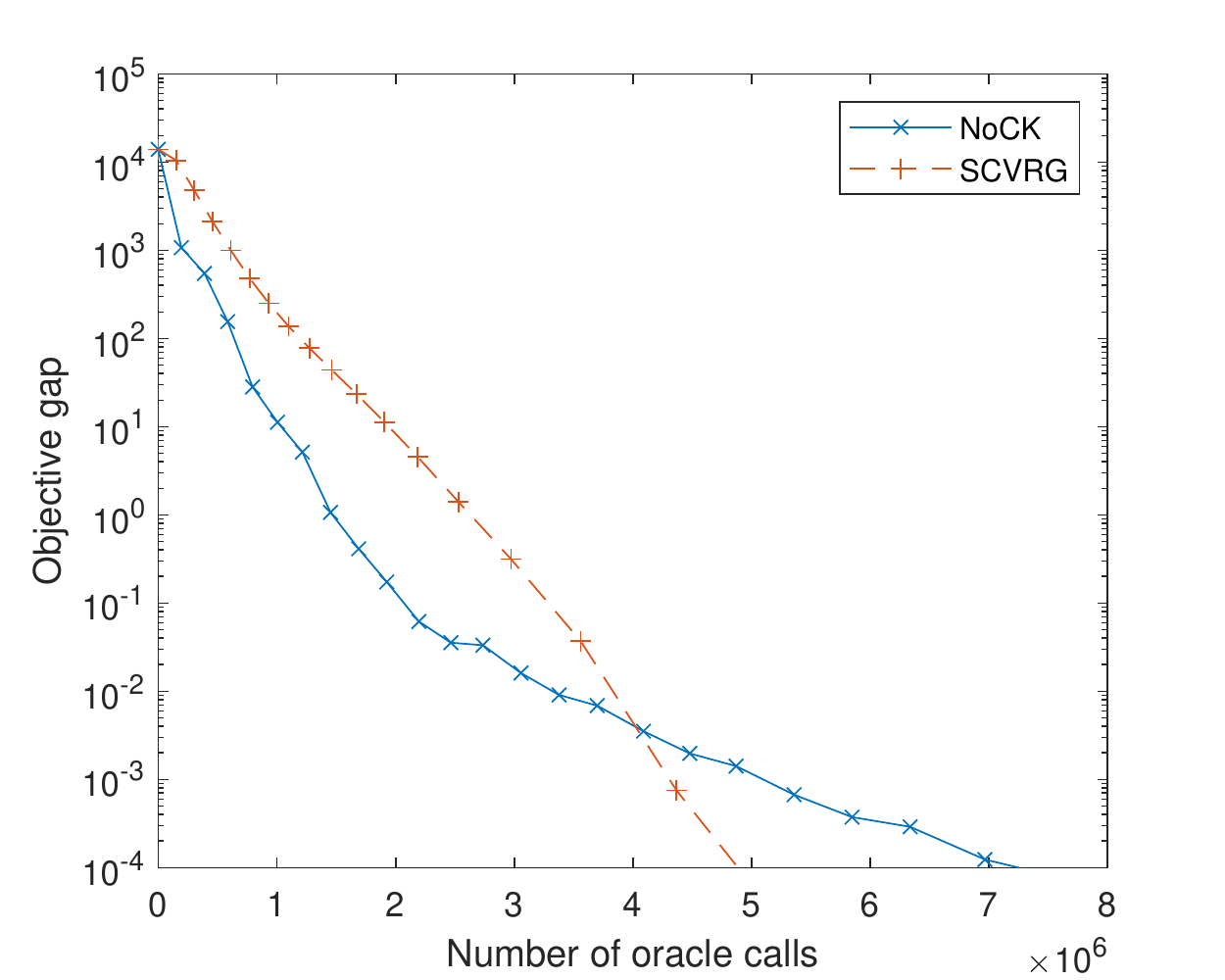}&
	\includegraphics[width=.23\linewidth]{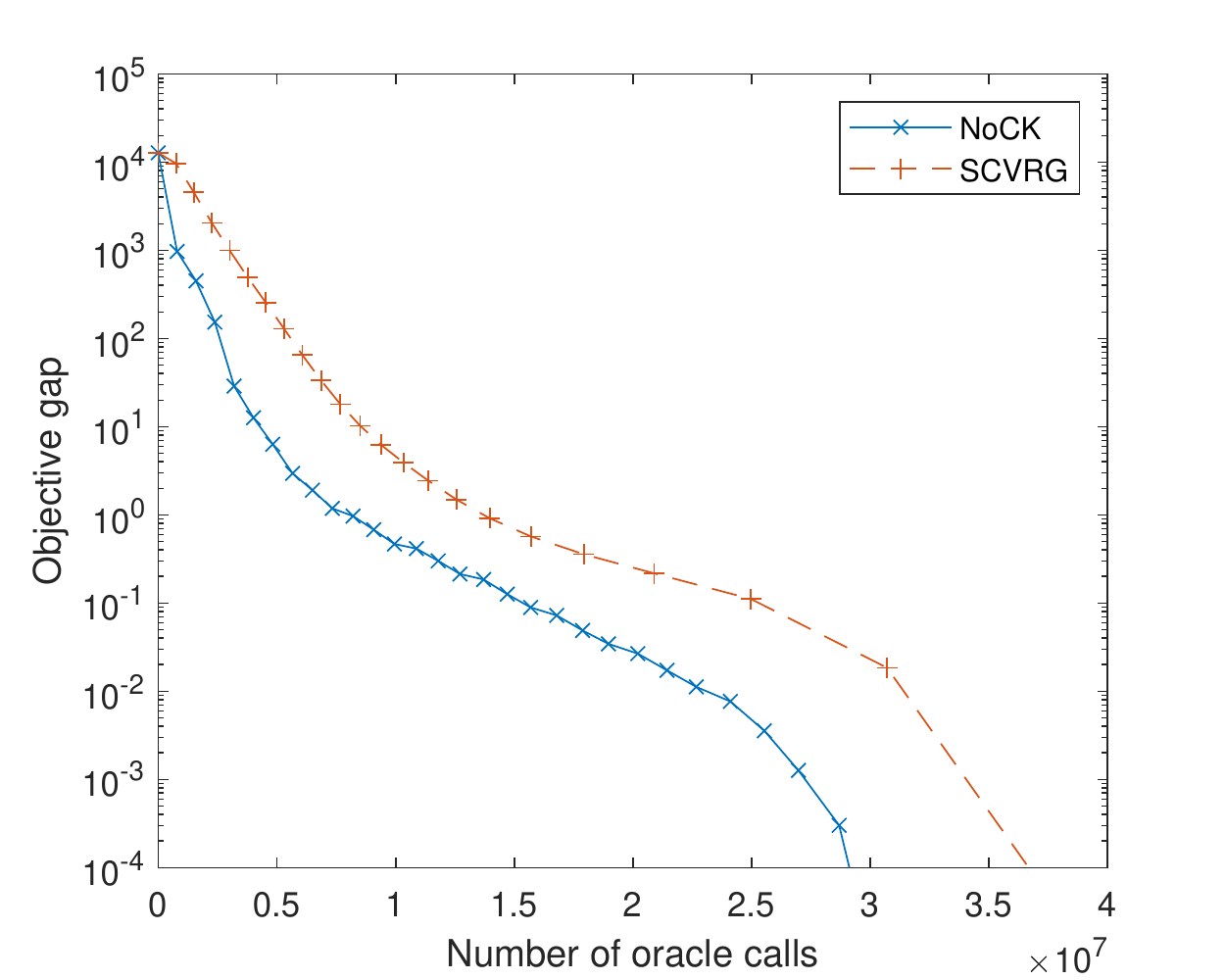} &
	\includegraphics[width=.23\linewidth]{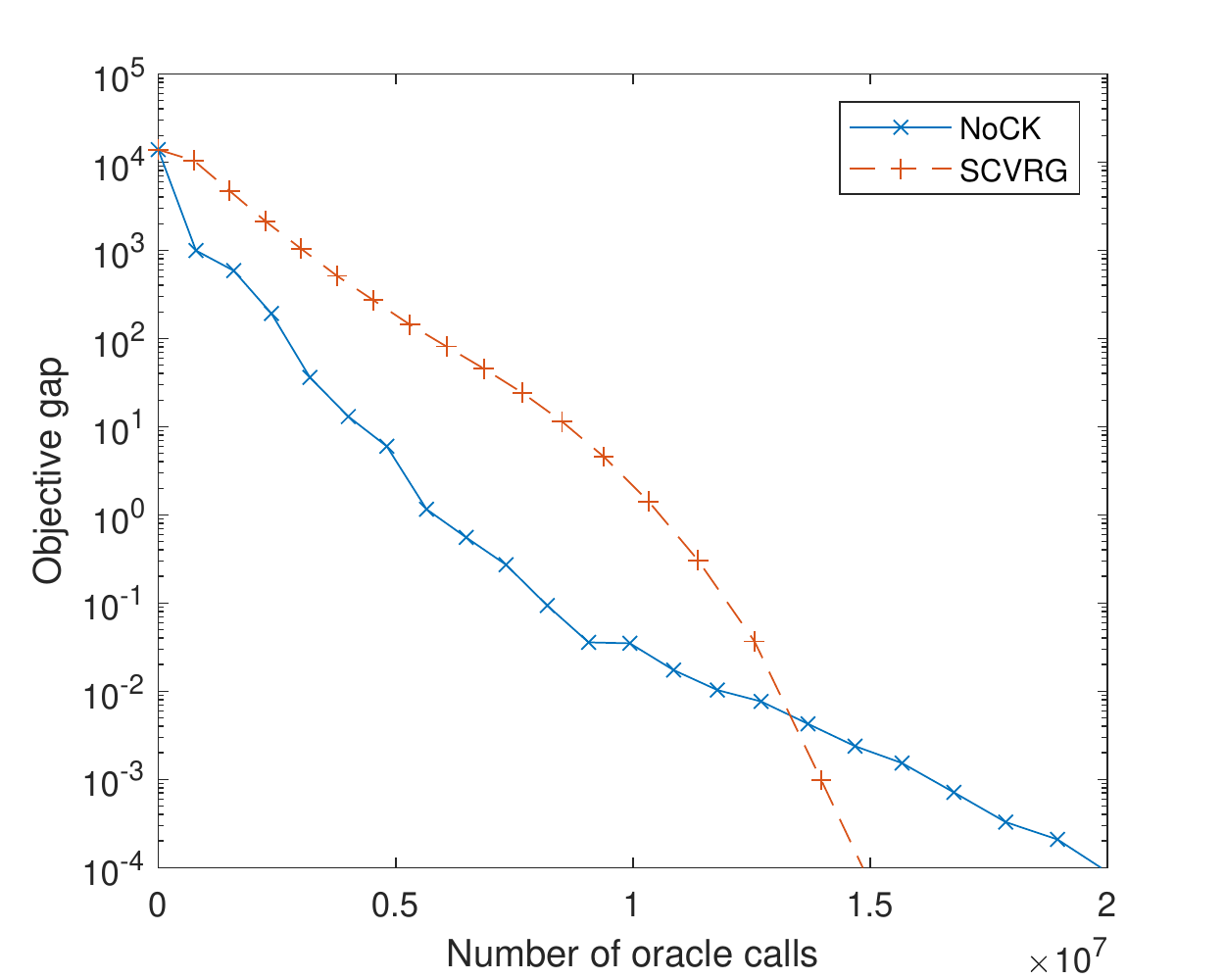}		
	\end{tabular}
\end{center}	
\end{figure}
} %for revise

\section{Conclusions}\label{sec:conclusion}
We have proposed an algorithm for solving the strongly convex case of the finite-sum compositional problem \eqref{genP}. To produce a stochastic $\varepsilon$-solution, the proposed algorithm generally needs $O\left((n_1+n_2 + \kappa^{2.5})\log\frac{1}{\varepsilon}\right)$ evaluations of component function/gradient/Jacobian, where $\kappa$ denotes the condition number. For convex cases of \eqref{genP}, we proposed an algorithm that approximately solves a sequence of strongly convex perturbed problems.  The complexity result is generally $O\left((n_1+n_2)\log\frac{1}{\varepsilon} + \varepsilon^{-2.5}\right)$.  Our complexity results are better than the existing best ones for general convex cases, and for strongly convex cases when $n_1+n_2>\kappa^{9/4}$ .

\appendix

\section{Proof of Lemma~\ref{usualrate}}
%\proof{Proof of Lemma~\ref{usualrate}.} %[Proof of Lemma~\ref{usualrate}]
The following chain of inequalities holds:
\begin{align*}
&\frac{(\tau_1+\tau_2-1+1/\theta-\overline{A}\tau_1)\theta}{\tau_2+\overline{B}\tau_1} 
= \frac{(1+\frac{\tau_2}{\tau_1}-\overline{A})\tau_1\theta  +1-\theta   }{(\frac{\tau_2}{\tau_1}+\overline{B})\tau_1} 
\stackrel{\text{\ding{172}}}{=} \frac{1+\frac{\tau_2}{\tau_1}-\overline{A}}{\frac{\tau_2}{\tau_1}+\overline{B}}\theta - \frac{1}{12m\tau_1(\frac{\tau_2}{\tau_1}+\overline{B})} \\
\stackrel{\text{\ding{173}}}{\geq}& 1+\frac{1-\overline{A}-\overline{B}}{\frac{\tau_2}{\tau_1}+\overline{B}}  - \frac{1}{6(\frac{\tau_2}{\tau_1}+\overline{B})}
= 1+ \frac{  1-\overline{A}-\overline{B} - \frac{1}{6}}{\frac{\tau_2}{\tau_1}+\overline{B}} 
\stackrel{\text{\ding{174}}}{\geq} 1+ \frac{  1-\overline{A}-\overline{B} - \frac{1}{6}}{1+\overline{B}}\geq 1+ \frac{  1-\frac{1}{4} - \frac{1}{6}}{1+\frac{1}{8}} > \frac{13}{12}.
\end{align*}
Here, \ding{172} holds because $\theta=1+\frac{1}{12m}$, \ding{173} follows from $\theta\geq 1$ and $2m\tau_1\geq 1$, and \ding{174} uses $\frac{\tau_2}{\tau_1}\leq 1.$
 
%\end{proof}

{
\section{Results for multi-level finite-sum COP}\label{sec:multilevel}
%}
In this section we extend SoCK and NoCK algorithms for a $p$-level finite-sum COP ($p\ge 2$) in the form of 
	\begin{equation}\label{multiP}
	\min_{x\in \mathbb{R}^{N_1}}  H(x)%\equiv f(x)+h(x)\equiv F(G(x))+h(x) 
	\equiv \left(f_{p} \circ f_{p-1} \circ\cdots\circ f_{1}\right) (x) +h(x),
\end{equation}
where $f_{i}=\frac{1}{n_i}\sum_{j=1}^{n_i}f_{i,j}$ is a finite-sum of $n_i$ differentiable maps $f_{i,j}\colon \mathbb{R}^{N_{i}}\rightarrow \mathbb{R}^{N_{i+1}}$ for all $i\in[p]$ and $j\in[n_{i}]$; and $h(x)$ is a convex regularizer that admits an easy proximal mapping. Clearly, only the last level $f_{p}$ must be a function with $N_{p+1}\equiv 1$. For consistency of notation, we write gradient information in terms of Jacobian matrices. As our results rely heavily on induction, we let $\phi^\prime$ denote the Jacobian matrix of map $\phi$, and define $\phi_{k}=f_{k} \circ f_{k-1} \circ\cdots\circ f_{1}$ for all $k\in[p]$. Then, by the chain rule, 
\begin{equation}\label{eq:phi-k-Jac}
 \phi_{k}^\prime(x)=f_{k}^\prime(\phi_{k-1}(x)) \cdot f_{k-1}^\prime(\phi_{k-2}(x)) \cdots  f_{1}^\prime(x)=f_{k}^\prime(\phi_{k-1}(x)) \cdot\phi_{k-1}^\prime(x).
 \end{equation}

The critical component of the convergence of SoCK as well as the main difference it has from the convergence of Katyusha is the biased Jacobian estimation $\widetilde{\nabla}$ at a query point $x$ given a snapshot point $\widetilde{x}$, and how the bias $\EE\|\widetilde{\nabla}-\phi_{p}^\prime(x)\|^2$ is bounded. Notice that the results in both Lemmas~\ref{lem:samplevariance} and \ref{lem:samplevariance3} are like 
\begin{equation}\label{gradest-ess}
	 \EE\|\widetilde{\nabla}-\phi_{p}^\prime(x)\|^2=O\left(\|x-\widetilde{x}\|^2\cdot\sum\frac{1}{\text{mini-batch size}}\right),
\end{equation}
and the success of SoCK is then from limiting this deviation from the Katyusha acceleration, so that all mini-batch sizes are set to $\Theta(\kappa^2)$. %{\color{red}(better to use $\Theta(\kappa^2)$?)}
 For a $p$-level COP in \eqref{multiP}, we can design a gradient estimator $\widetilde{\nabla}$ that achieves \eqref{gradest-ess} and then implement it in SoCK. By this, everything else will follow exactly as those in Theorems~\ref{thm:sock} and \ref{thm:gocknos}.

\begin{assump}\label{multi-one}
	~The function $\phi_{p}$ given in \eqref{multiP} is convex, and the function %and $L_f$-smooth, i.e. with an $L_f$-Lipschitz gradient. 
	$h$ in \eqref{multiP} is $\mu$-strongly convex with $\mu>0$.
\end{assump}

\begin{assump}\label{multiLip}
	For every $i\in [p]$ $j\in[n_i]$, $f_{i,j}\colon \mathbb{R}^{N_{i}}\rightarrow \mathbb{R}^{N_{i+1}}$ is $L_i$-smooth and has a $B_i$-bounded Jacobian matrix. Also, $\phi_{p}$ is $L$-smooth.%, and for every $j\in [n_2]$, $G_j$ is $L_G$-smooth and has a $B_G$-bounded Jacobian matrix.
\end{assump}
With Assumption~\ref{multiLip}, it is known that $\phi_{k}$ is $\ell_{k}$-smooth and has a $\gamma_{k}$-bounded Jacobian matrix, where $\gamma_{k}=\prod_{i=1}^{k}B_i$ and $\ell_{k}=\sum_{i=1}^{k}L_{i}(\prod_{j=1}^{i-1}B_j^2)(\prod_{j=i+1}^{k}B_j)$; cf. \cite{zhang2019multi}. To be consistent with the notation of SoCK, we still use $L$ as the smoothness parameter of $\phi_{p}$, and hence $L\le\ell_{p}$. In practice $\ell_{p}$ may not be well-defined, or may dominate a good estimate of $L$. In this section we will utilize the fact that $f_{k,j}\circ\phi_{k-1}$ is also $\ell_{k}$-smooth. 

For a $p$-level composition, we replace the snapshot computation and gradient estimation of SoCK as follows and obtain Algorithm~\ref{alg:gock-multi}. At a snapshot point $\widetilde{x}$, we compute
\begin{equation}\label{eq:snap-01}
	\phi_{1}(\widetilde{x}),\ldots,\,\phi_{k}(\widetilde{x}),\ldots,\,\phi_{p-1}(\widetilde{x})
\end{equation}
and
\begin{equation}\label{eq:snap-02}
	f_1^\prime(\widetilde{x}),\,f_2^\prime(\phi_{1}(\widetilde{x})),\ldots,\,f_{k+1}^\prime(\phi_{k}(\widetilde{x})),\ldots,\,f_{p}^\prime(\phi_{p-1}(\widetilde{x})).
\end{equation}
Hence, by \eqref{eq:phi-k-Jac}, we can also compute
\begin{equation}\label{eq:snap-03}
	\phi_{1}^\prime(\widetilde{x}),\ldots,\,\phi_{k}^\prime(\widetilde{x}),\ldots,\,\phi_{p}^\prime(\widetilde{x}).
\end{equation} 
For any query-snapshot pair $(x,\widetilde{x})$, compute the inner map and Jacobian estimations inductively
\begin{equation}\label{lvl-1-uv}
	u_1=\frac{1}{a_1}\sum_{j \in \A_1}\left(f_{1,j}(x)-f_{1,j}(\widetilde{x}) \right)+ \phi_{1}(\widetilde{x} ) \text{ and }
	v_1=\frac{1}{b_1}\sum_{j \in \B_1}\left(f_{1,j}^\prime(x)-f_{1,j}^\prime(\widetilde{x}) \right)+ \phi_{1}^\prime(\widetilde{x}),
\end{equation}
where  $\A_1$ and $\B_1$ are sampled uniformly at random from $[n_1]$ with replacement such that $|\A_1|=a_1$ and $|\B_1|=b_1$. For $k=2,\ldots, p$, we compute the inner map estimation
\begin{equation}\label{lvl-k-u}
	u_k=\frac{1}{a_k}\sum_{j \in \A_k}\left(f_{k,j}(u_{k-1})-f_{k,j}(\phi_{k-1}(\widetilde{x})) \right)+ \phi_{k}(\widetilde{x})	
\end{equation}
and the Jacobian estimation
\begin{equation}\label{lvl-k-v}
v_k=\frac{1}{b_k}\sum_{j \in \B_k}\left(f_{k,j}^\prime(u_{k-1})v_{k-1}-f_{k,j}^\prime(\phi_{k-1}(\widetilde{x})) \phi_{k-1}^\prime(\widetilde{x})\right)+ \phi_{k}^\prime(\widetilde{x}),
\end{equation}
where  $\A_k$ and $\B_k$ are sampled uniformly at random from $[n_k]$ with replacement such that $|\A_k|=a_k$ and $|\B_k|=b_k$. We return $\widetilde{\nabla}=v_p^\top$ as the gradient estimation.

\begin{algorithm}[H]                      
	\caption{ Multi-level SoCK}	          
	\label{alg:gock-multi}      
	{\small                     
	\begin{algorithmic} 	
		\State \textbf{Input:} $x_0\in \mathbb{R}^{N_1},$ %$L,$ $\mu,$  
		$S,$ $m\leq \frac{L}{2\mu},$ $\theta>1$, mini-batch sizes $\{a_i\}_{i=1}^{p-1}$ and $\{b_i\}_{i=1}^{p}$;
		\State Let $y_0=z_0=\widetilde{x}^0 \leftarrow x_0;$ %$\theta\leftarrow 1+\min\{\frac{\alpha\mu}{2},\frac{1}{12m}\};$           
		\For{$s=0$ to $S-1$}
		%\State $x_0^{s+1}=\widetilde{x}^s;$ 
		\State Compute \eqref{eq:snap-01}, \eqref{eq:snap-02}, and \eqref{eq:snap-03}; %$\phi_{p}^\prime(\widetilde{x}^s)$ by  $G(\widetilde{x}^s),$ $\nabla G(\widetilde{x}^s)$ and $\nabla f(\widetilde{x}^s)\leftarrow \left[\nabla G(\widetilde{x}^s)\right]^\top  \nabla F(G(\widetilde{x}^s));$ 
		\algorithmiccomment{take a snapshot}
		\For{$j=0$ to $m-1$}
		\State $k\leftarrow s m+j;$
		%		\State  Pick $i$ uniformly random from $\{1,\ldots,n\};$
		\State $x_{k+1}\leftarrow\tau_1 z_k+ \tau_2 \widetilde{x}^s +(1-\tau_1-\tau_2)y_k;$ \algorithmiccomment{linear coupling step}
		\State Compute $v_p(x_{k+1},\widetilde{x}^s)$ by \eqref{lvl-1-uv}, \eqref{lvl-k-u}, and \eqref{lvl-k-v};
		\State Update 
		$\textstyle\widetilde{\nabla}_{k+1}\leftarrow [v_p(x_{k+1},\widetilde{x}^s)]^\top $; 	\algorithmiccomment{biased gradient estimator}
		\State  Let $z_{k+1}\leftarrow\argmin_{z}\, \langle \widetilde{\nabla}_{k+1},z\rangle +\frac{1}{2\alpha}\|z-z_{k}\|^2+h(z);$ \algorithmiccomment{mirror descent step}
		\State  Let $y_{k+1}\leftarrow\argmin_{y}\, \langle \widetilde{\nabla}_{k+1},y\rangle +\frac{3L}{2}\|y-x_{k+1}\|^2+h(y);$ \algorithmiccomment{gradient descent step}
		\EndFor
		\State   $\widetilde{x}^{s+1}\leftarrow(\sum_{j=0}^{m-1} \theta^{j})^{-1} \cdot \sum_{j=0}^{m-1} \theta^{j}y_{s m+j+1};$ \algorithmiccomment{update to the snapshot point}
		\EndFor
		\State \Return $\widetilde{x}^S$.	
	\end{algorithmic}
}	
\end{algorithm}

Below, we establish \eqref{gradest-ess}. In order to do this, we bound $\EE\|\widetilde{\nabla}-\phi_{p}^\prime(x)\|^2=\EE\|v_p-\phi_{p}^\prime(x)\|^2$ by forming recurrence relations.
	 
\begin{lem}\label{lem:samplevariance-multi}
	Let the query point $x$ and the snapshot point $\widetilde{x}$ be given, and $\widetilde{\nabla}=v_p$ be computed by \eqref{lvl-1-uv}, \eqref{lvl-k-u}, and \eqref{lvl-k-v}. Then under Assumption~\ref{multiLip},
\begin{equation}\label{SVmulti}
		 \EE [ \|\widetilde{\nabla}-\phi_{p}^\prime(x)\|^2|x,\widetilde{x} ]
		 \le \left( \sum_{j=1}^{p} \frac{2}{b_j }\frac{4^{p-j}\ell_{j}^2\gamma_{p}^2}{\gamma_{j}^2}  +   \sum_{i=1}^{p-1}\frac{1}{a_i} \sum_{j=i+1}^{p} \frac{4^{p-j+1}\gamma_{p}^2  \gamma_{j-1}^4 L_{j}^2}{ \gamma_{j}^2 }   \right) \|x-\widetilde{x}\|^2. 
\end{equation}
\end{lem}	
\begin{proof}
	The first recurrence on $\EE\|v_k-\phi_{k}^\prime(x)\|^2$ is follows the proof of Lemma~\ref{lem:samplevariance3}. %{\color{red}(i.e., the follows?)}
\begin{align}
	\|v_k-\phi_{k}^\prime(x)\|^2 \le& 2 \left\|v_k- \frac{1}{b_k}\sum_{j \in \B_k}\left(f_{k,j}^\prime(\phi_{k-1}(x)) \phi_{k-1}^\prime(x)-f_{k,j}^\prime(\phi_{k-1}(\widetilde{x})) \phi_{k-1}^\prime(\widetilde{x})\right)- \phi_{k}^\prime(\widetilde{x})\right\|^2\nonumber\\
	&+2\left\| \frac{1}{b_k}\sum_{j \in \B_k}\left(f_{k,j}^\prime(\phi_{k-1}(x)) \phi_{k-1}^\prime(x)-f_{k,j}^\prime(\phi_{k-1}(\widetilde{x})) \phi_{k-1}^\prime(\widetilde{x})\right)+ \phi_{k}^\prime(\widetilde{x}) -\phi_{k}^\prime(x) \right\|^2.\label{recurrence01}
\end{align}
The first term of the r.h.s. of \eqref{recurrence01} can be simplified and bounded as follows,
\begin{align*}
	&\left\|v_k- \frac{1}{b_k}\sum_{j \in \B_k}\left(f_{k,j}^\prime(\phi_{k-1}(x)) \phi_{k-1}^\prime(x)-f_{k,j}^\prime(\phi_{k-1}(\widetilde{x})) \phi_{k-1}^\prime(\widetilde{x})\right)- \phi_{k}^\prime(\widetilde{x})\right\|^2\\
	= &  \frac{1}{b_k^2}\left\| \sum_{j \in \B_k}\left(f_{k,j}^\prime(u_{k-1})v_{k-1} -f_{k,j}^\prime(\phi_{k-1}(x)) \phi_{k-1}^\prime(x)\right) \right\|^2 \\
	\le &  \frac{1}{b_k}\sum_{j \in \B_k}\left\| f_{k,j}^\prime(u_{k-1})v_{k-1} -f_{k,j}^\prime(\phi_{k-1}(x)) \phi_{k-1}^\prime(x)  \right\|^2 \\
	\le &  \frac{1}{b_k}\sum_{j \in \B_k} 2\left\| f_{k,j}^\prime(u_{k-1})v_{k-1} -f_{k,j}^\prime(u_{k-1}) \phi_{k-1}^\prime(x)  \right\|^2 +2\left\| f_{k,j}^\prime(u_{k-1}) \phi_{k-1}^\prime(x) -f_{k,j}^\prime(\phi_{k-1}(x)) \phi_{k-1}^\prime(x)  \right\|^2\\
	\le &  2 B_k^2\left\| v_{k-1} - \phi_{k-1}^\prime(x)  \right\|^2 +2 \gamma_{k-1}^2 L_{k}^2\left\|  u_{k-1}   - \phi_{k-1}(x)     \right\|^2,
\end{align*}
where the equality follows from \eqref{lvl-k-v}, the first and second inequalities use Cauchy-Schwarz inequaity, and the last inequality follows from the boundedness of $f_{k,j}^\prime$ and $\phi_{k-1}^\prime$, and the $L_k$-smoothness of $f_{k,j}^\prime$. 
The second term of the r.h.s. of \eqref{recurrence01} can be bounded by the same process for the second term of the r.h.s. of \eqref{chain1}, by using the $\ell_{k}$-smoothness of each $f_{k,j}\circ\phi_{k-1}$, and it results in
{
	\[\expect*{\left\| \frac{1}{b_k}\sum_{j \in \B_k}\left(f_{k,j}^\prime(\phi_{k-1}(x)) \phi_{k-1}^\prime(x)-f_{k,j}^\prime(\phi_{k-1}(\widetilde{x})) \phi_{k-1}^\prime(\widetilde{x})\right)+ \phi_{k}^\prime(\widetilde{x}) -\phi_{k}^\prime(x) \right\|^2|x,\widetilde{x} }\le \frac{\ell_{k}^2}{b_k}\|x-\widetilde{x}\|^2.\]
}Utilizing the above two bounds for \eqref{recurrence01}, we have the first recurrence relation 
\begin{equation}\label{eq:rec-1}
	\EE [\|v_k-\phi_{k}^\prime(x)\|^2 |x,\widetilde{x} ]\le 4 B_k^2\EE [ \| v_{k-1} - \phi_{k-1}^\prime(x)   \|^2|x,\widetilde{x} ] +4 \gamma_{k-1}^2 L_{k}^2\EE [ \|  u_{k-1}   - \phi_{k-1}(x)      \|^2|x,\widetilde{x} ] + \frac{2\ell_{k}^2}{b_k}\|x-\widetilde{x}\|^2.
\end{equation}
We establish the recurrence for $\|  u_{k}   - \phi_{k}(x) \|^2$ as follows, %{\color{red}(some brackets need bigger)(added a macro for conditional expectation)}
\begin{align}
	&\EE [\|  u_{k}   - \phi_{k}(x) \|^2 |x,\widetilde{x},u_{k-1} ] \cr
%	= &  \frac{1}{a_k^2} \EE [ \| \sum_{j \in \A_k}\left(f_{k,j}(u_{k-1})-f_{k,j}(\phi_{k-1}(\widetilde{x}))+ \phi_{k}(\widetilde{x})- \phi_{k}(x) \right) \|^2 |x,\widetilde{x},u_{k-1} ]	 \cr
	= & { \frac{1}{a_k^2} \expect*{\left\| \sum_{j \in \A_k}\left(f_{k,j}(u_{k-1})-f_{k,j}(\phi_{k-1}(\widetilde{x}))+ \phi_{k}(\widetilde{x})- \phi_{k}(x) \right) \right\|^2 |x,\widetilde{x},u_{k-1}}	} \cr
%	= &  \frac{1}{a_k^2} \EE [ \| \sum_{j \in \A_k}\left(f_{k,j}(u_{k-1})-f_{k,j}(\phi_{k-1}(\widetilde{x})) -f_{k}(u_{k-1}) + \phi_{k}(\widetilde{x}) +f_{k}(u_{k-1}) - \phi_{k}(x) \right) \|^2 |x,\widetilde{x},u_{k-1} ]	 \cr 
	= & { \frac{1}{a_k^2} \expect*{\left\| \sum_{j \in \A_k}\left(f_{k,j}(u_{k-1})-f_{k,j}(\phi_{k-1}(\widetilde{x})) -f_{k}(u_{k-1}) + \phi_{k}(\widetilde{x}) +f_{k}(u_{k-1}) - \phi_{k}(x) \right) \right\|^2 |x,\widetilde{x},u_{k-1} }	} \cr 
	= &  \frac{1}{a_k^2}\sum_{j \in \A_k} \EE [ \| f_{k,j}(u_{k-1})-f_{k,j}(\phi_{k-1}(\widetilde{x})) -f_{k}(u_{k-1}) + \phi_{k}(\widetilde{x}) \|^2|x,\widetilde{x},u_{k-1} ] +\|f_{k}(u_{k-1}) - \phi_{k}(x)  \|^2 	 \cr
	\le & \frac{1}{a_k^2}\sum_{j \in \A_k} \EE [ \| f_{k,j}(u_{k-1})-f_{k,j}(\phi_{k-1}(\widetilde{x})) \|^2|x,\widetilde{x},u_{k-1} ] +\|f_{k}(u_{k-1}) - \phi_{k}(x)  \|^2\cr
	\le & \frac{B_k^2}{a_k}   \left\| u_{k-1} - \phi_{k-1}(\widetilde{x})  \right\|^2 +  B_{k}^2\left\|  u_{k-1}   - \phi_{k-1}(x)     \right\|^2, \label{eq:rec-2}
\end{align}
where the first equality follows from \eqref{lvl-k-u}, the third equality 
comes from the fact that $\big\{ f_{k,j}(u_{k-1})-f_{k,j}(\phi_{k-1}(\widetilde{x})) -f_{k}(u_{k-1}) + \phi_{k}(\widetilde{x}) \big\}$ are conditionally independent with each other, and their expectations all equal 0, the first inequality holds because the variance is bounded by the second moment, and the second inequality follows from the intermediate value theorem and the boundedness  of the Jacobian of each $f_{k,j}$ and $f_k$. 
A similar argument gives us the following recurrence for $\|  u_{k}   - \phi_{k}(\widetilde{x}) \|^2$:
\begin{align}
 \|  u_{k}   - \phi_{k}(\widetilde{x}) \|^2 %|x,\widetilde{x},u_{k-1} ] 
	= &  \frac{1}{a_k^2}   \left\| \sum_{j \in \A_k}\left(f_{k,j}(u_{k-1})-f_{k,j}(\phi_{k-1}(\widetilde{x})) \right) \right\|^2 %|x,\widetilde{x},u_{k-1} ]
		 \cr
	\le & \frac{1}{a_k}\sum_{j \in \A_k}  \| f_{k,j}(u_{k-1})-f_{k,j}(\phi_{k-1}(\widetilde{x})) \|^2%|x,\widetilde{x},u_{k-1} ]  %+\|f_{k}(u_{k-1}) - \phi_{k}(\widetilde{x})  \|^2
	\le     B_{k}^2\left\|  u_{k-1}   - \phi_{k-1}(x)     \right\|^2, \label{eq:rec-3}%\frac{B_k^2}{a_k}   \left\| u_{k-1} - \phi_{k-1}(\widetilde{x})  \right\|^2 +
\end{align}
where the first inequality uses Cauchy-Schwarz inequaity, and the second inequality follows from the intermediate value theorem and the $B_k$-boundedness of the Jacobian of each $f_{k,j}$. Also note that by the same argument as that of \eqref{eq:rec-3}, we have $\| u_{1} - \phi_{1}(\widetilde{x}) \|^2\le B_{1}^2 \| x - \widetilde{x} \|^2$, which together with \eqref{eq:rec-3} implies that for all $k\in[p]$,
\begin{equation}\label{eq:form-1}
 \|  u_{k}   - \phi_{k}(\widetilde{x}) \|^2\le \gamma_{k}^2 \| x - \widetilde{x} \|^2.
\end{equation}
Plugging \eqref{eq:form-1} into \eqref{eq:rec-2}, we obtain the recurrence
\[\EE [\|  u_{k}   - \phi_{k}(x) \|^2 |x,\widetilde{x}  ] \le B_k^2 \EE [\|  u_{k-1}   - \phi_{k-1}(x) \|^2 |x,\widetilde{x}  ] +  \frac{\gamma_{k}^2}{a_k} \| x - \widetilde{x} \|^2 ;\]
together with the fact that $ \EE [\|  u_{1}   - \phi_{1}(x) \|^2 |x,\widetilde{x} ] \le \frac{B_1^2}{a_1} \| x - \widetilde{x} \|^2$, which is similar to  \eqref{eq:lem-summandbound-ineq2}, we have 
\begin{equation}\label{eq:form-2}
	\EE[ \|  u_{k}   - \phi_{k}(x) \|^2 |x,\widetilde{x} ]\le \gamma_{k}^2\left(\sum_{i=1}^{k}\frac{1}{a_i}\right) \| x - \widetilde{x} \|^2.
\end{equation}
Plug \eqref{eq:form-2} into \eqref{eq:rec-1} to obtain the recurrence
\[
	\EE [\|v_k-\phi_{k}^\prime(x)\|^2 |x,\widetilde{x} ]\le 4 B_k^2\EE [ \| v_{k-1} - \phi_{k-1}^\prime(x)   \|^2|x,\widetilde{x} ] + \left(4 \gamma_{k-1}^4 L_{k}^2 \sum_{i=1}^{k-1}\frac{1}{a_i} +  \frac{2\ell_{k}^2}{b_k}\right)\|x-\widetilde{x}\|^2;
\] together with the fact that $ \EE [\|  v_{1}   - \phi_{1}^\prime(x) \|^2 |x,\widetilde{x} ] \le \frac{L_1^2}{b_1} \| x - \widetilde{x} \|^2$, which appeared right after \eqref{eq:lem-summandbound-ineq2} , we have
\begin{align*}
	\EE [\|v_k-\phi_{k}^\prime(x)\|^2 |x,\widetilde{x} ] \le&\|x-\widetilde{x}\|^2 4^k\gamma_{k}^2 \sum_{j=1}^{k} \frac{1}{4^j\gamma_{j}^2}\left(\frac{2\ell_{j}^2}{b_j}  + 4 \gamma_{j-1}^4 L_{j}^2 \sum_{i=1}^{j-1}\frac{1}{a_i}  \right) \cr
	=&  \left( \sum_{j=1}^{k} \frac{2}{b_j }\frac{4^{k-j}\ell_{j}^2\gamma_{k}^2}{\gamma_{j}^2}  +   \sum_{i=1}^{k-1}\frac{1}{a_i} \sum_{j=i+1}^{k} \frac{4^{k-j+1}\gamma_{k}^2  \gamma_{j-1}^4 L_{j}^2}{ \gamma_{j}^2 }   \right) \|x-\widetilde{x}\|^2. 
\end{align*}
Let $k=p$ in the above completes the proof. 
\end{proof}

Plugging \eqref{SVmulti} into \eqref{eq:ineq-for-all-case}, we are able to show \eqref{eq:cp2} and thus \eqref{eq:key} with an appropriate $M$. 
\begin{lem}%[coupling step 2]
	\label{lem:cp2multi}
	~Suppose $\tau_1\in (0, \frac{1}{3\alpha L}]$ and $\tau_2\in[0,1-\tau_1] $ in the linear coupling step of Algorithm \ref{alg:gock-multi}. Let $x^*$ be the solution of \eqref{multiP}. If $\widetilde{\nabla}_{k+1}$ is computed by  
	\eqref{lvl-1-uv}, \eqref{lvl-k-u}, and \eqref{lvl-k-v}, then \eqref{eq:cp2} holds with  
	\begin{equation}\label{eq:M-set-mult}
		M = 3\Big(\frac{1}{4\tau_1L}+\frac{3}{\mu}\Big)\left( \sum_{j=1}^{p} \frac{2}{b_j }\frac{4^{p-j}\ell_{j}^2\gamma_{p}^2}{\gamma_{j}^2}  +   \sum_{i=1}^{p-1}\frac{1}{a_i} \sum_{j=i+1}^{p} \frac{4^{p-j+1}\gamma_{p}^2  \gamma_{j-1}^4 L_{j}^2}{ \gamma_{j}^2 }   \right).%\left(\frac{4 B_G^4 L_F^2}{A} + \frac{4 B_F^2 L_G^2}{B}+ \frac{2 L^2}{C}\right).
	\end{equation} 
\end{lem}
\begin{proof} 
Taking conditional expectation $\EE_{k-1}$ on both sides of \eqref{eq:ineq-for-all-case} with $\beta=\frac{6}{\mu}$ and plugging \eqref{SVmulti}, we immediately have \eqref{eq:cp2} by using the choice of $M$ in \eqref{eq:M-set-mult}.
\end{proof} 

The convergence of Algorithm~\ref{alg:gock-multi} follows that of Algorithm~\ref{alg:gock}, based on Lemma~\ref{lem:cp2multi} and a similar parameter choice.

\begin{thm}[convergence result for multi-level SoCK%without Assumption~\ref{genfour}
	]\label{thm:gocknos-multi}
	~Under Assumptions~\ref{multi-one} %{\color{red}(repeat this assumption to problem (B.1))}
	 and \ref{multiLip}, let $\{\widetilde x^s\}$ be generated from Algorithm~\ref{alg:gock-multi} with $\widetilde\nabla_{k+1}$ computed by \eqref{lvl-1-uv}, \eqref{lvl-k-u}, \eqref{lvl-k-v} and with parameters set as follows:
	\begin{subequations}\label{eq:para-multi}
		\begin{align}
			m\leftarrow\left\lceil \frac{1}{2}\sqrt{\frac{L}{\mu}}\right\rceil,\, \tau_1\leftarrow  \frac{1}{2m},\, \tau_2\leftarrow \frac{1}{2m},\, \theta\leftarrow 1+\frac{1}{12m},\, \alpha\leftarrow \frac{1}{3\tau_1 L},\\
			a_i\leftarrow \frac{180(2p-1)}{\mu^2} \sum_{j=i+1}^{p} \frac{4^{p-j+1}\gamma_{p}^2  \gamma_{j-1}^4 L_{j}^2}{ \gamma_{j}^2 } \quad \text{ for } i=1,\ldots,p-1,\label{eq:para-multi-a}\\
						b_i\leftarrow  \frac{360(2p-1)}{\mu^2} \frac{4^{p-i}\ell_{i}^2\gamma_{p}^2}{\gamma_{i}^2}  \quad \text{ for } i=1,\ldots,p.\label{eq:para-multi-b}
		\end{align}
	\end{subequations}
	Then
	\begin{equation*}
		\EX\big[H(\widetilde{x}^{S})-H(x^{\ast})\big] \leq %\begin{cases}
		%O\left( (1+\sqrt{\frac{\mu}{24m L_f}})^{-Sm}\right)\sqrt{\frac{L_f}{m\mu}}(H(x_0)-H(x^{\ast})) &\text{if $m\leq \frac{2L_f}{3\mu};$}\\
		11\cdot \left(  \frac{12}{13}\right)^{S}\big(H(x_0)-H(x^{\ast})\big). %&\text{if  $m> \frac{2L_f}{3\mu}.$}
		%\end{cases}
	\end{equation*}
\end{thm}

\begin{proof} 
Notice that the parameters given in \eqref{eq:para-multi} are the same as those in \eqref{eq:para-1} except for $\{a_i\}_{i=1}^{p-1}$ and $\{b_i\}_{i=1}^{p}$, and also notice that the choices of $\{a_i\}_{i=1}^{p-1}$ and $\{b_i\}_{i=1}^{p}$ only affect the value of $M$. Plugging into \eqref{eq:M-set-mult} the values of $\{a_i\}_{i=1}^{p-1}$ and $\{b_i\}_{i=1}^{p}$ given in \eqref{eq:para-multi-a} and \eqref{eq:para-multi-b}, we can easily verify that $M\le \frac{\mu}{16}$. Now following the same arguments in the proof of Theorem~\ref{thm:sock}, we obtain the desired result.
\end{proof}

By Theorem~\ref{thm:gocknos-multi}, we can estimate the complexity of  Algorithm~\ref{alg:gock-multi} in terms of the number of evaluations on $ f_{i,j}$, $f_{i,j}^\prime$, and the proximal mapping of $h$. Its proof follows that of Corollary~\ref{cor:sock}, and we omit it.

\begin{cor}\label{cor:gocknos-multi}
	~Given $\varepsilon>0$, under the same assumptions as those in Theorem~\ref{thm:gocknos-multi}, the number of component map/Jacobian evaluations %{\color{red}(number of gradient/Jacobi evaluation?)(updated)}
	 of Algorithm~\ref{alg:gock-multi} to produce a stochastic $\varepsilon$-solution is \[O\left( \Big(\sum_{i=1}^{p}n_i   +\sqrt{\frac{L}{\mu}}\frac{ \ell_{p}^2 }{\mu^2}\Big) \log\frac{H(x_0)-H(x^{\ast})}{\varepsilon}\right).\] 
\end{cor}
The extension to a $p$-level finite-sum COP \eqref{multiP} that is  non-strongly convex can be done easily, following the argument in Section~\ref{sec:nonsc}, hence we omit the details.

\section{More numerical tests}\label{app:test}
%{\color{red}Edit these figures as what I did in Figures 1 and 2(updated). It is strange that for all trials $\kappa$ is the same (I fixed the random seed for each model so did $\kappa$ in trials). If the triple $(n,\kappa,v)$ if for all trials, just state it in the figure caption.}
In this section, we provide more numerical results, which can show the stable performance of the proposed methods.  
In Figures~\ref{fig:sock_comp_k_large} and \ref{fig:sock_comp_k_small}, we compare Algorithm~\ref{alg:gock} with both options to VRSC-PG in \cite{huo2018accelerated}, C-SAGA and GC-SAGA in \cite{zhang2019composite} on four independent trials with different $n$ and $\kappa$. % for both options of SoCK under different $n$ and $\kappa$.  
In Figures~\ref{fig:nock_comp_k_large}, \ref{fig:nock_comp_k_small}, \ref{fig:nock_comp_k_large2} and \ref{fig:nock_comp_k_small2}, we compare  Algorithm~\ref{alg:gock-cvx} to SCVRG in \cite{lin2018improved} on four independent trials with different $n$ but similar smoothness constant $L$. %but having different $v$ and (unknown) strong-convexity parameter, i.e., different curvature on the objective.

\begin{figure}[h]
	\caption{Comparison of the proposed SOCK (i.e., Algorithm~\ref{alg:gock} with Option I) and GOCK (i.e., Algorithm~\ref{alg:gock} with Option II) to VRSC-PG in \cite{huo2018accelerated}, C-SAGA and GC-SAGA in \cite{zhang2019composite} on solving instances of \eqref{eq:exp1} having $\kappa>(2n)^{9/4}$, with the parameter tuple $(n,\kappa,v)=(5e3, 96, 30)$ in the top row and $(n,\kappa,v)=(5e4, 211, 10)$ in the bottom row. All methods take the explicitly-computed strong convexity constant $\mu$. The value of $v$ is varied to change the condition number $\kappa$.}
	\label{fig:sock_comp_k_large}
	\begin{center}
		\begin{tabular}{cccc}
			{\small trial $1$} & {\small trial $2$} & {\small trial $3$} & {\small trial $4$}\\	
			\includegraphics[width=.23\linewidth]{pics/Case_N=500_by_n=5000_with_v=30_T1.pdf}		&
			\includegraphics[width=.23\linewidth]{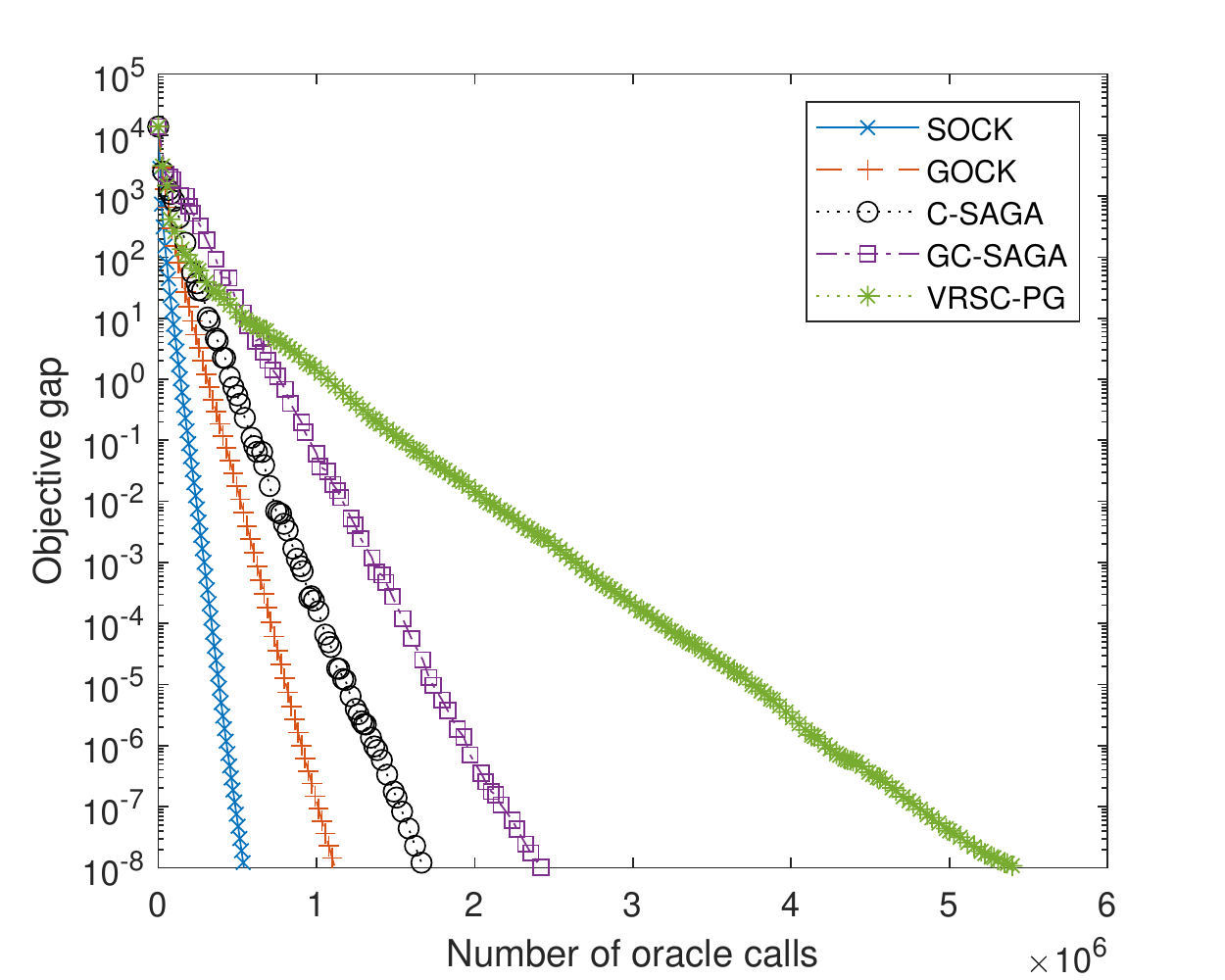}	&
			\includegraphics[width=.23\linewidth]{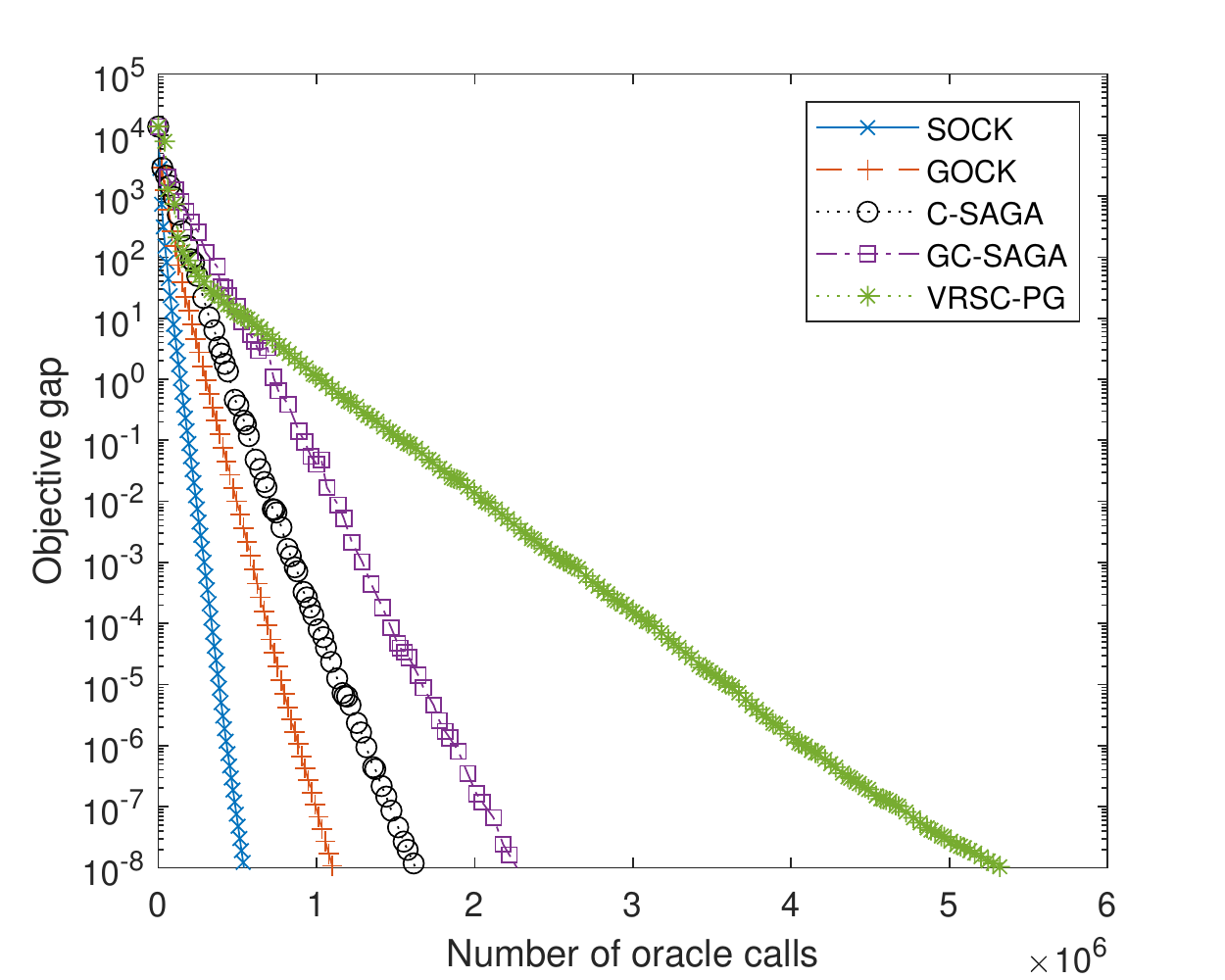}	&
			\includegraphics[width=.23\linewidth]{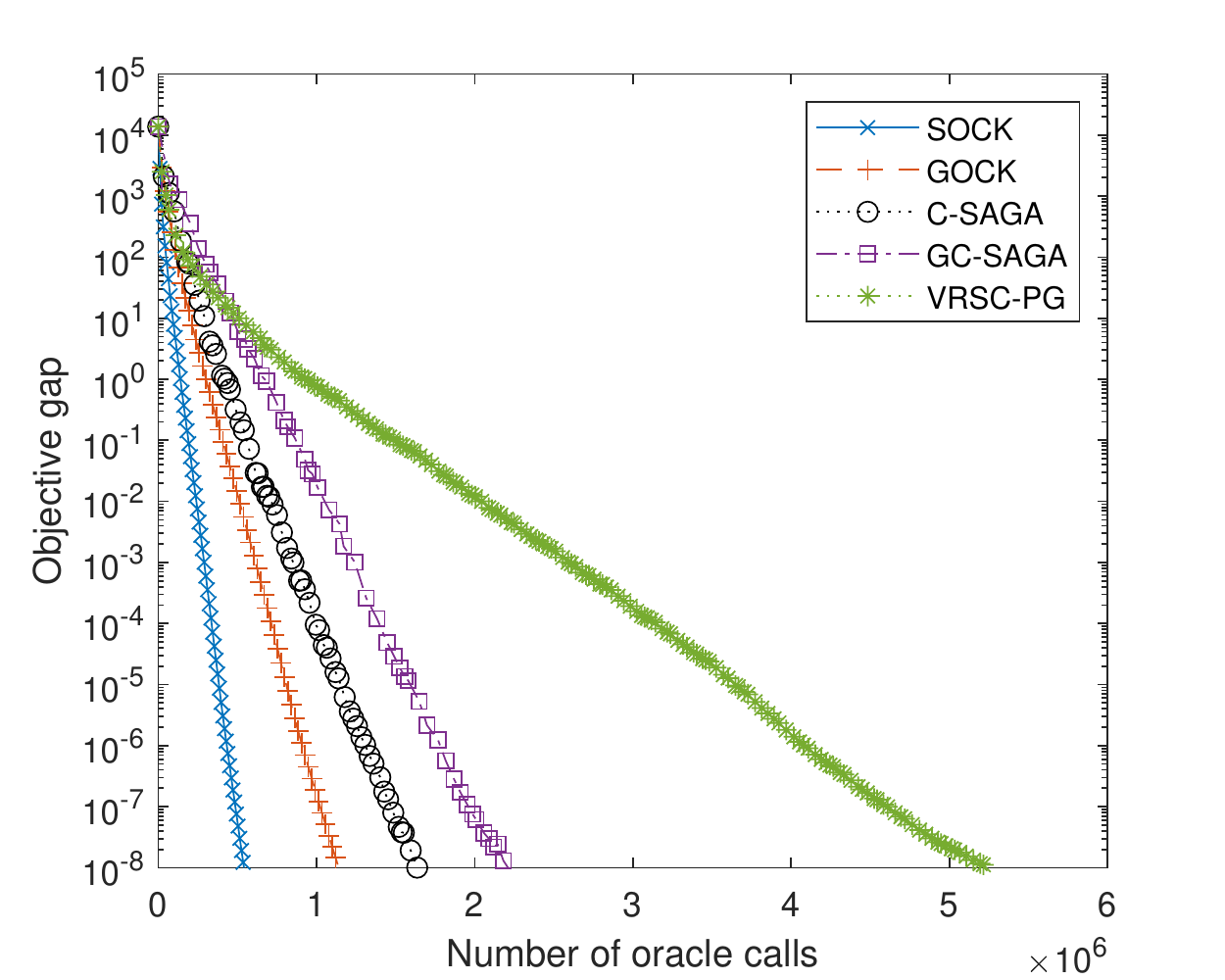}	\\
			\includegraphics[width=.23\linewidth]{pics/Case_N=500_by_n=50000_with_v=10_T1.pdf} &
			\includegraphics[width=.23\linewidth]{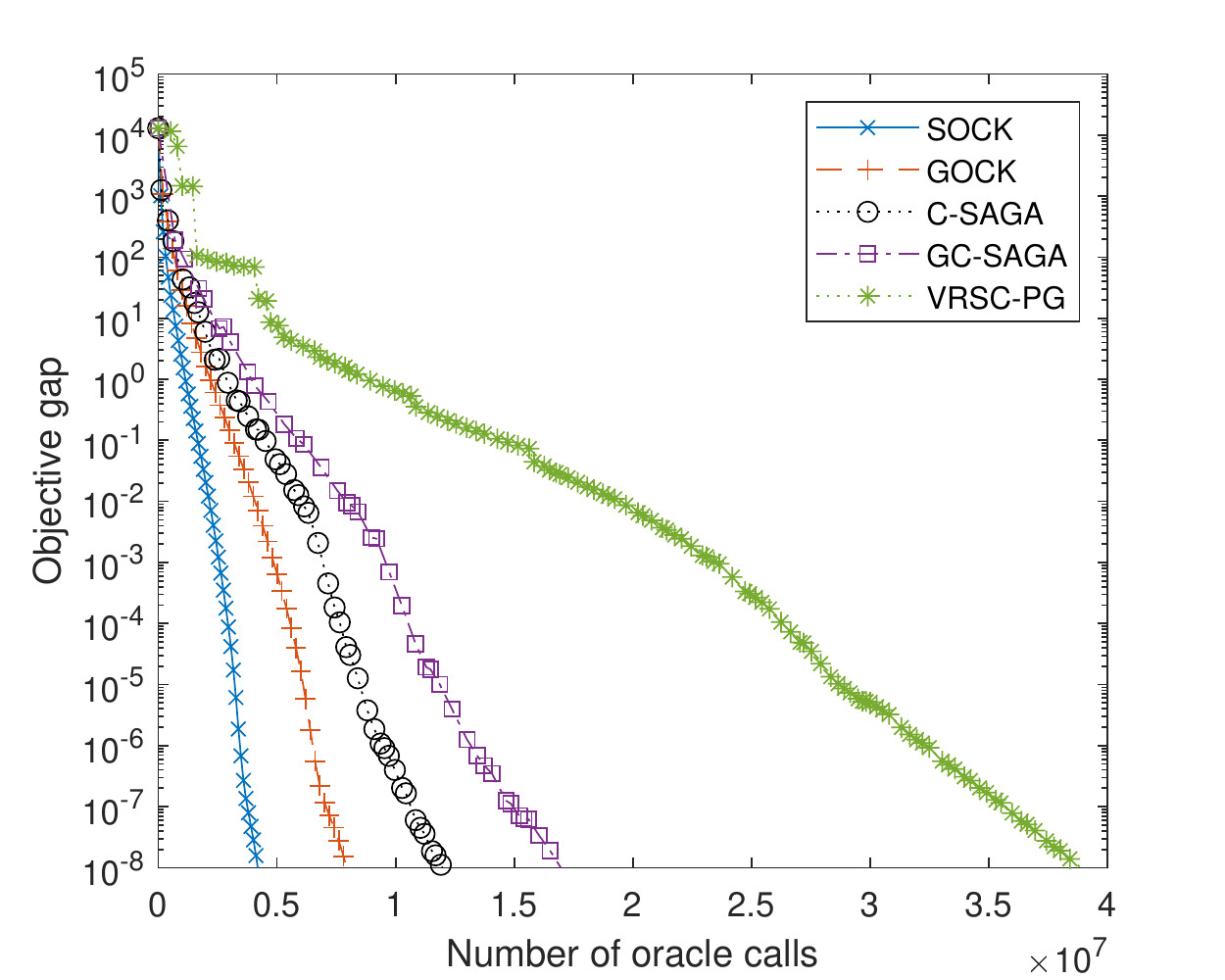}&
			\includegraphics[width=.23\linewidth]{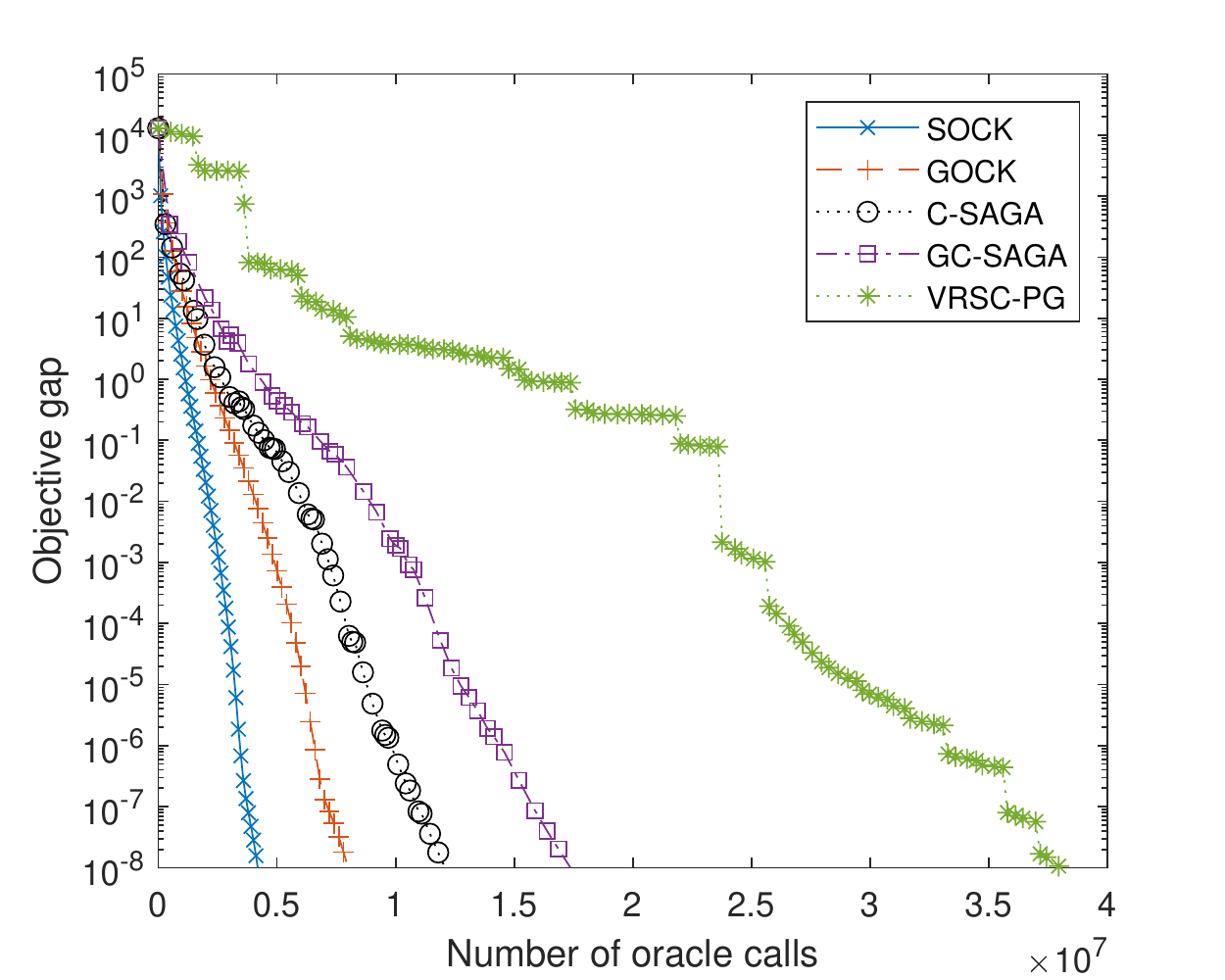} &
			\includegraphics[width=.23\linewidth]{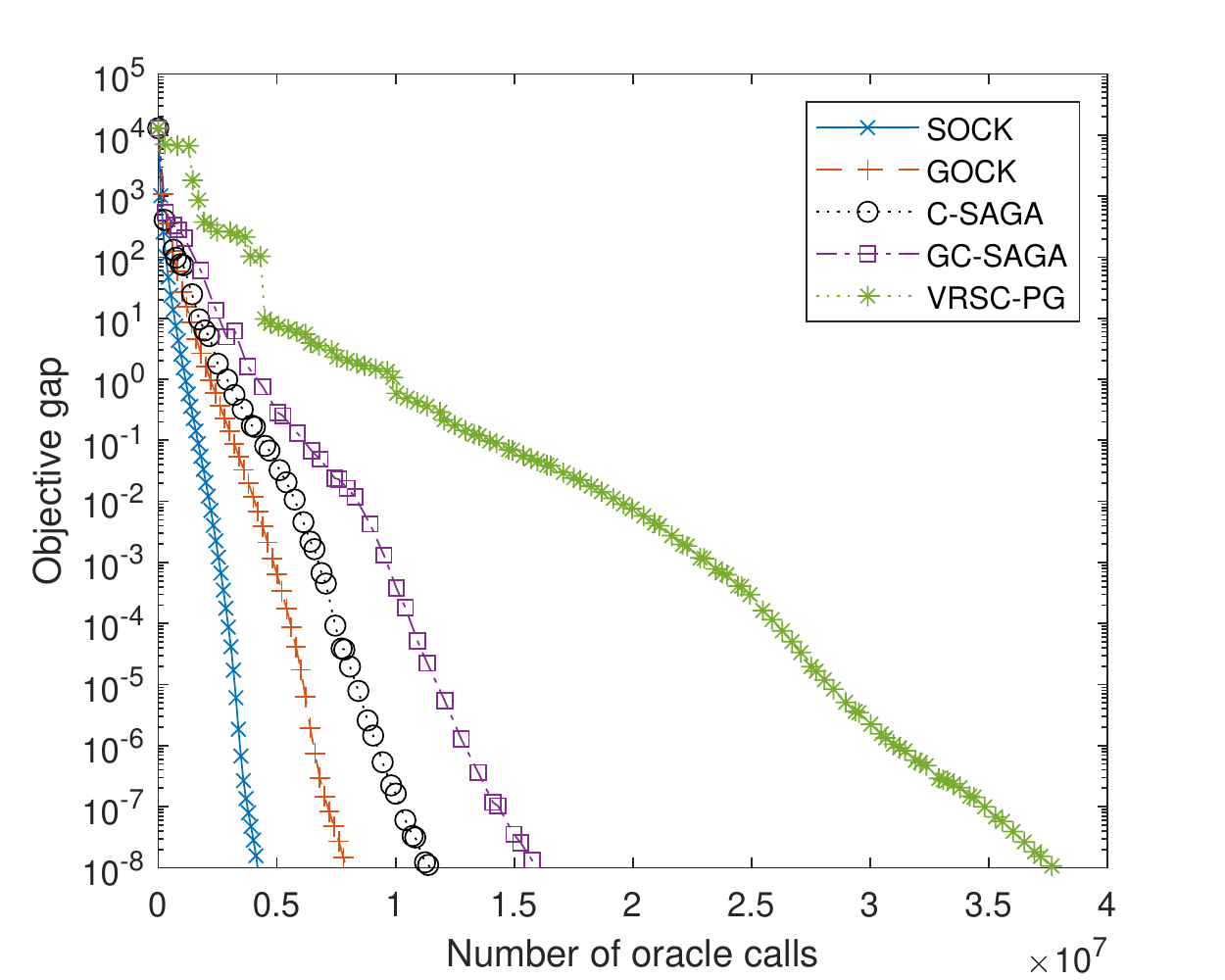}		
		\end{tabular}
	\end{center}	
\end{figure}

\begin{figure}[h]
	\caption{Comparison of the proposed SOCK (i.e., Algorithm~\ref{alg:gock} with Option I) and GOCK (i.e., Algorithm~\ref{alg:gock} with Option II) to VRSC-PG in \cite{huo2018accelerated}, C-SAGA and GC-SAGA in \cite{zhang2019composite} on solving instances of \eqref{eq:exp1} having $(2n)^{1/3}<\kappa<(2n)^{4/9}$, with the parameter tuple $(n,\kappa,v)=(5e3, 52, 60)$ in the top row and $(n,\kappa,v)=(5e4, 93, 30)$ in the bottom row. All methods take the explicitly-computed strong convexity constant $\mu$. The value of $v$ is varied to change the condition number $\kappa$.}
	\label{fig:sock_comp_k_small}
	\begin{center}
		\begin{tabular}{cccc}
						{\small trial $1$} & {\small trial $2$} & {\small trial $3$} & {\small trial $4$}\\	
			\includegraphics[width=.23\linewidth]{pics/Case_N=500_by_n=5000_with_v=60_T1.pdf}		&
			\includegraphics[width=.23\linewidth]{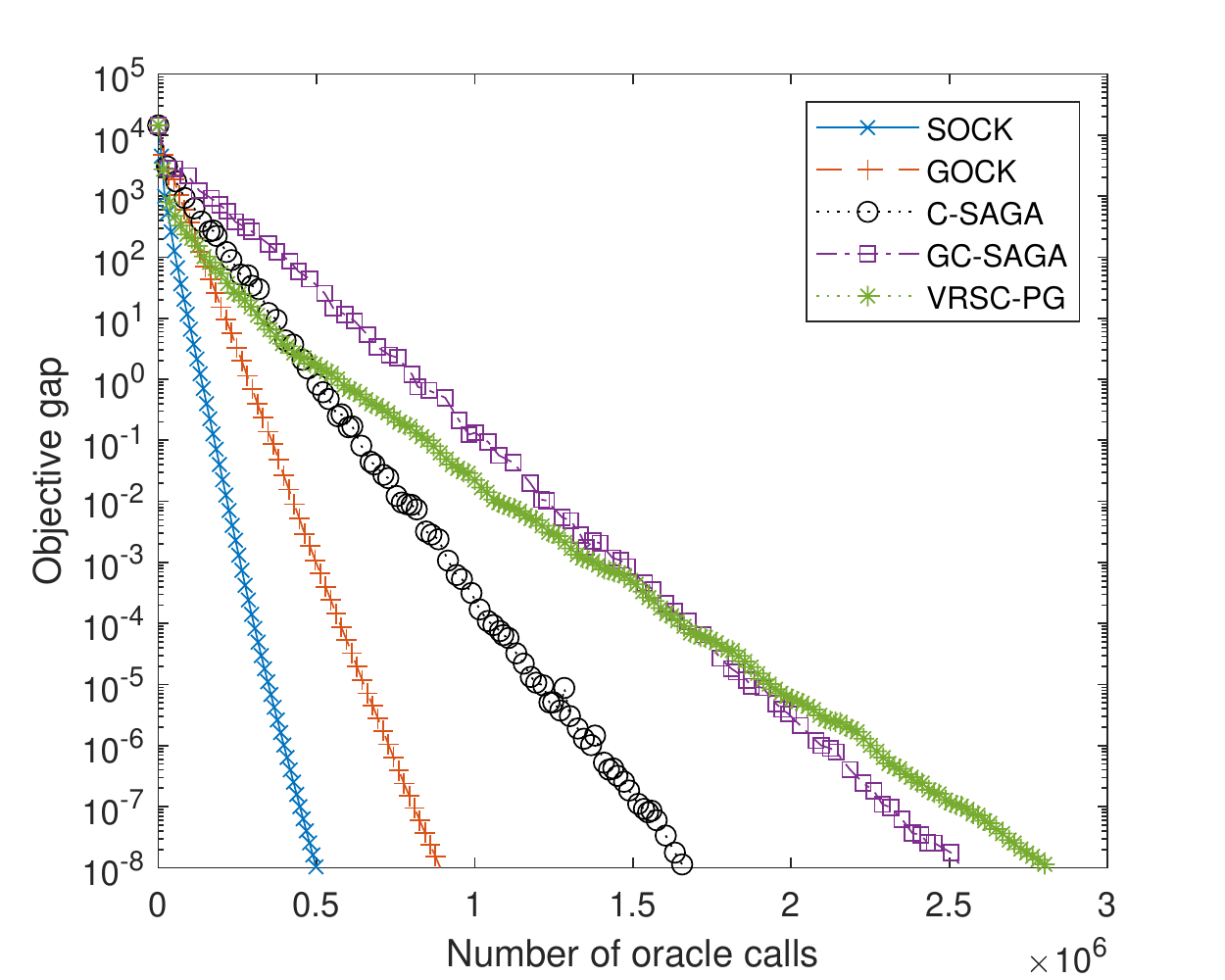}	&
			\includegraphics[width=.23\linewidth]{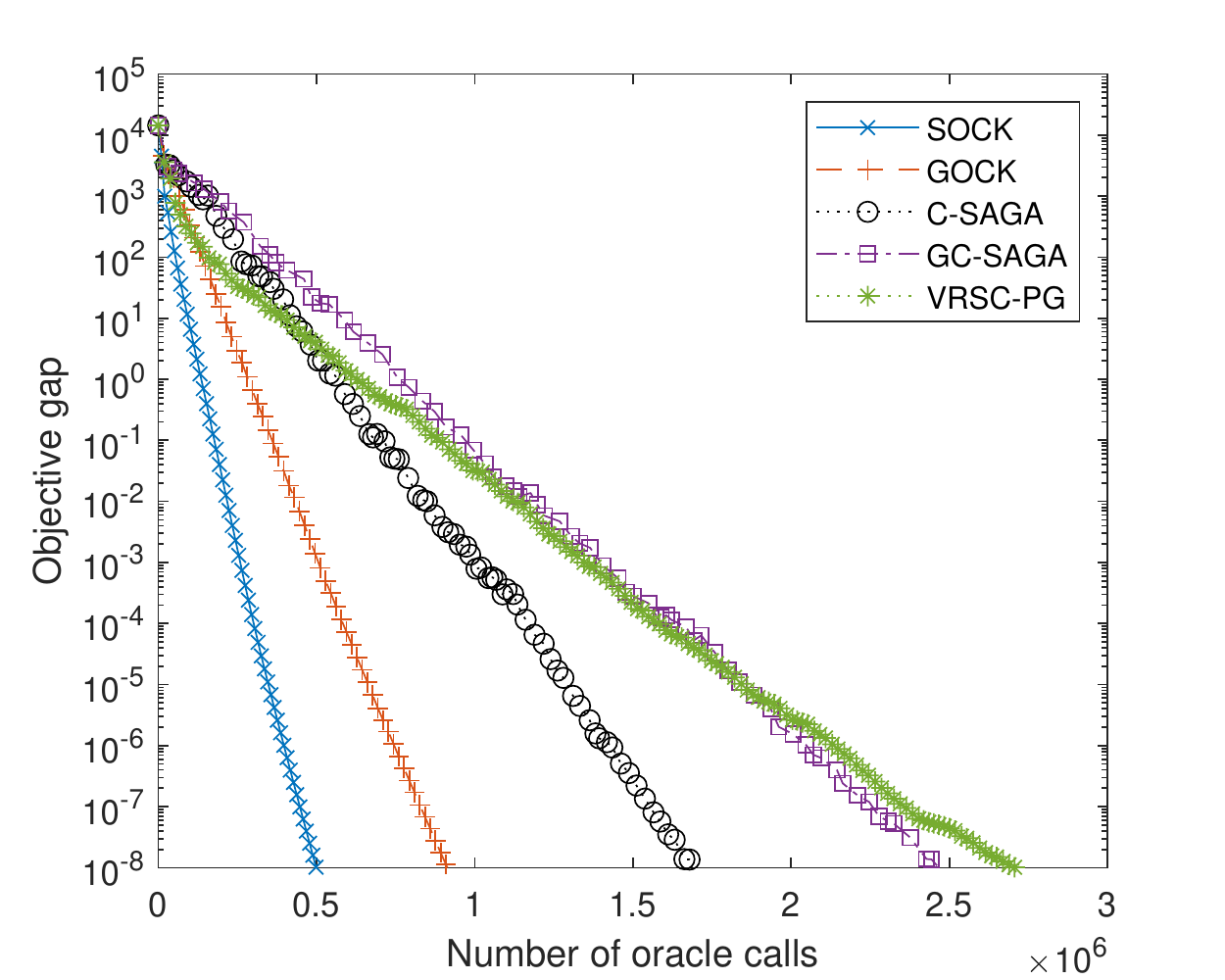}	&
			\includegraphics[width=.23\linewidth]{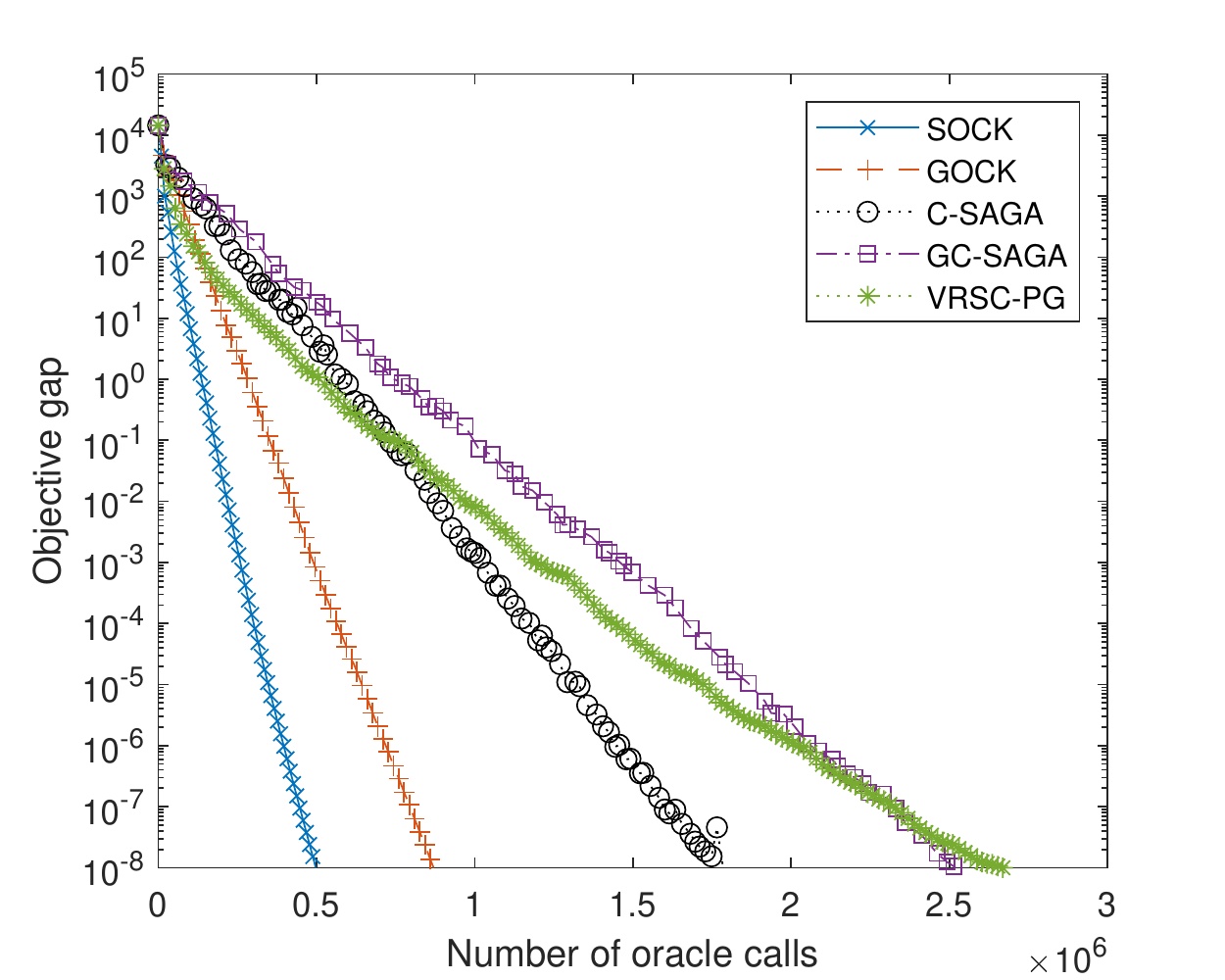}	\\
			\includegraphics[width=.23\linewidth]{pics/Case_N=500_by_n=50000_with_v=30_T1.pdf} &
			\includegraphics[width=.23\linewidth]{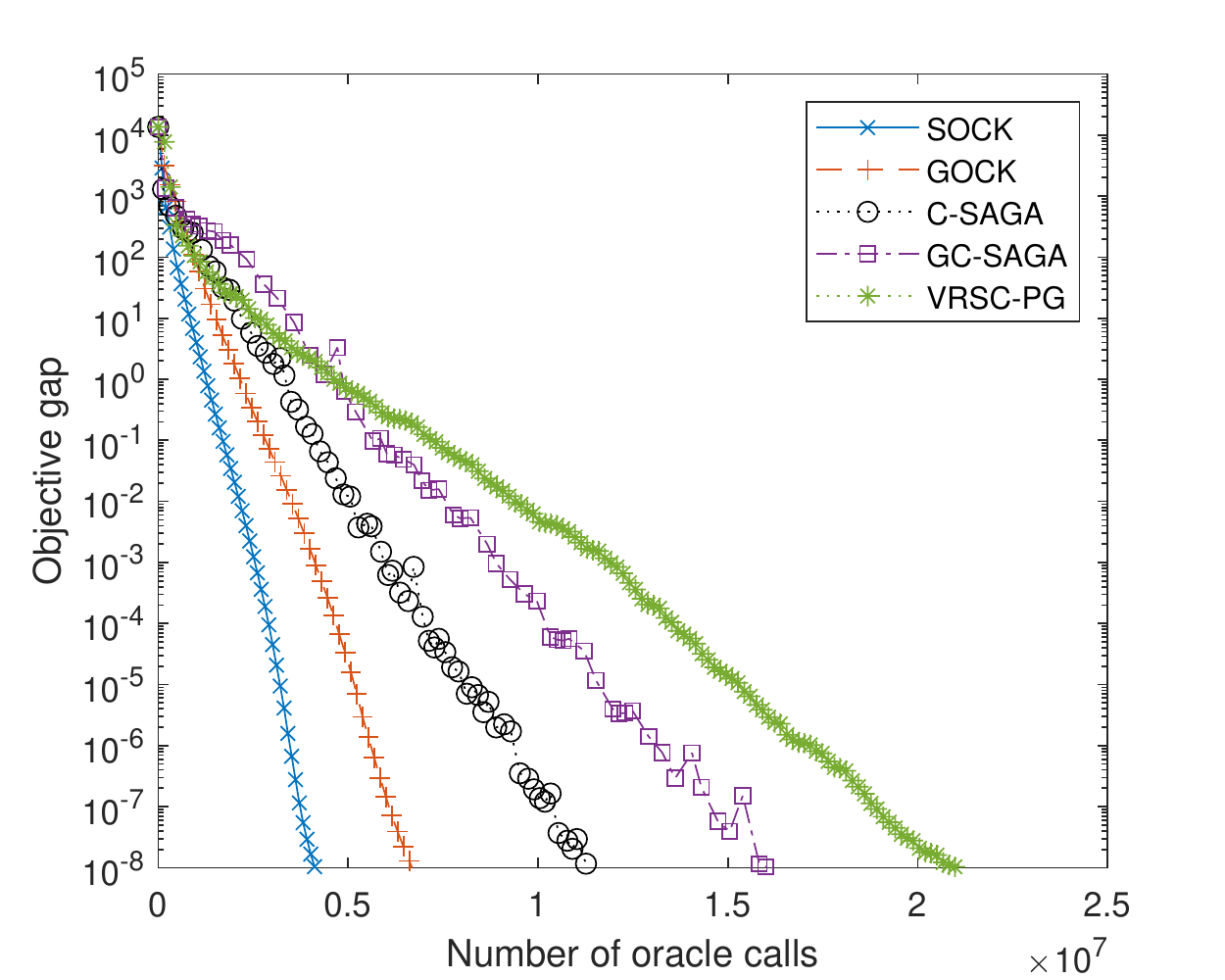}&
			\includegraphics[width=.23\linewidth]{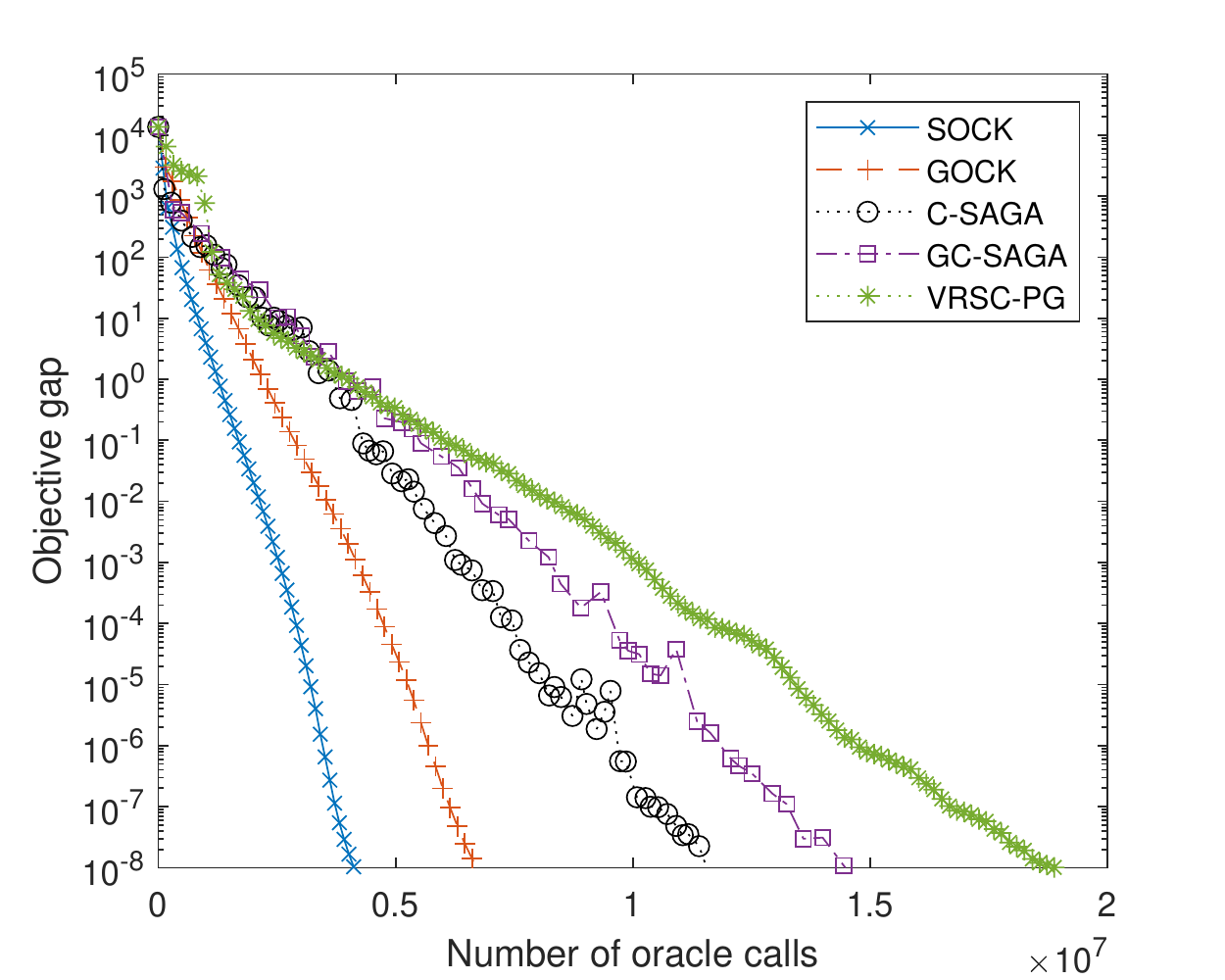} &
			\includegraphics[width=.23\linewidth]{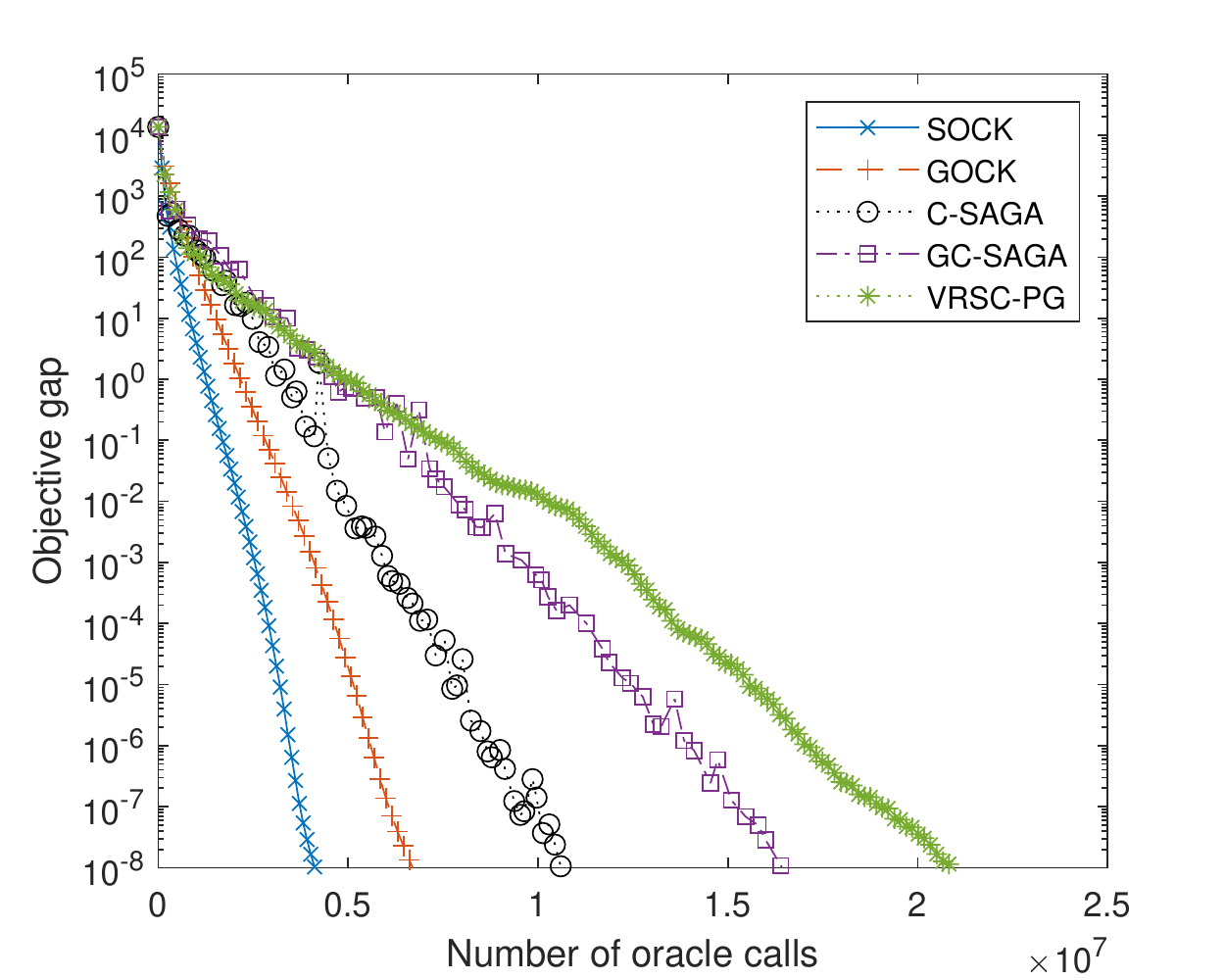}		
		\end{tabular}
	\end{center}	
\end{figure}

\begin{figure}[h]
	\caption{Comparison of the proposed NoCK method (i.e., Algorithm~\ref{alg:gock-cvx}) to SCVRG in \cite{lin2018improved} on solving instances of \eqref{eq:exp1} with the parameter tuple $(n, L, v) = (5e4, 791, 0)$. The instances are treated as convex, even though they can be strongly convex. Parameters of the algorithms are set by the theoretical parameter choice in the top row and the heuristic choice in the bottom row.}
	\label{fig:nock_comp_k_large}
	\begin{center}
		\begin{tabular}{cccc}
			{\small trial $1$} & {\small trial $2$} & {\small trial $3$} & {\small trial $4$}\\	
			\includegraphics[width=.23\linewidth]{pics/Case_N=500_by_n=50000_with_v=0_T1_O2.pdf}		&
			\includegraphics[width=.23\linewidth]{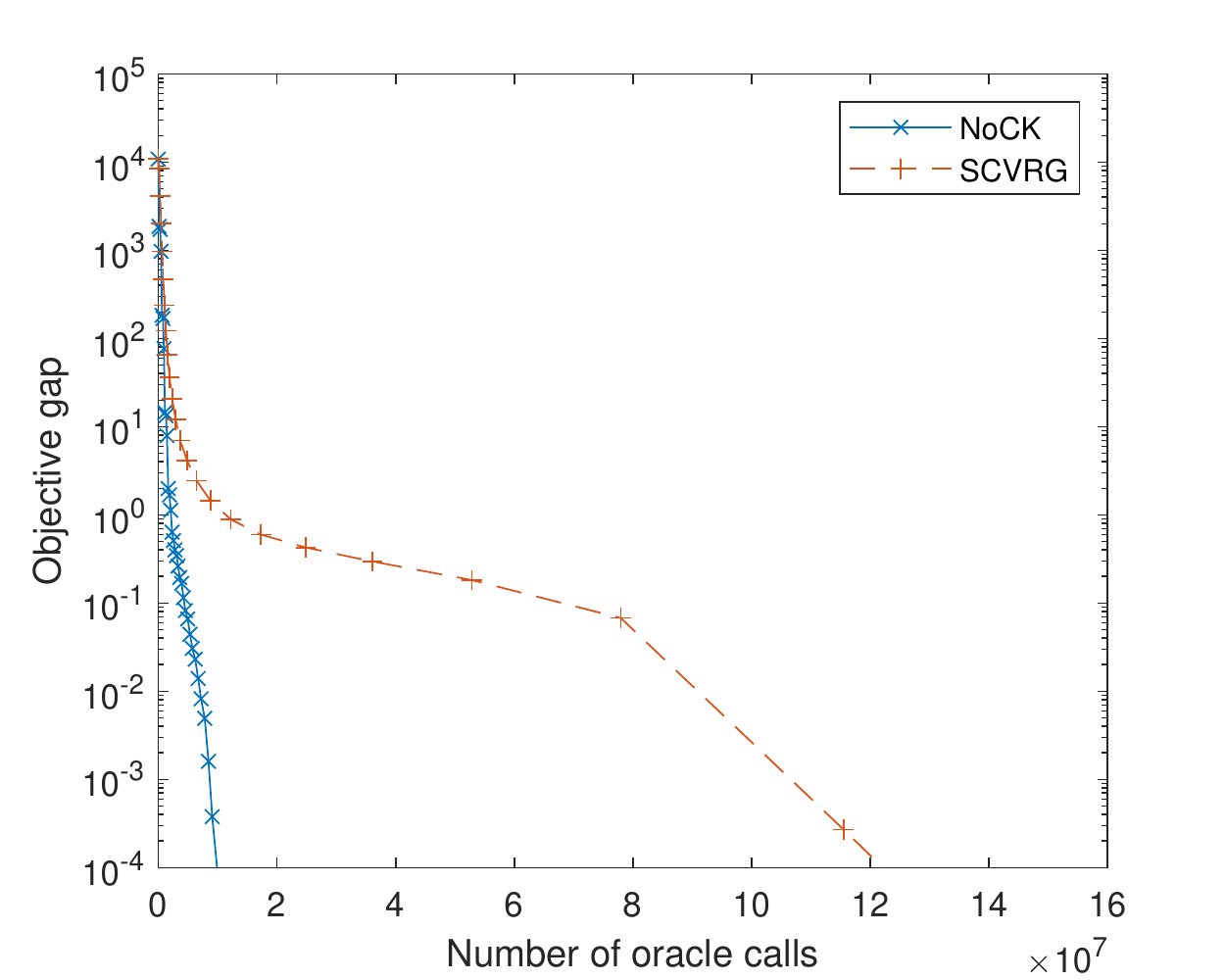}	&
			\includegraphics[width=.23\linewidth]{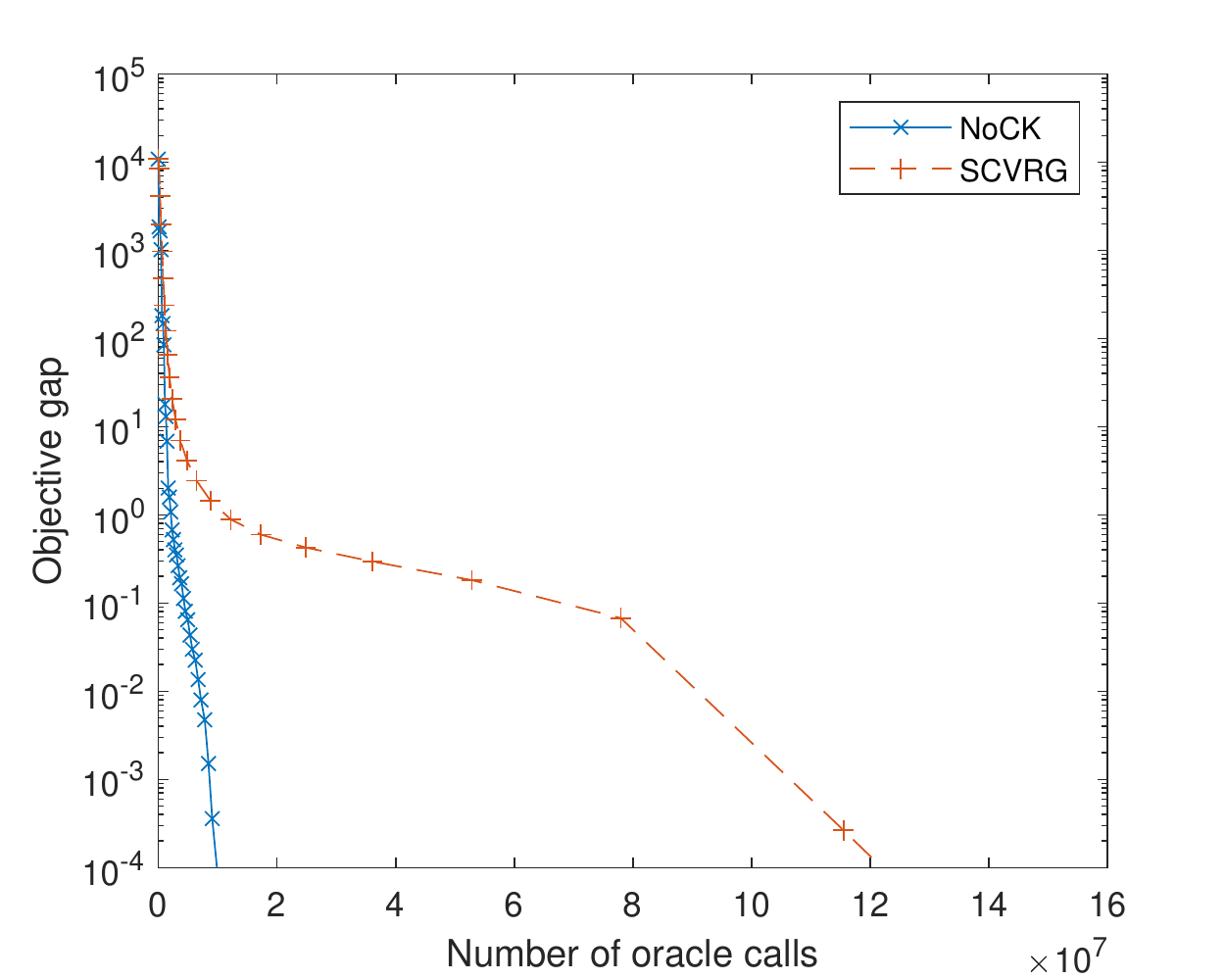}	&
			\includegraphics[width=.23\linewidth]{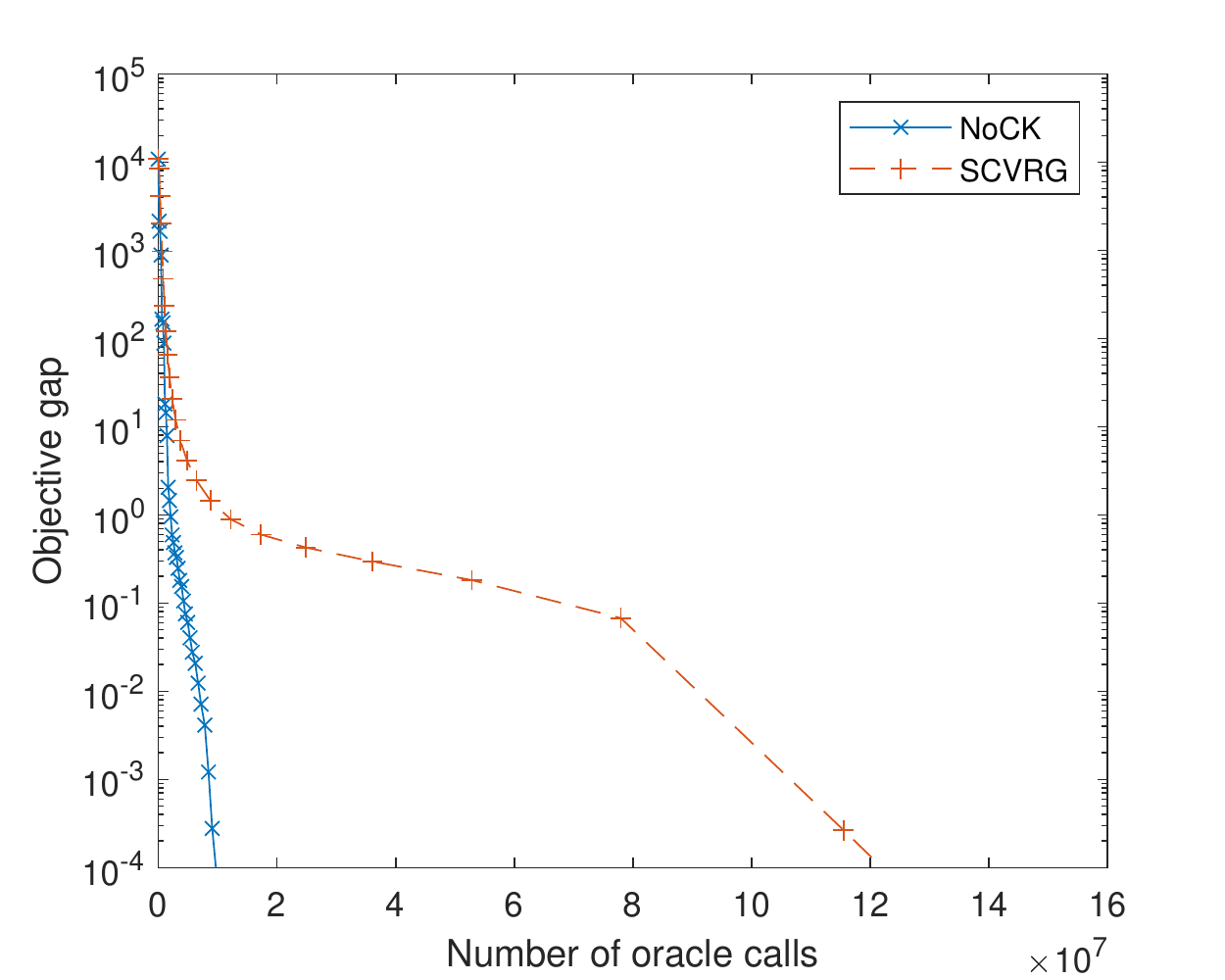}	\\
			\includegraphics[width=.23\linewidth]{pics/Case_N=500_by_n=50000_with_v=0_T1_O1.pdf}		&
			\includegraphics[width=.23\linewidth]{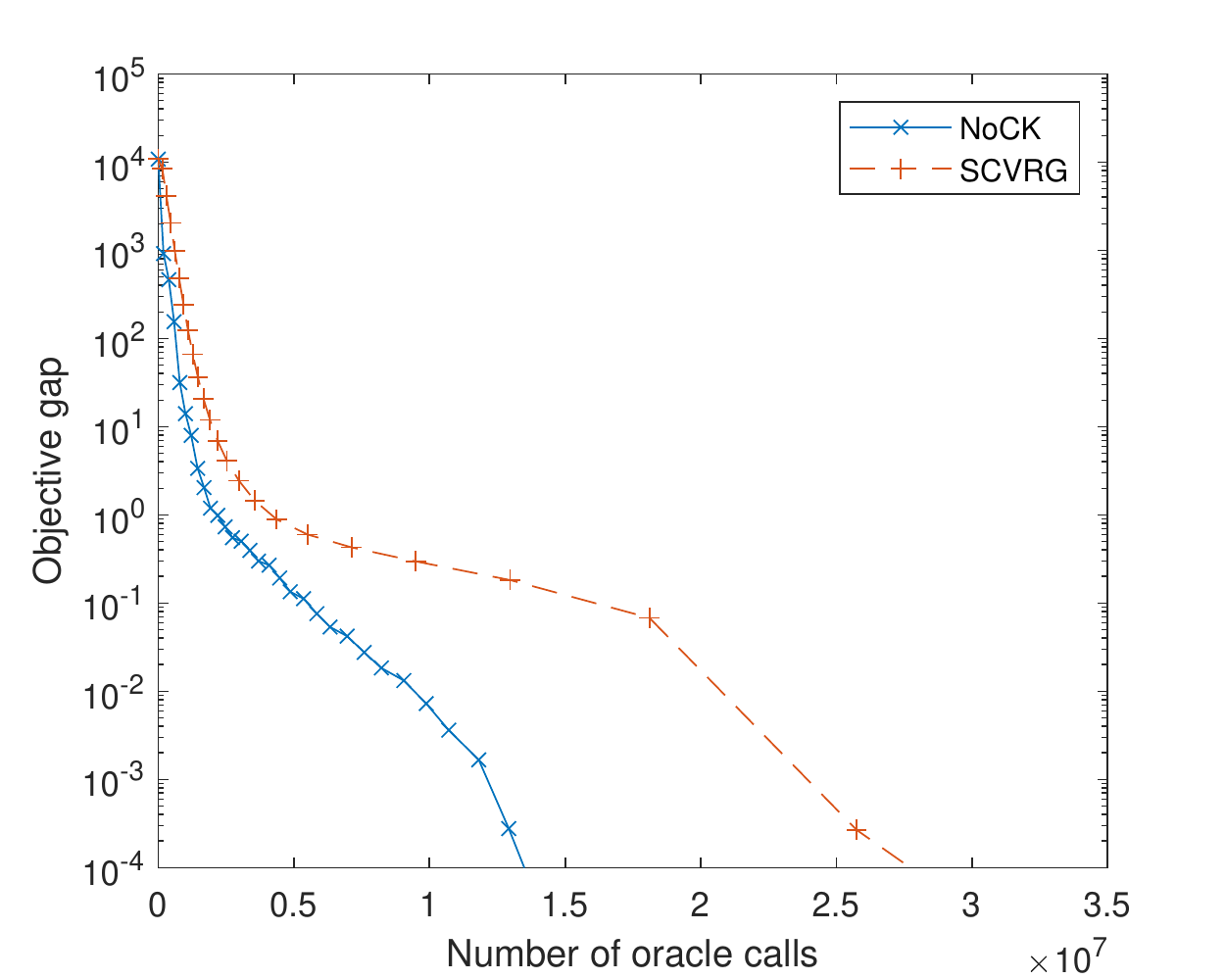}	&
			\includegraphics[width=.23\linewidth]{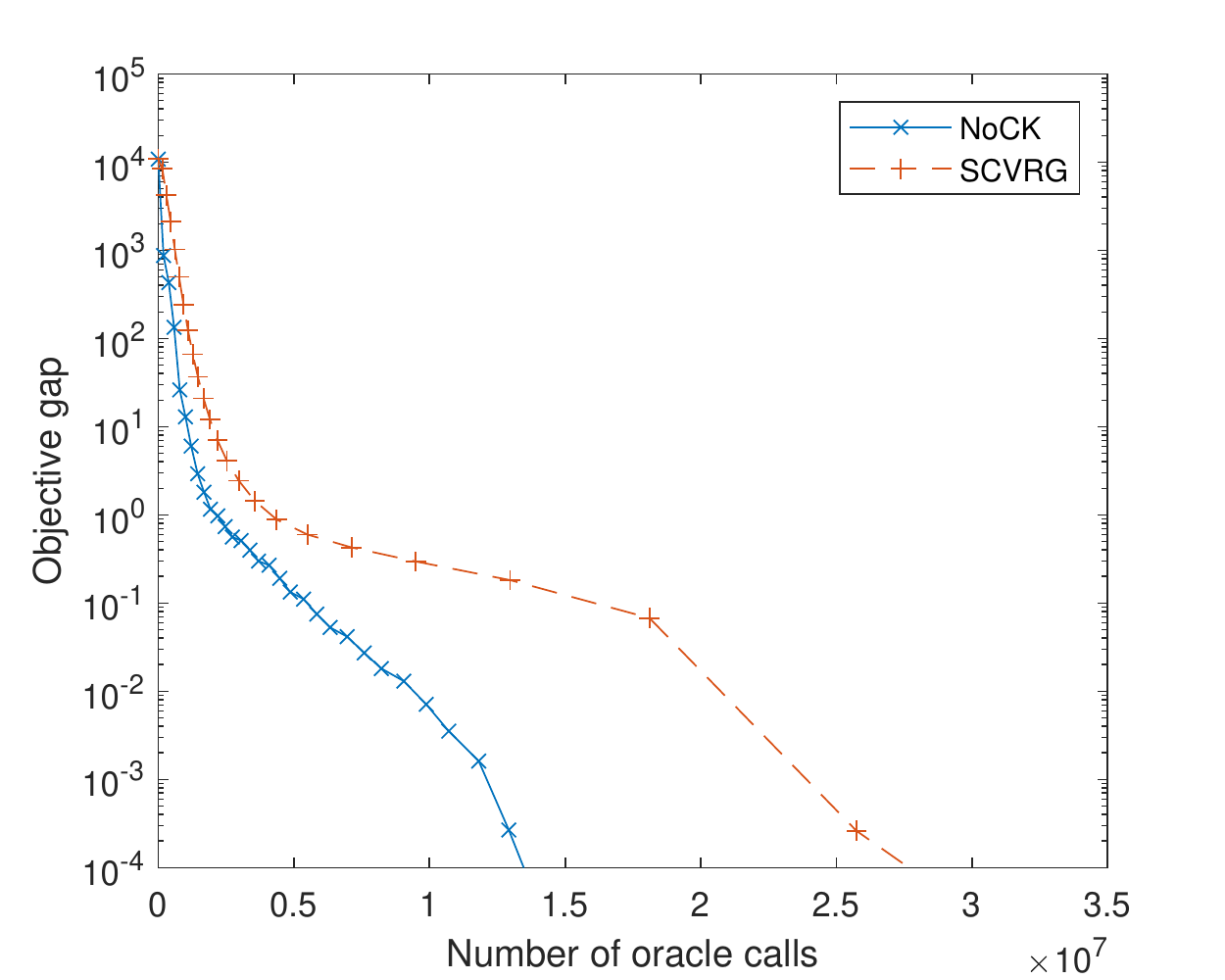}	&
			\includegraphics[width=.23\linewidth]{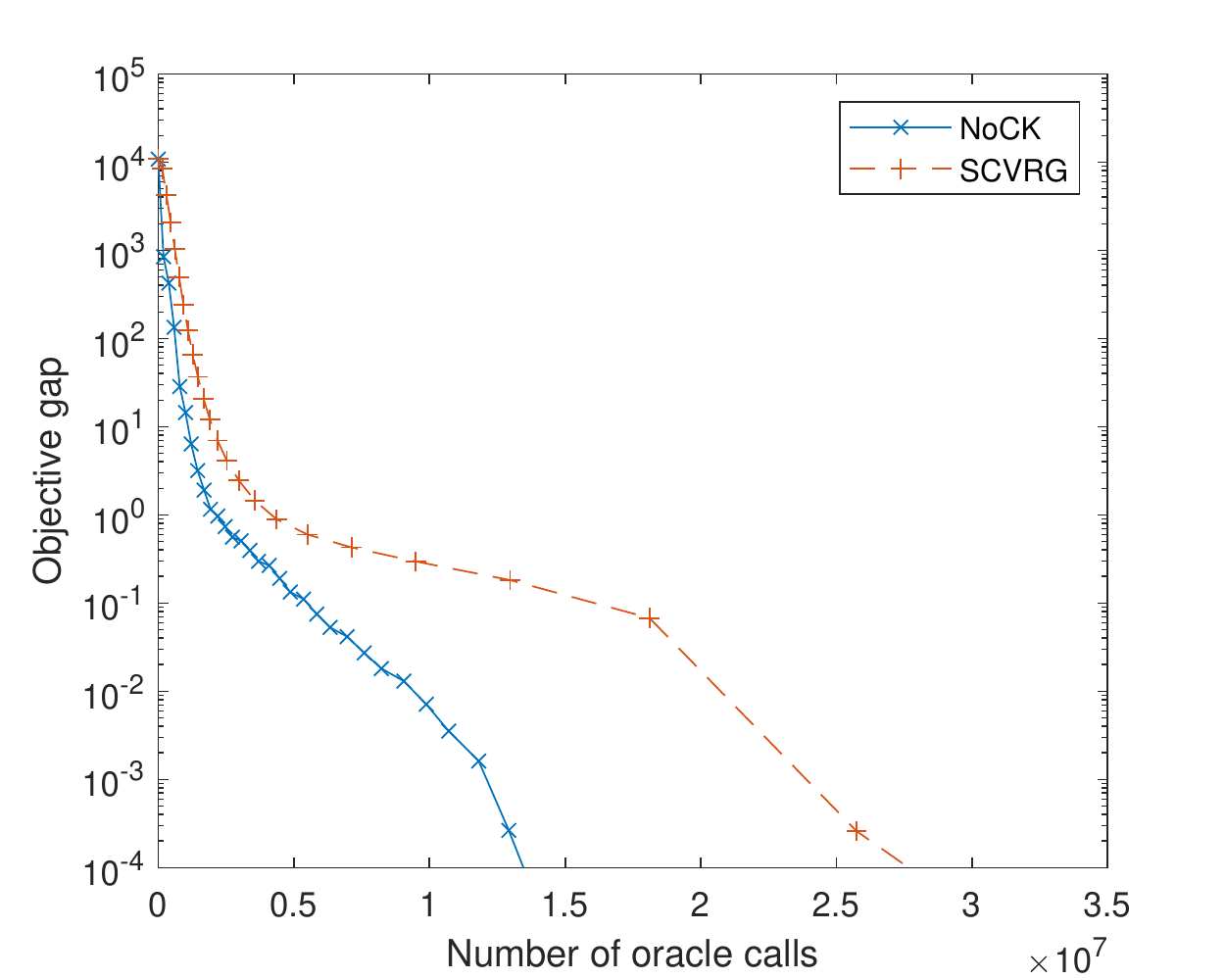}		
		\end{tabular}
	\end{center}	
\end{figure}

\begin{figure}[h]
	\caption{Comparison of the proposed NoCK method (i.e., Algorithm~\ref{alg:gock-cvx}) to SCVRG in \cite{lin2018improved} on solving instances of \eqref{eq:exp1} with the parameter tuple $(n, L, v) = (2.5e5, 806, 0)$. The instances are treated as convex, even though they can be strongly convex. Parameters of the algorithms are set by the theoretical parameter choice in the top row and the heuristic choice in the bottom row.}
	\label{fig:nock_comp_k_small}
	\begin{center}
		\begin{tabular}{cccc}
			{\small trial $1$} & {\small trial $2$} & {\small trial $3$} & {\small trial $4$}\\	
			\includegraphics[width=.23\linewidth]{pics/Case_N=500_by_n=250000_with_v=0_T1_O2.pdf}		&
			\includegraphics[width=.23\linewidth]{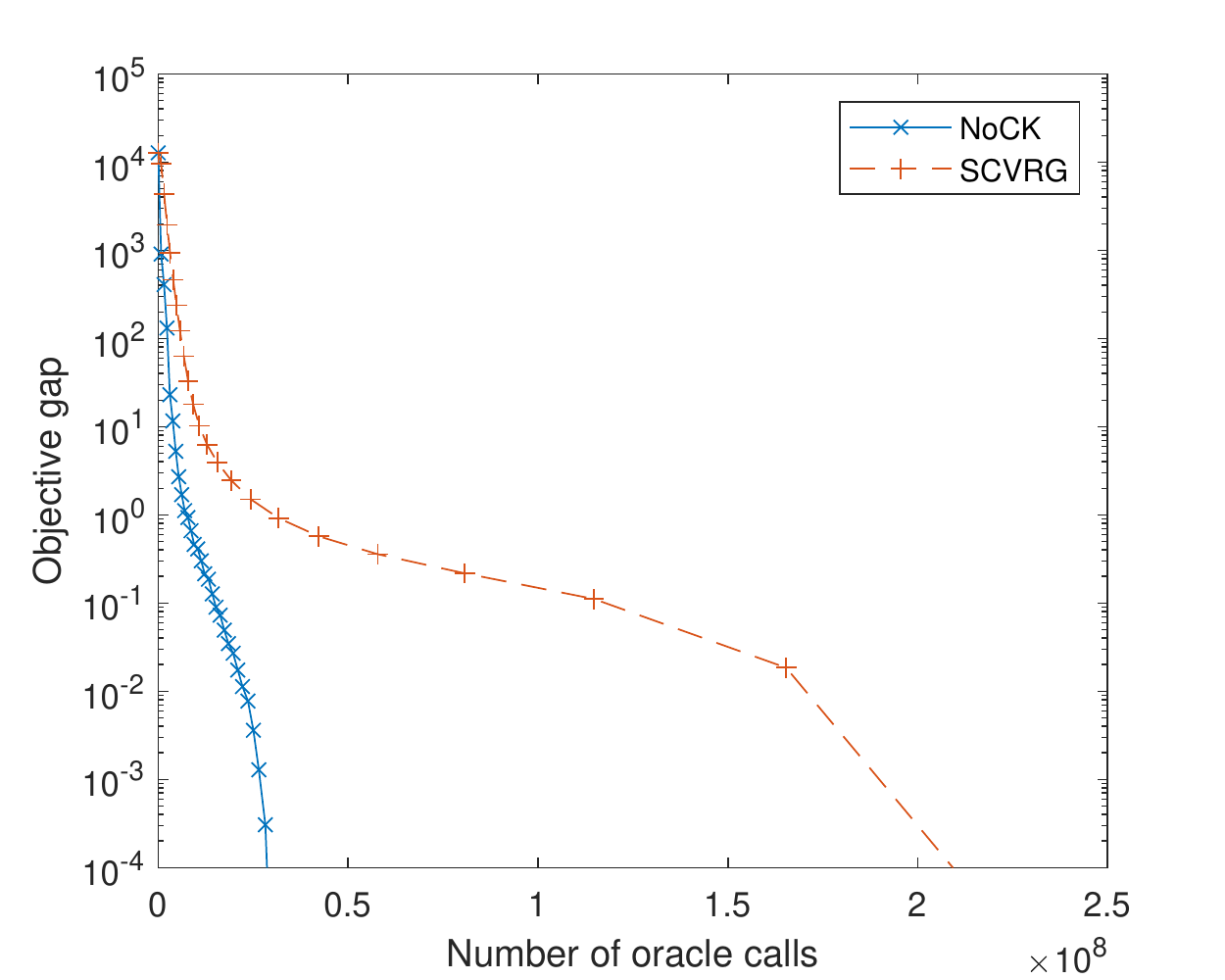}	&
			\includegraphics[width=.23\linewidth]{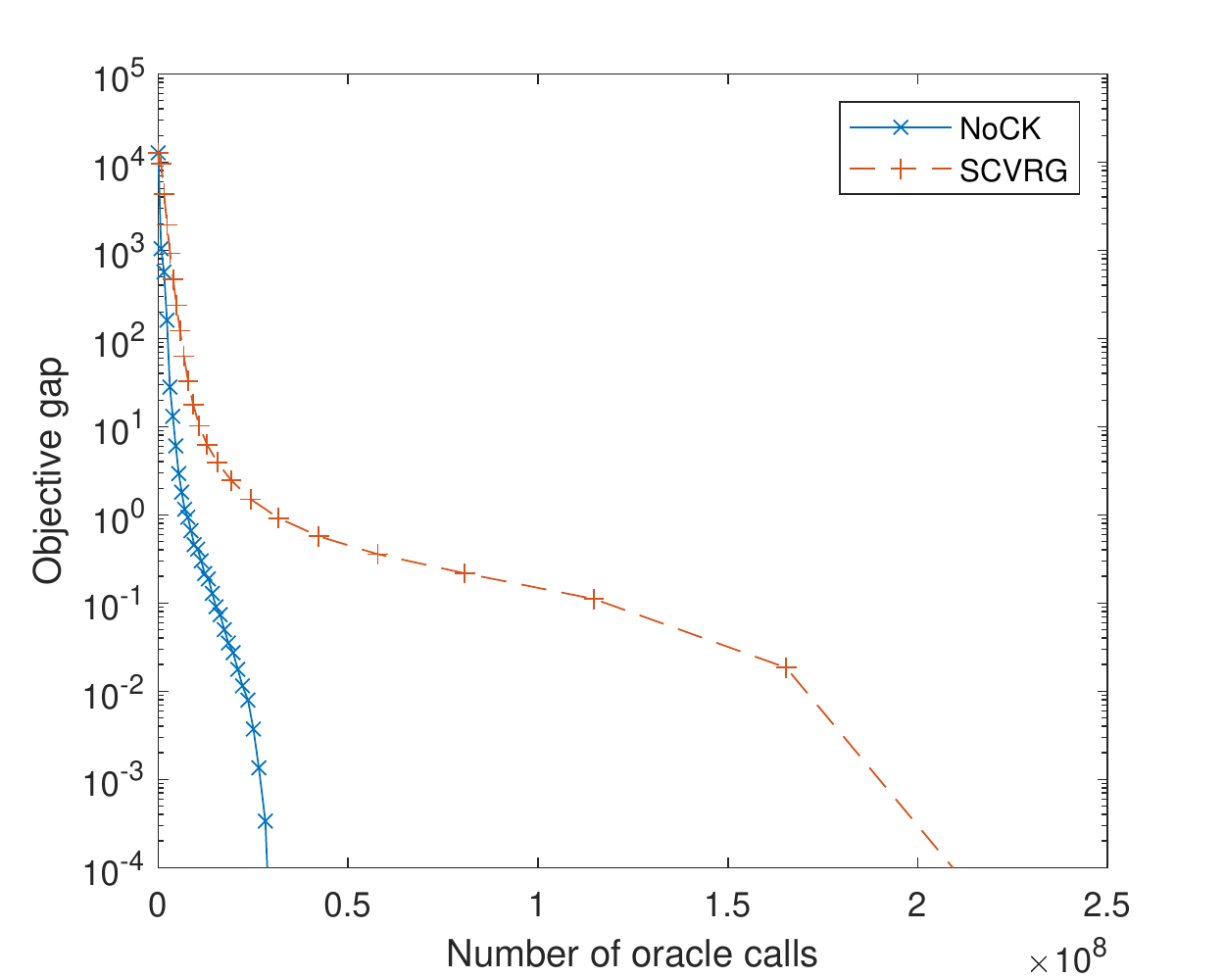}	&
			\includegraphics[width=.23\linewidth]{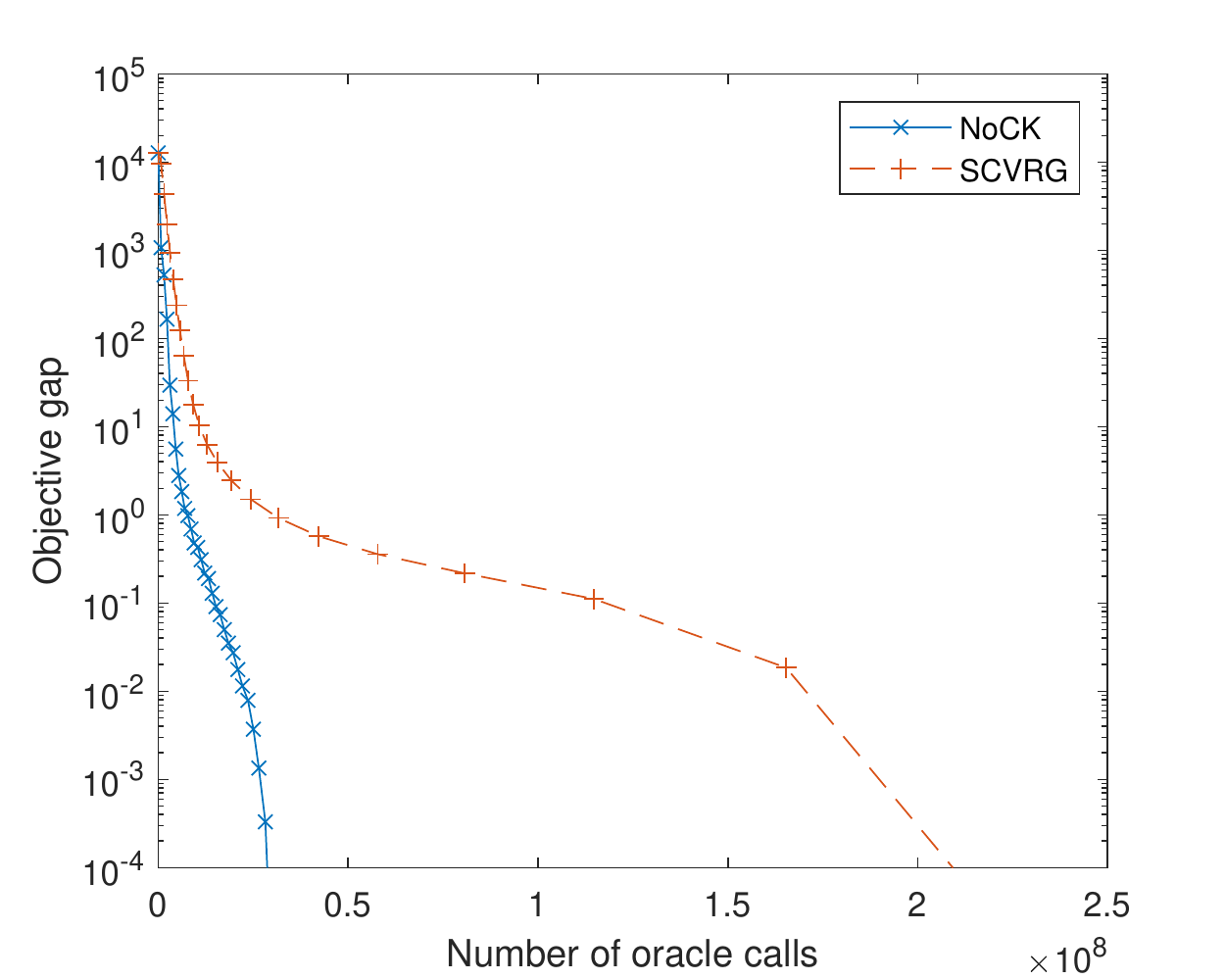}	\\
			\includegraphics[width=.23\linewidth]{pics/Case_N=500_by_n=250000_with_v=0_T1_O1.pdf}		&
			\includegraphics[width=.23\linewidth]{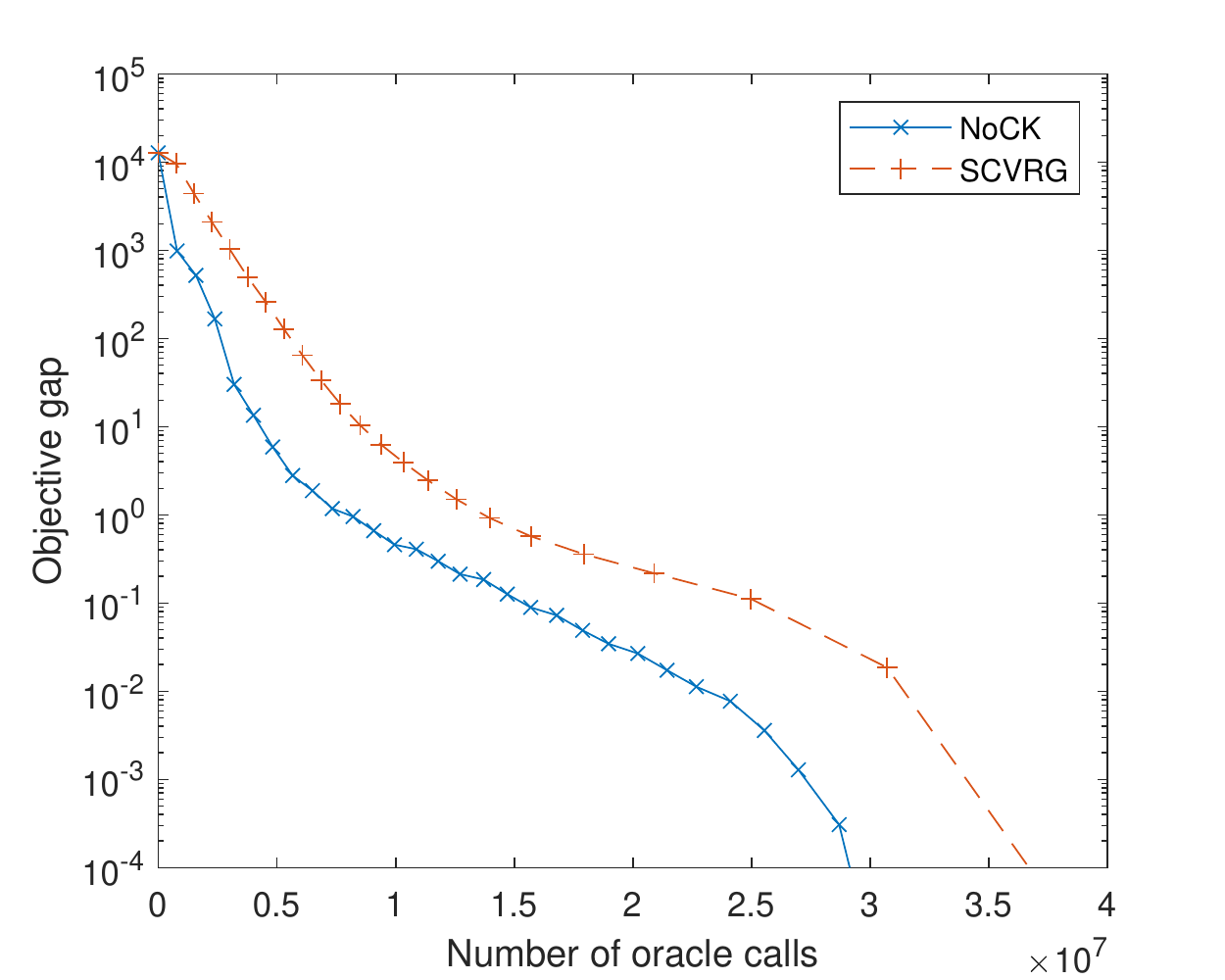}	&
			\includegraphics[width=.23\linewidth]{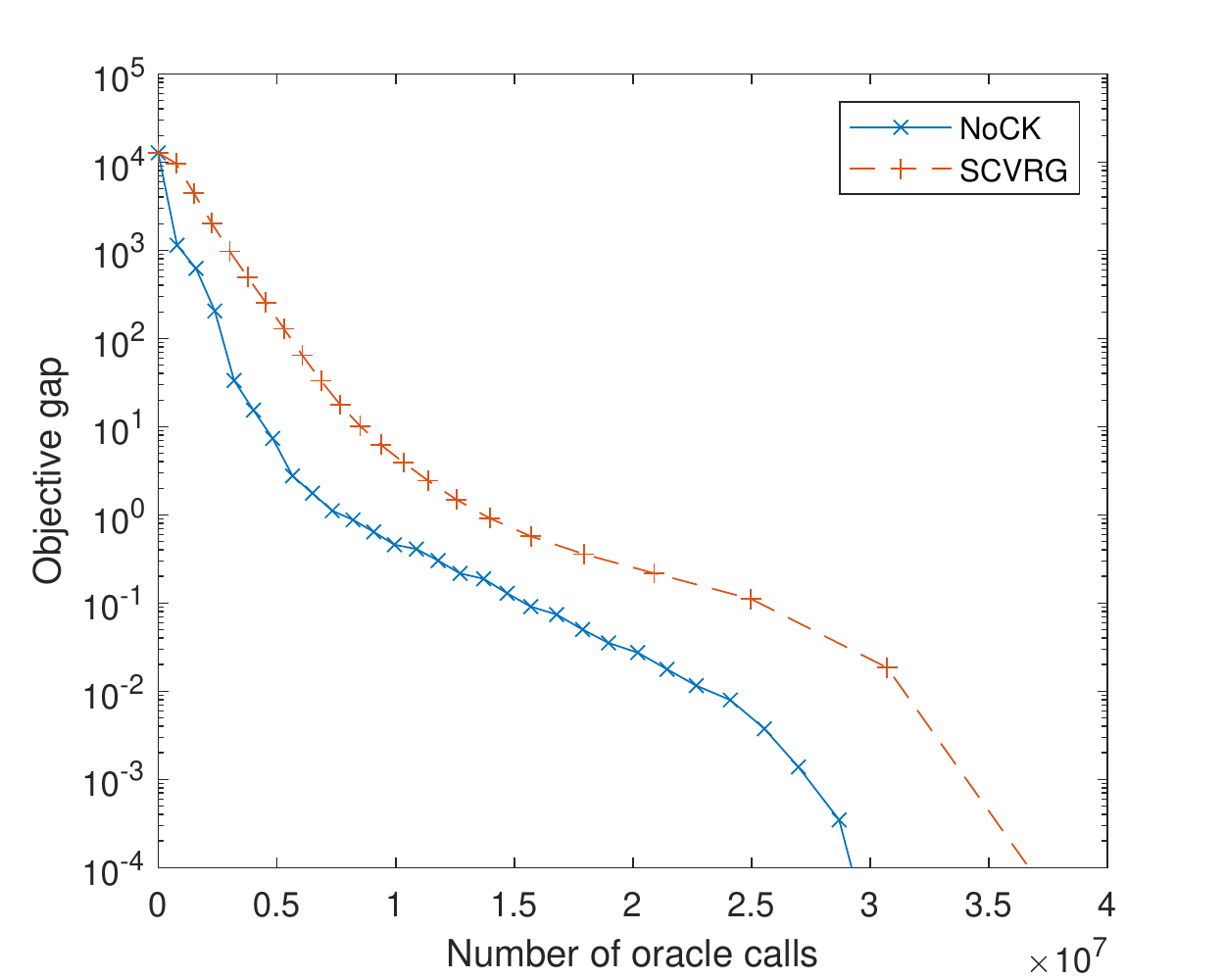}	&
			\includegraphics[width=.23\linewidth]{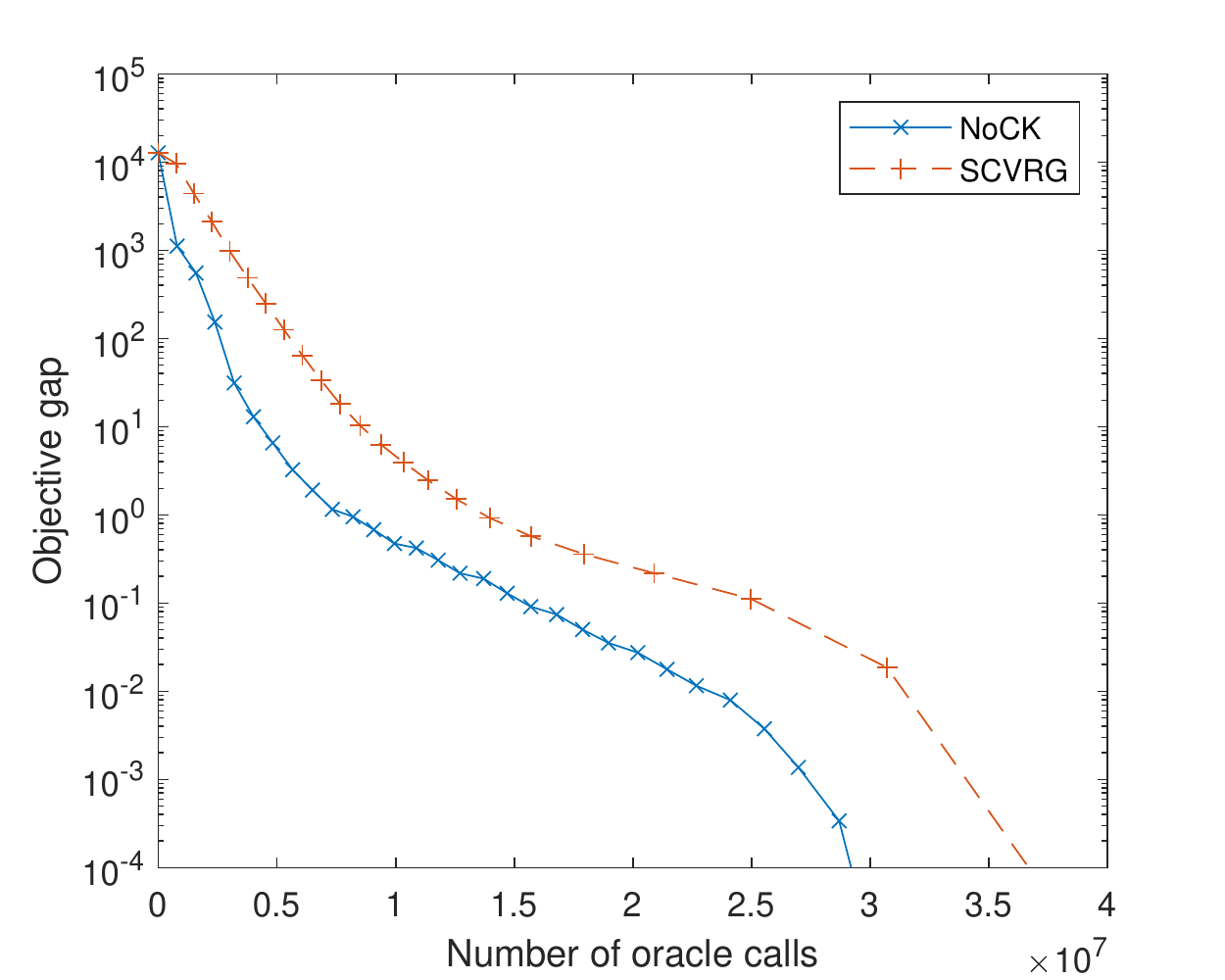}		
		\end{tabular}
	\end{center}	
\end{figure}

\begin{figure}[h]
	\caption{Comparison of the proposed NoCK method (i.e., Algorithm~\ref{alg:gock-cvx}) to SCVRG in \cite{lin2018improved} on solving instances of \eqref{eq:exp1} with the parameter tuple $(n, L, v) = (5e4, 798, 10)$. The instances are treated as convex, even though they can be strongly convex. Parameters of the algorithms are set by the theoretical parameter choice in the top row and the heuristic choice in the bottom row.}
	\label{fig:nock_comp_k_large2}
	\begin{center}
		\begin{tabular}{cccc}
			{\small trial $1$} & {\small trial $2$} & {\small trial $3$} & {\small trial $4$}\\	
			\includegraphics[width=.23\linewidth]{pics/Case_N=500_by_n=50000_with_v=10_T1_O2.pdf}		&
			\includegraphics[width=.23\linewidth]{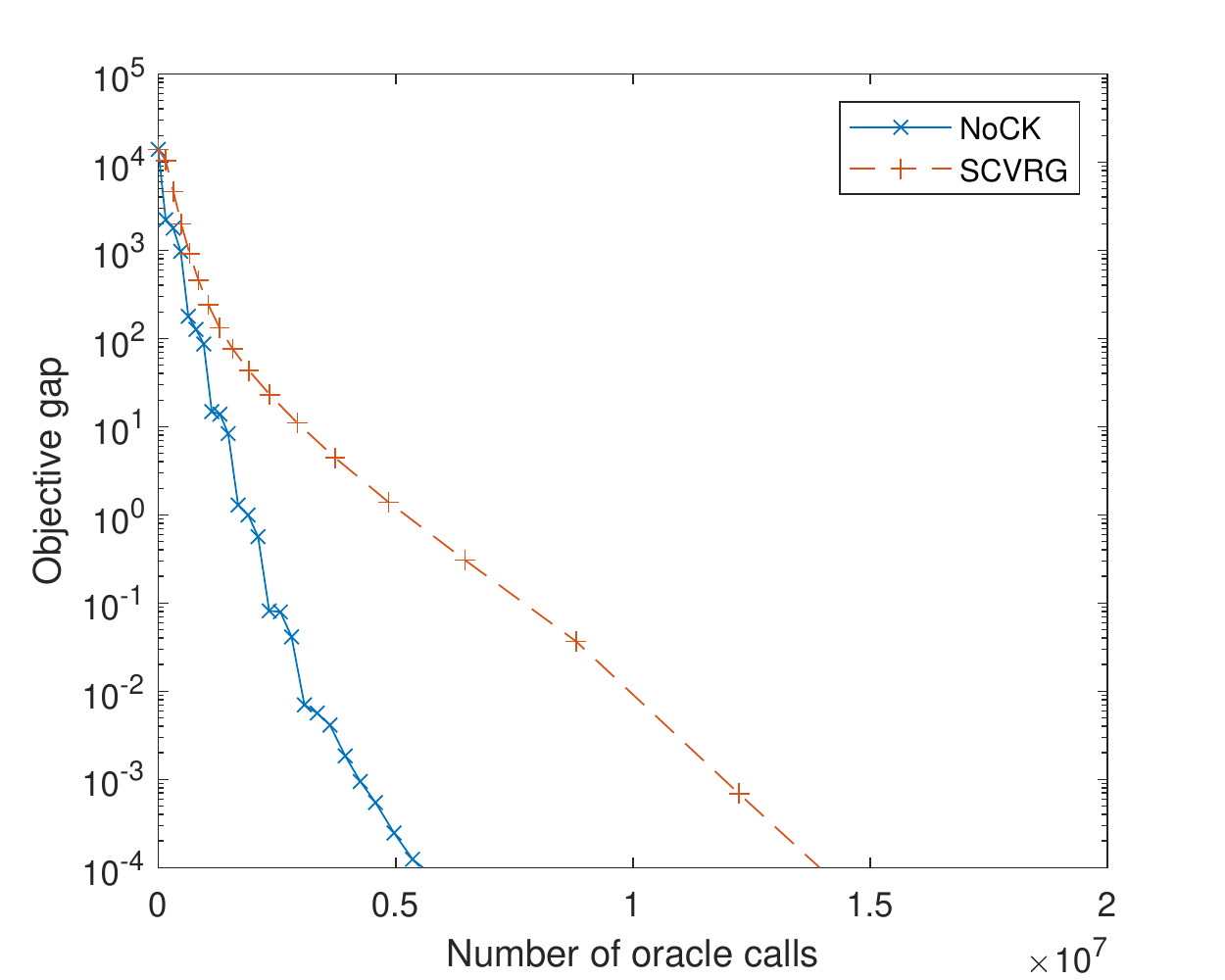}	&
			\includegraphics[width=.23\linewidth]{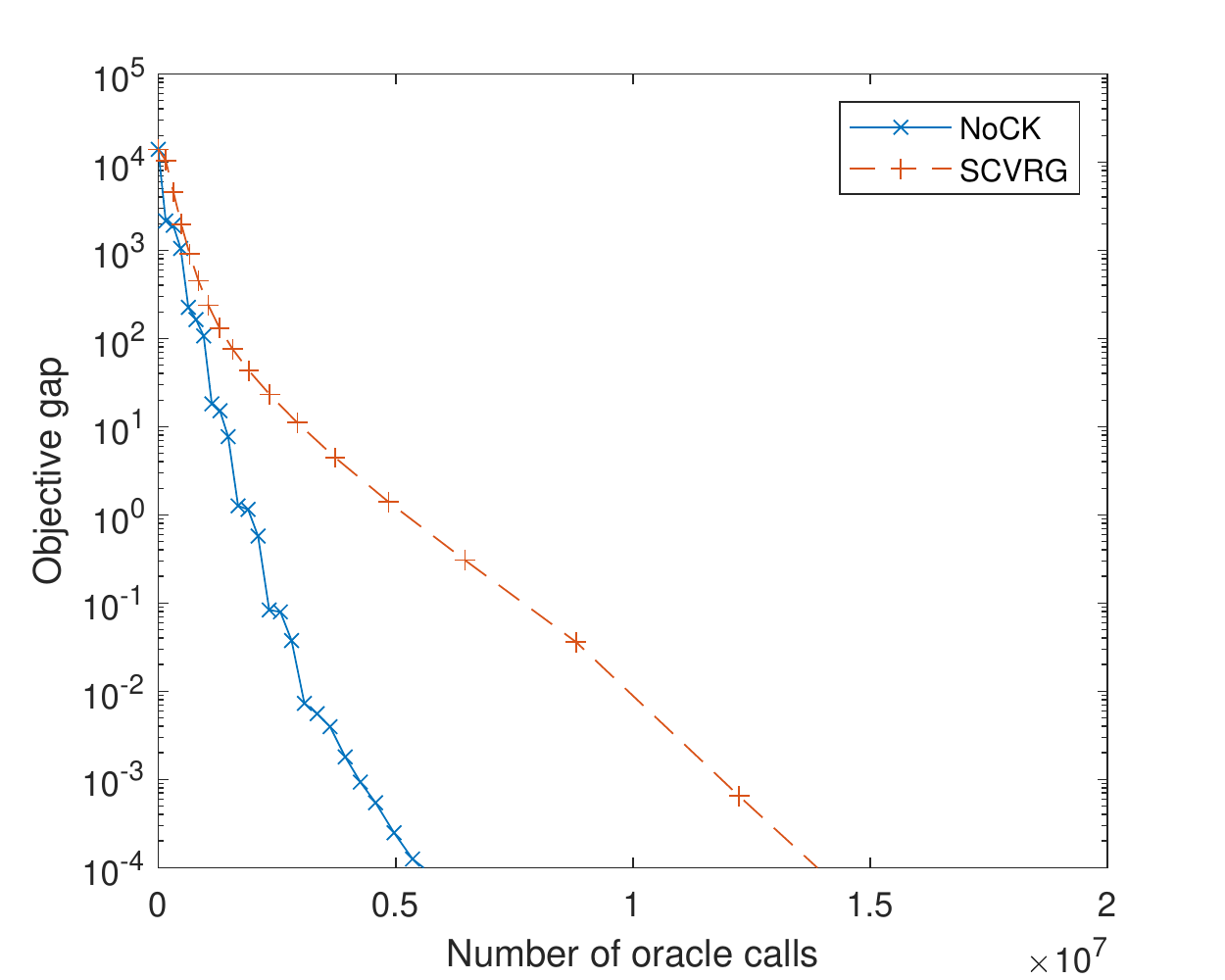}	&
			\includegraphics[width=.23\linewidth]{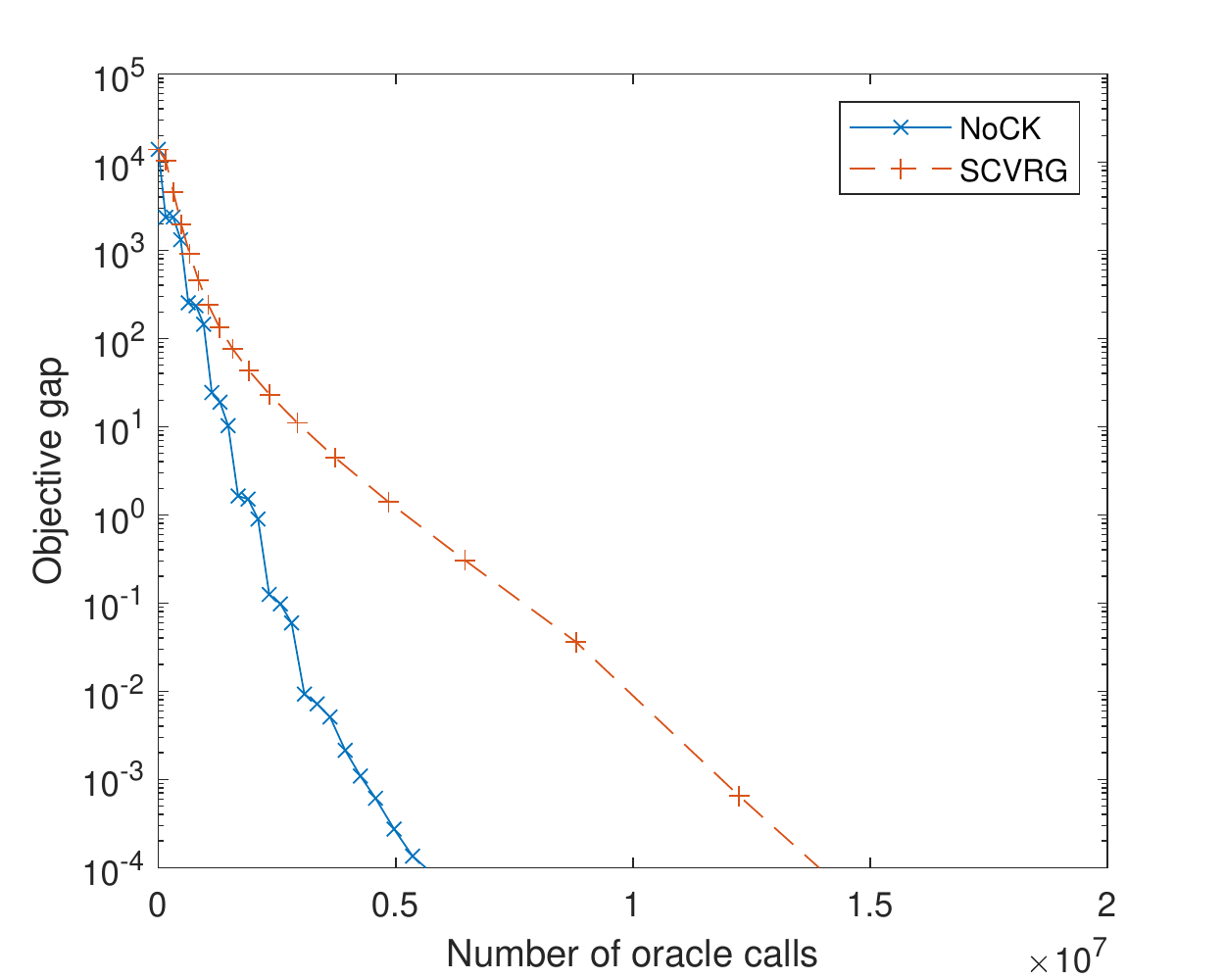}	\\
			\includegraphics[width=.23\linewidth]{pics/Case_N=500_by_n=50000_with_v=10_T1_O1.pdf}		&
			\includegraphics[width=.23\linewidth]{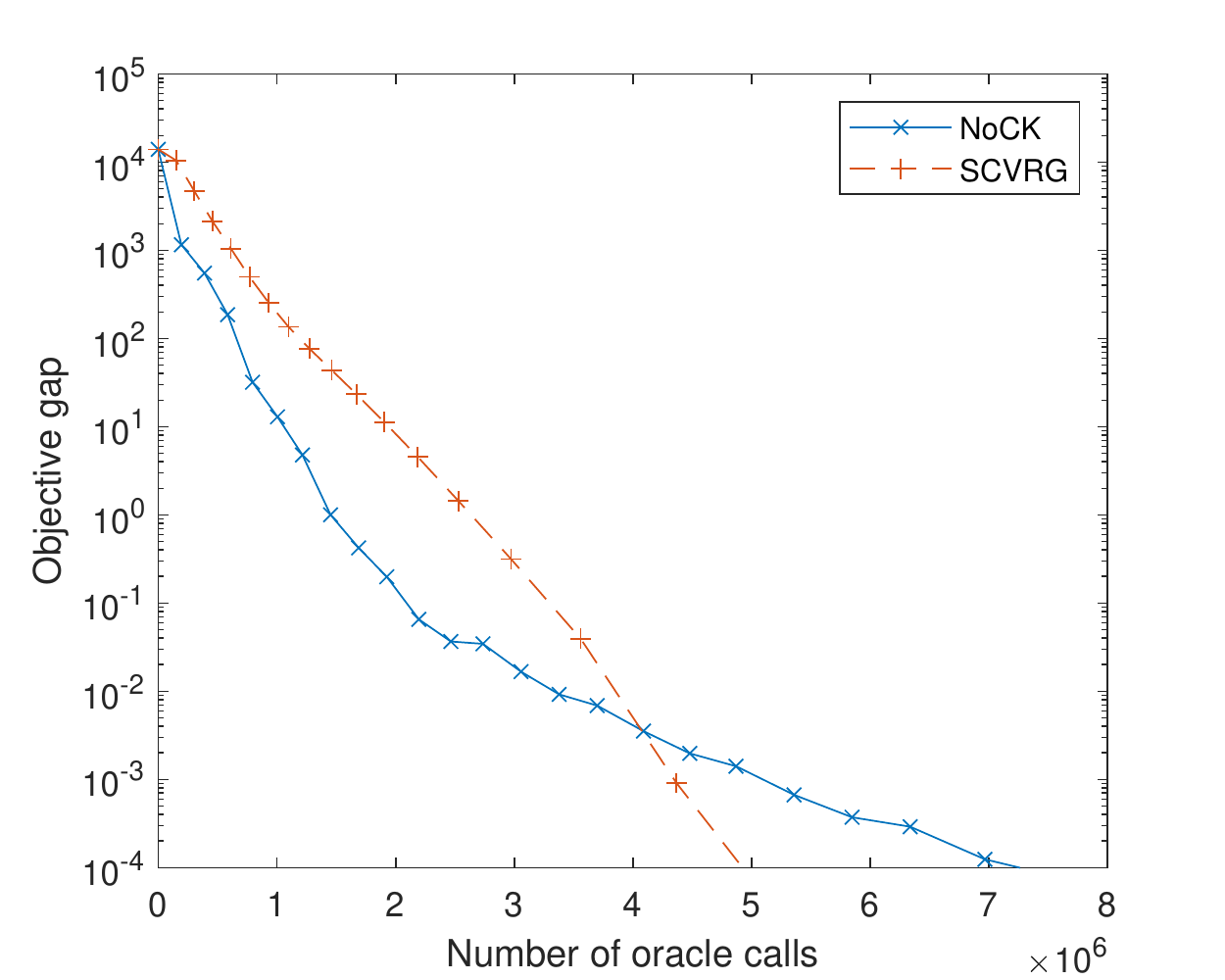}	&
			\includegraphics[width=.23\linewidth]{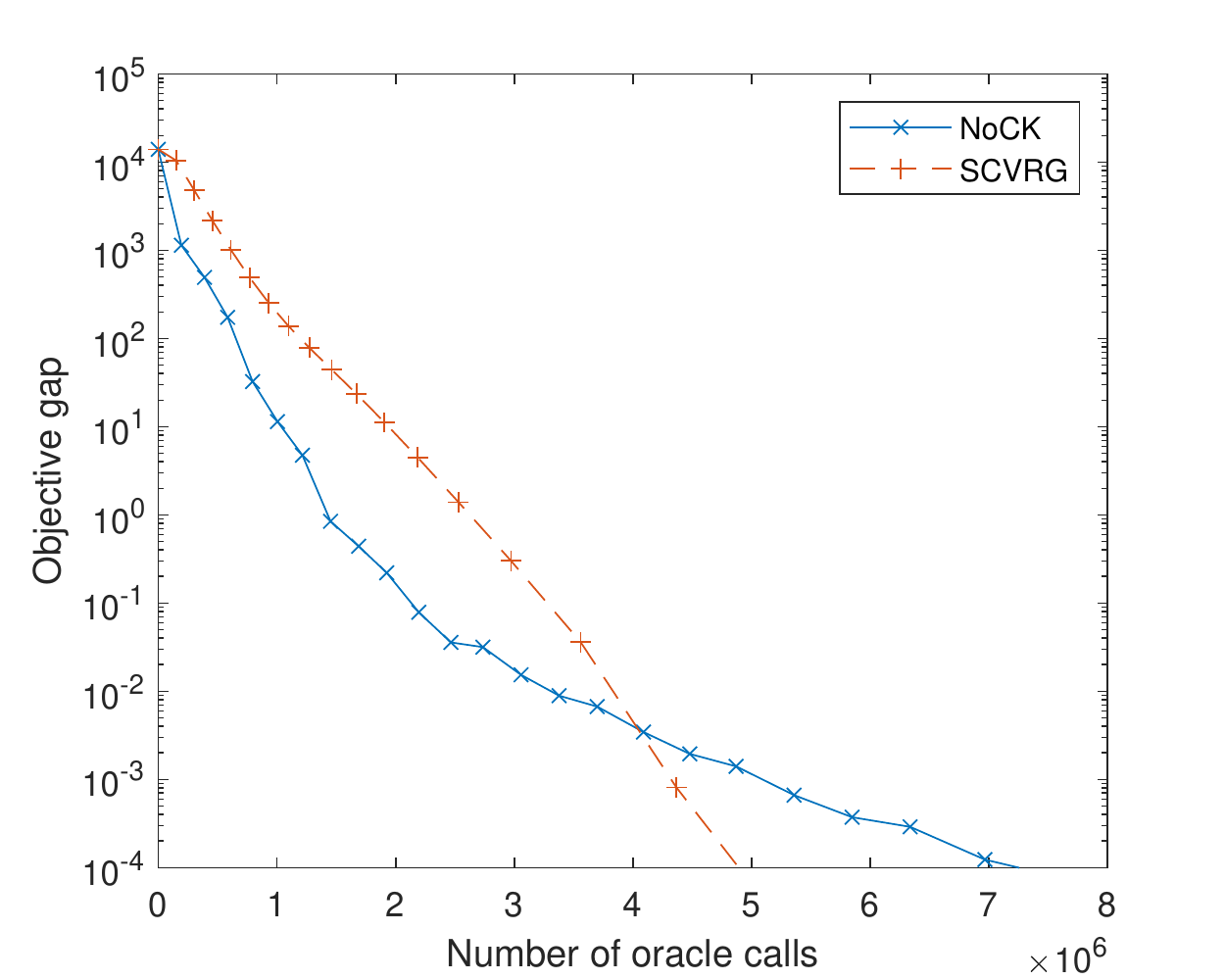}	&
			\includegraphics[width=.23\linewidth]{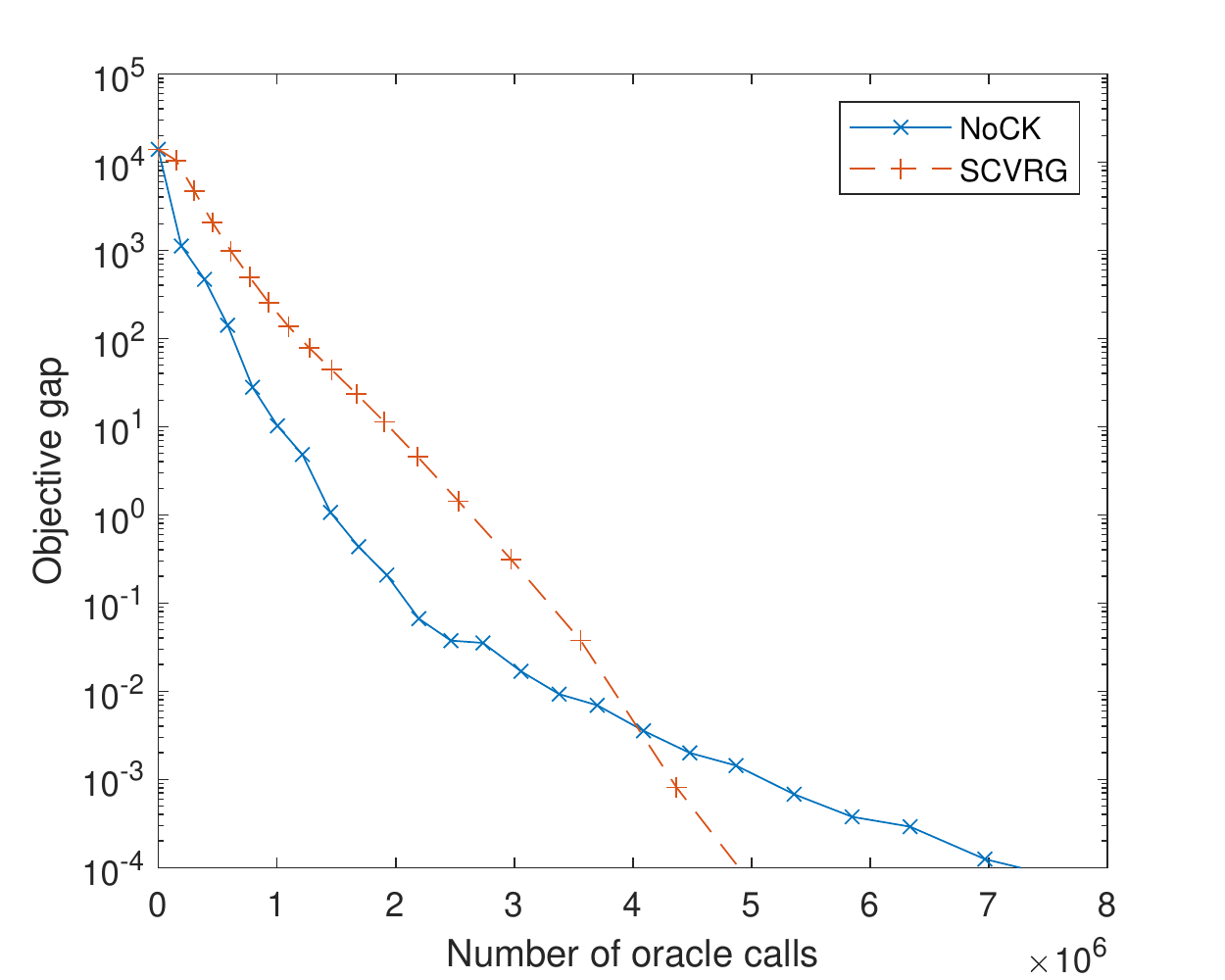}		
		\end{tabular}
	\end{center}	
\end{figure}

\begin{figure}[h]
	\caption{Comparison of the proposed NoCK method (i.e., Algorithm~\ref{alg:gock-cvx}) to SCVRG in \cite{lin2018improved} on solving instances of \eqref{eq:exp1} with the parameter tuple $(n, L, v) = (2.5e5, 799, 10)$. The instances are treated as convex, even though they can be strongly convex. Parameters of the algorithms are set by the theoretical parameter choice in the top row and the heuristic choice in the bottom row.}
	\label{fig:nock_comp_k_small2}
	\begin{center}
		\begin{tabular}{cccc}
			{\small trial $1$} & {\small trial $2$} & {\small trial $3$} & {\small trial $4$}\\	
			\includegraphics[width=.23\linewidth]{pics/Case_N=500_by_n=250000_with_v=10_T1_O2.pdf}		&
			\includegraphics[width=.23\linewidth]{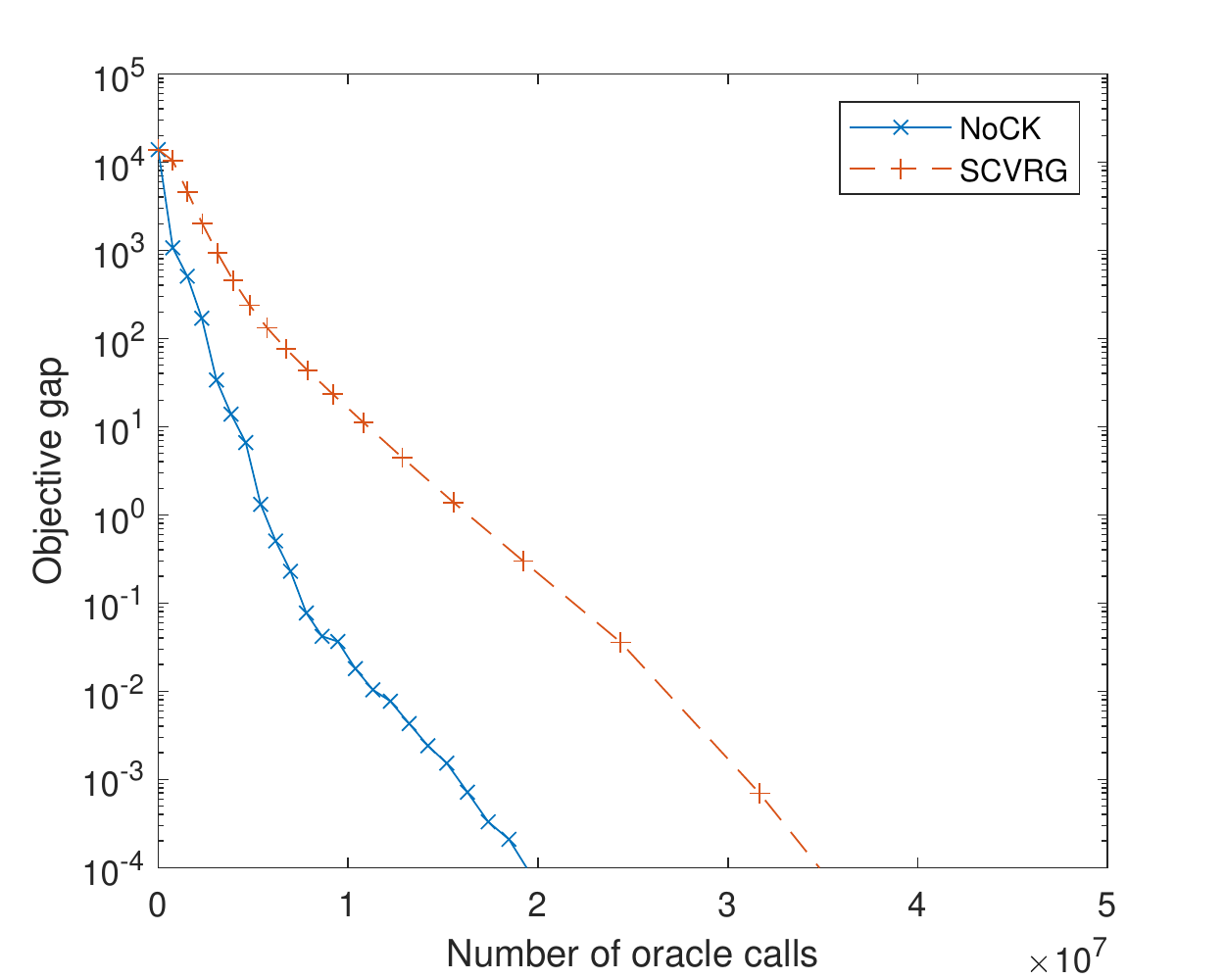}	&
			\includegraphics[width=.23\linewidth]{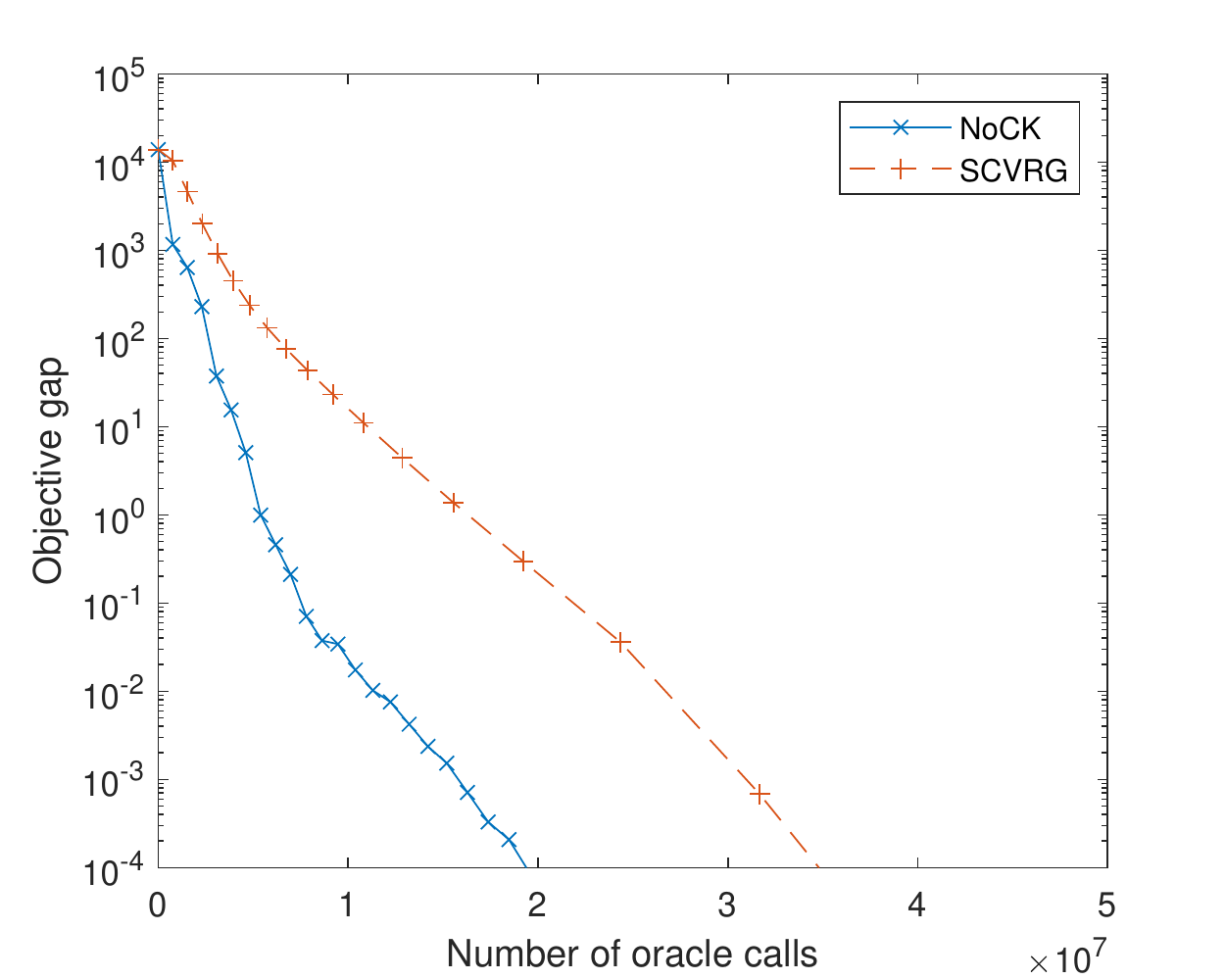}	&
			\includegraphics[width=.23\linewidth]{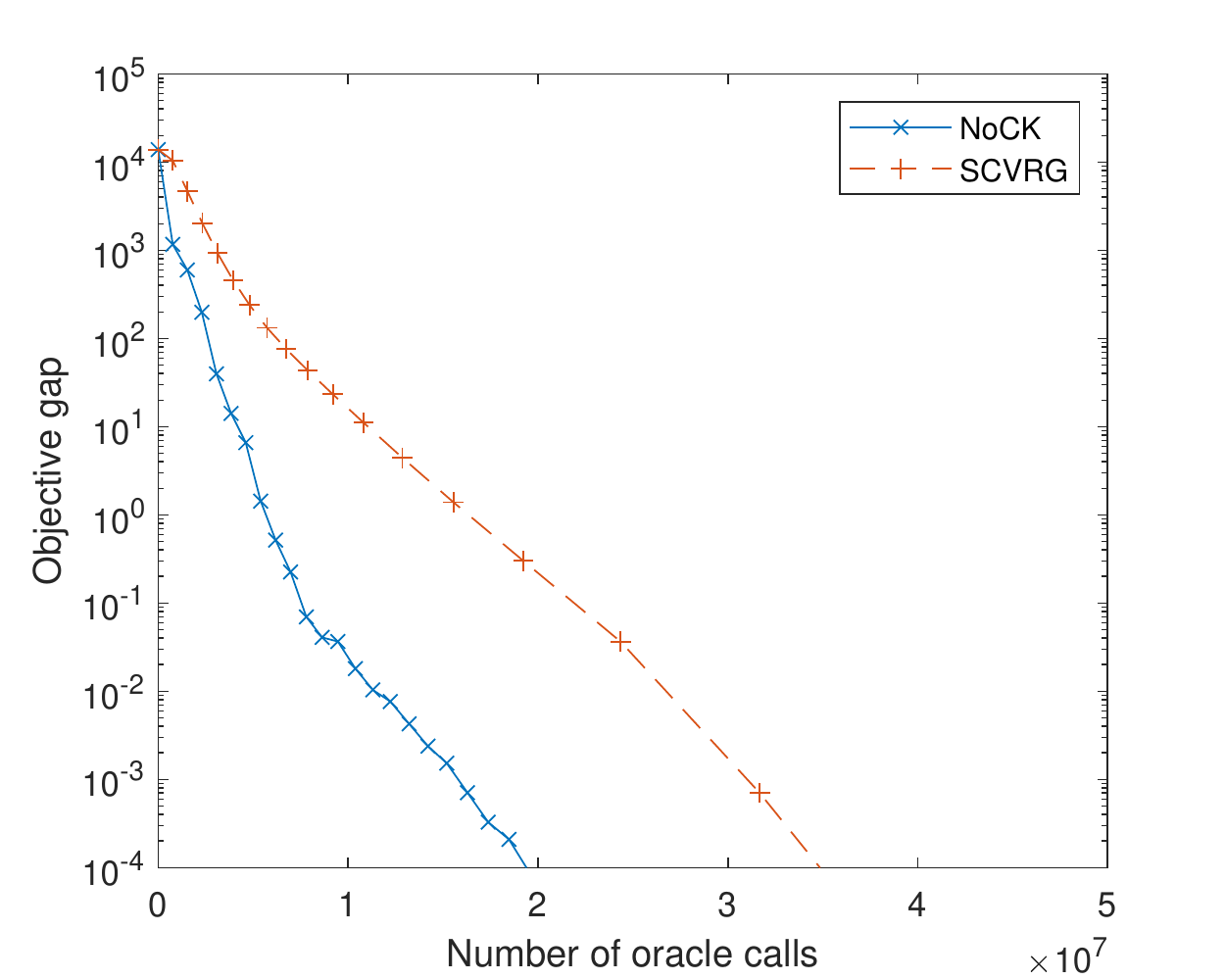}	\\
			\includegraphics[width=.23\linewidth]{pics/Case_N=500_by_n=250000_with_v=10_T1_O1.pdf}		&
			\includegraphics[width=.23\linewidth]{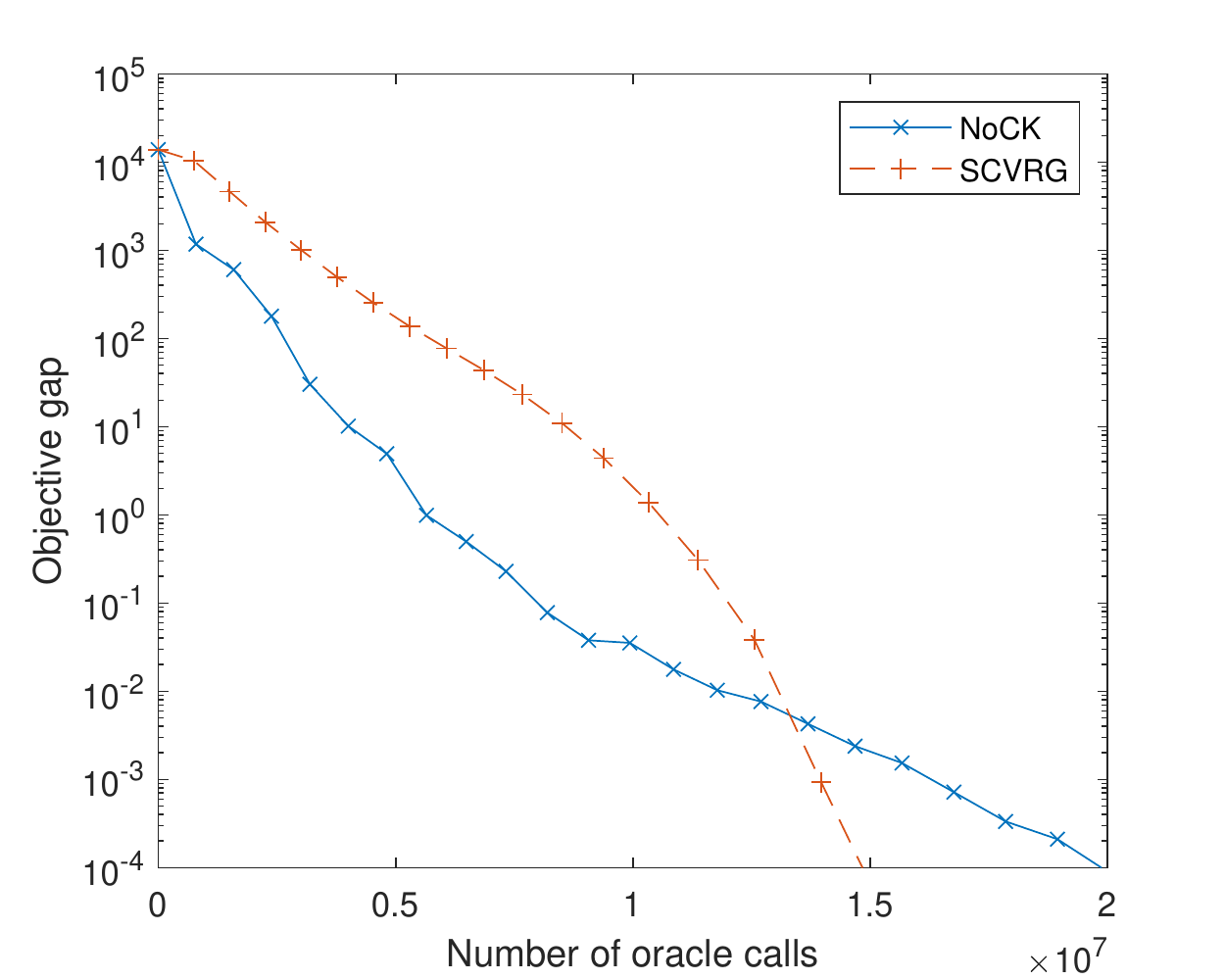}	&
			\includegraphics[width=.23\linewidth]{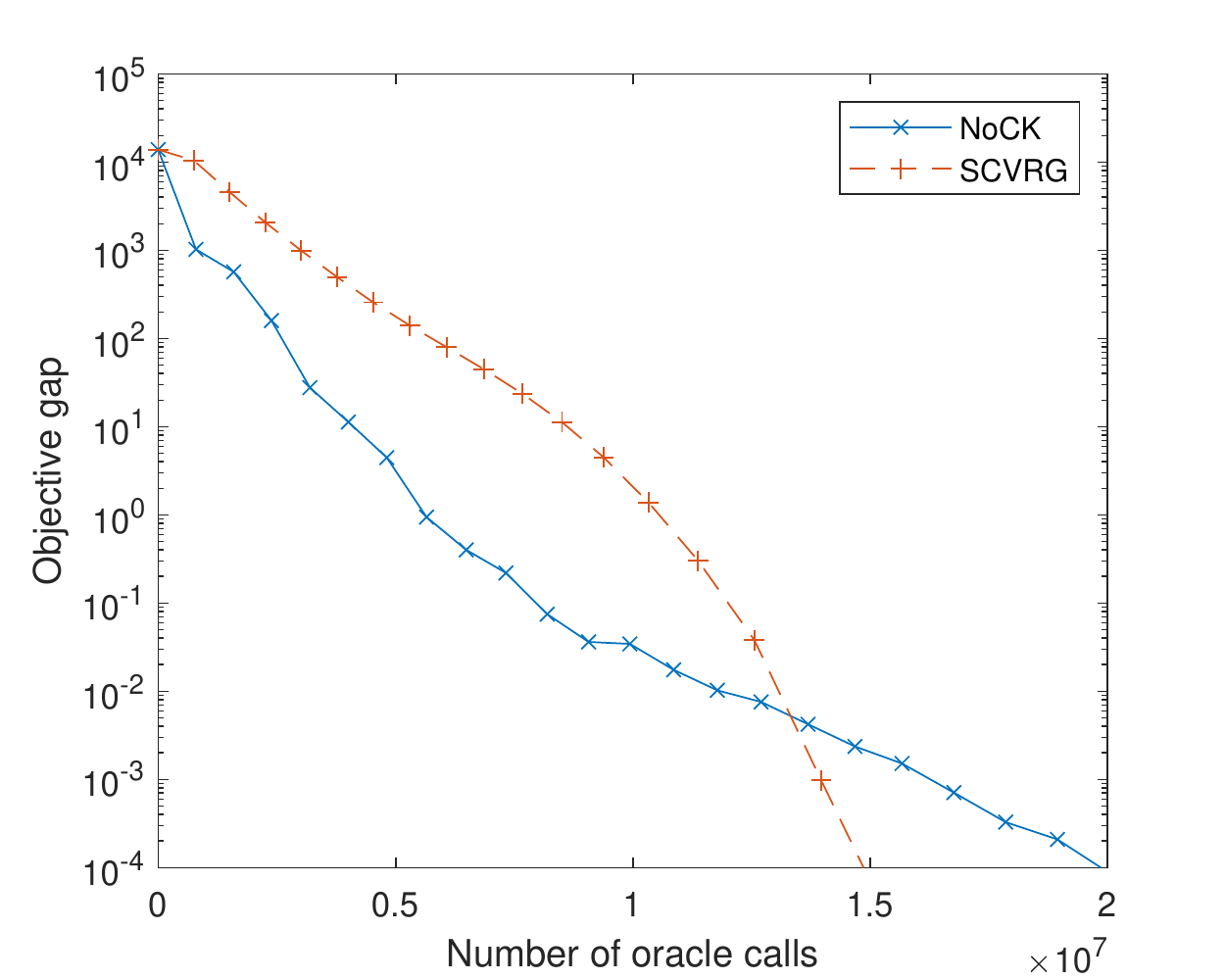}	&
			\includegraphics[width=.23\linewidth]{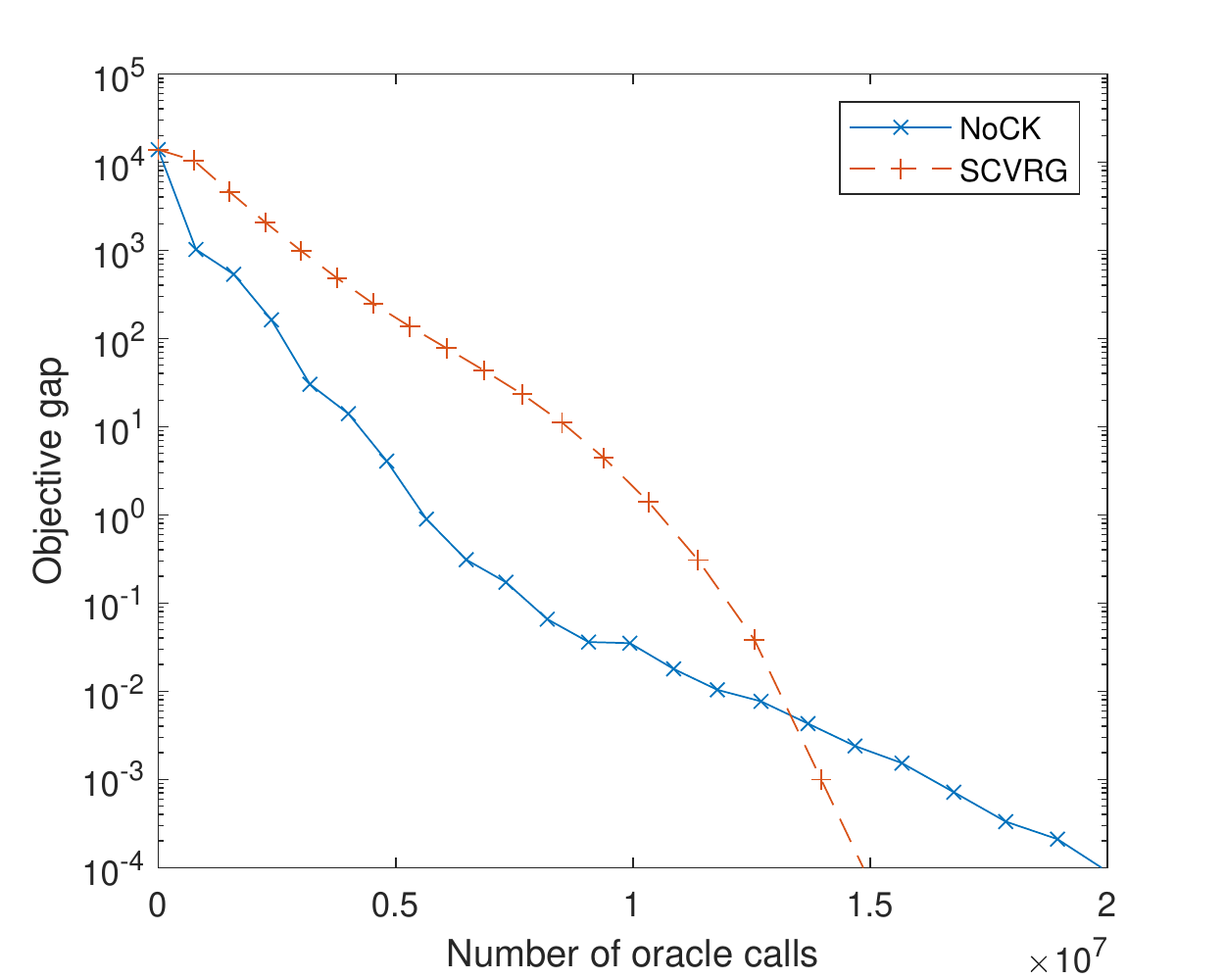}		
		\end{tabular}
	\end{center}	
\end{figure}	

} %for revise

\bibliographystyle{abbrv} % outcomment this and next line in Case 1
\bibliography{compositional}

\end{document}